\documentclass[11pt]{amsart}
\usepackage{amscd, amssymb, mathrsfs, tikz-cd, amsthm, thmtools, bbm, tikz, soul}
\usetikzlibrary{matrix,arrows,backgrounds,cd, decorations.markings}
\usepackage[all]{xy}
\usepackage{hyperref}     
\usepackage{longtable} 

\author[Favero]{David Favero}
\address{
	\begin{tabular}{l}
		David Favero \\
		\hspace{.1in} University of Alberta, Department of Mathematical and Statistical Sciences \\
		\hspace{.1in} Central Academic Building 632, Edmonton, AB, Canada T6G 2C7 \\
		\hspace{.1in} Korea Institute for Advanced Study \\
		\hspace{.1in} 85 Hoegiro, Dongdaemun-gu, Seoul, Republic of Korea 02455 \\
		\hspace{.1in} Email: {\bf favero@ualberta.ca} \\
	\end{tabular}
}

\author[Kim]{Bumsig Kim}
\address{
	\begin{tabular}{l}
		Bumsig Kim \\
		\hspace{.1in} Korea Institute for Advanced Study \\
		\hspace{.1in} 85 Hoegiro, Dongdaemun-gu, Seoul, Republic of Korea 02455 \\
		\hspace{.1in} Email: {\bf bumsig@kias.re.kr} \\
	\end{tabular}
}

\input{xy}
\xyoption{all}

\newtheorem{Thm}{Theorem}[section]

\newtheorem{Prop}[Thm]{Proposition}
\newtheorem{Prop/Def}[Thm]{Proposition/Definition}
\newtheorem{Def}[Thm]{Definition}
\newtheorem{disclaimer}[Thm]{Disclaimer}
\newtheorem{Def/Thm}[Thm]{Definition/Theorem}
\newtheorem{Cor}[Thm]{Corollary}
\newtheorem{Lemma}[Thm]{Lemma}

\theoremstyle{definition}
\newtheorem{Rmk}[Thm]{Remark}
\newtheorem{example}[Thm]{Example}

\numberwithin{equation}{section}
\newcommand{\ti }{\times}
\newcommand{\ot }{\otimes}
\newcommand{\ra }{\rightarrow}

\newcommand{\Hom }{{\mathrm{Hom}}}
\newcommand{\tr }{{\mathrm{tr}}}

\newcommand{\Spec}{{\mathrm{Spec}}\hspace{.5mm}}

\newcommand{\Ker}{{\mathrm{Ker}}}

\newcommand{\Sym}{{\mathrm{Sym}}}

\newcommand{\rank }{{\mathrm{rank}}}
\newcommand{\cA}{{\mathcal{A}}}
\newcommand{\cO}{{\mathcal{O}}}

\newcommand{\cE}{{\mathcal{E}}}
\newcommand{\cF}{{\mathcal{F}}}
\newcommand{\cH}{{\mathcal{H}}}
\newcommand{\cP}{{\mathcal{P}}}
\newcommand{\cQ}{{\mathcal{Q}}}
\newcommand{\cV}{{\mathcal{V}}}
\newcommand{\cC}{{\mathcal{C}}}
\newcommand{\cY}{{\mathcal{Y}}}
\newcommand{\cT}{{\mathcal{T}}}
\newcommand{\cX}{{\mathcal{X}}}

\newcommand{\cK}{{\mathcal{K}}}

\newcommand{\fC}{{\mathfrak{C}}}

\newcommand{\HH}{{\mathbb H}}

\newcommand{\PP }{{\mathbb P}}

\newcommand{\QQ }{{\mathbb Q}}
\newcommand{\CC }{{\mathbb C}}
\newcommand{\ZZ }{{\mathbb Z}}
\newcommand{\RR }{{\mathbb R}}

\newcommand{\ke }{{\varepsilon }}
\newcommand{\kb }{{\beta}}
\newcommand{\ka }{{\alpha}}

\newcommand{\kg }{{\gamma}}
\newcommand{\kG }{{\Gamma}}

\newcommand{\ch}{\mathrm{ch}}

\newcommand{\Dol}{\sA^{(0,\bullet)}_{\bar \partial}}
\newcommand{\dR}{\sA^{\bullet}_{\mathrm{dR}}}
\newcommand{\Dolbb}{\sA^{(\bullet,\bullet)}}

\newcommand{\rmlog}{\text{log}}
\def\nolog{ \omega_{\fC}}

\newcommand{\sHom}{\cH\kern -.5pt om}
\newcommand{\lan}{\langle}
\newcommand{\ran}{\rangle}

\newcommand{\cS}{\mathcal{S}}

\newcommand{\State}{\mathscr{H}}

\newcommand{\fB}{\mathfrak{B}}

\newcommand{\LL}{\mathbb{L}}

\newcommand{\vir}{\mathrm{vir}}

\newcommand{\sH}{\mathscr{H}}

\newcommand{\sA}{\mathscr{A}}

\newcommand{\Mgr}{\overline{M}_{g, r}}

\newcommand{\td}{\mathrm{td}}

\newcommand{\cB}{\mathcal{B}}
\newcommand{\Cone}{\mathrm{Cone}}

\newcommand{\fBo}{\fB^{\circ}}

\newcommand{\sG}{\mathscr{G}}

\newcommand{\tot}{\mathrm{tot}}

\newcommand{\cone}{\mathrm{Cone}}
\newcommand{\diag}{\mathsf{diag}}

\newcommand{\rC}{\mathrm{C}}
\newcommand{\tda}{\underline{\mathrm{td}}}

\newcommand{\CR}{\CC^{\ti}_{R}}

\newcommand{\virdim}{\mathrm{virdim}}

\newcommand{\bfdeg}{d_w}

\newcommand{\one}{\mathbbm{1}}

\newcommand{\age}{\mathrm{age}}
\newcommand{\inv}{\mathrm{inv}}

\newcommand{\tdch}{\mathrm{tdch}}

\newcommand{\tM}{{\tt{M}}}
\newcommand{\tX}{{\tt{X}}}
\newcommand{\tY}{{\tt{Y}}}
\newcommand{\tZ}{{\tt{Z}}}
\newcommand{\tT}{{\tt T}}
\newcommand{\tW}{{\tt W}}
\newcommand{\tV}{{\tt V}}

\newcommand{\cG}{\mathcal{G}}
\newcommand{\cR}{\mathcal{R}}

\newcommand{\fU}{\mathfrak{U}}

\newcommand{\bfk}{{\bf k}}
\newcommand{\vC}{\text{\bf \v{C}} }

\newcommand{\ath}{\widehat{at}}

\newcommand{\Thb}{\mathrm{Th}^{\bullet}}
\newcommand{\Th}{\mathrm{Th}^{\bullet}\mathrm{GK}}
\newcommand{\ThbGod}{\mathrm{Th}^{\bullet}\mathrm{G}}
\newcommand{\Kos}{\mathrm{K}}
\newcommand{\God}{\mathrm{G}}
\newcommand{\GodKos}{\mathrm{GK}}

\newcommand{\tN}{{\tt N}}

\newcommand{\ug}{\underline{h}}

\newcommand{\U}{\mathfrak{U}}
\newcommand{\Z}{\mathfrak{Z}}
\newcommand{\zzeta}{\mathfrak{n}}
\newcommand{\fullS}{S}

\newcommand{\GammaZ}{\Gamma _{\tZ}}
\newcommand{\ubfk}{\underline{\bfk}}

\newcommand{\tL}{{\tt L}}

\newcommand{\cXo}{\cX_0}
\newcommand{\So}{S^{\circ}}

\newcommand{\LGvir}{[\U]_W^{\vir}}
\newcommand{\Ugrd}{\U_{g, r, d}}

\newcommand{\tdchc}{\tdch _c}

\begin{document}

\title[General GLSM invariants \& their CoFTs]{General GLSM invariants and their cohomological field theories}

\begin{abstract}
We construct GLSM invariants for a general choice of stability in both the narrow and broad sector cases and prove they form a Cohomological Field Theory.  This is obtained by forming the analogue of a virtual fundamental class which lives in the local cohomology of the twisted Hodge complex.  This general construction comes from the use of two new ingredients. First, the use of the Thom-Sullivan and Godement resolutions applied to matrix factorizations are introduced  to handle poorly behaved (non-separated) moduli spaces.  Second, a localized Chern character map built from the Atiyah class of a matrix factorization is utilized to forgo the use of Hochschild homology.
\end{abstract} 

\maketitle 

\tableofcontents

\section{Introduction}

In mathematics, a gauged linear sigma model  $(V, \Gamma, \chi, w, \nu)$ is a fairly straightforward collection of data.     
Gauged linear sigma models (GLSMs)  
roughly consist of a choice of complex vector space $V$, a reductive subgroup $\Gamma \subseteq GL(V)$, 
a semi-invariant polynomial function $w$ on $V$, and a character $\nu$ of $\Gamma$.  The name GLSM is coined from high energy theoretical physics 
where this type of data is used to describe physical models for string theory.

GLSMs are an interesting object to study mathematically since they can specialize to objects appearing in both K\"ahler geometry and singularity theory (such as complete intersections in projective space and quantum singularities or affine Landau-Ginzburg models).  
Furthermore, by varying the character $\nu$, one can often make comparisons between these subjects (see Example~\ref{ex: LG/CY}).

This paper focuses on constructing a curve-counting enumerative theory for GLSMs broadly (pun intended) for many types of stability.  More precisely, to a GLSM $(V, \Gamma, \chi, w, \nu)$ one associates a graded vector space
\[
\State :=  \HH ^* (I\cX , (\Omega ^{\bullet}_{I\cX}, dw )) (\text{see }\S\ref{subsection:State})
\]
called the \emph{state space}.  Think of this as the cohomology of the GLSM.  
Our GLSM invariants are then a collection of linear maps
\[
\Omega_{g,r,d}: \State^{\otimes r} \to  H^*(\overline{M}_{g,r}^{an}, \CC) (\text{see }\S\ref{sec glsm inv})
\]
landing in the cohomology of the moduli space of genus $g$ curves of degree $d$ with $r$ markings.   

The GLSM invariants $\Omega_{g,r,d}$ are sometimes called the A-model of the GLSM.  They are a deformation invariant, independent of the complex structure of the GLSM.

Together, these maps satisfy a number of beautiful structural properties which were axiomatized by Kontsevich--Manin and called cohomological field theories.  Our invariants satisfy a later incarnation of these axioms.  The precise statement is the following.  
\begin{Thm}
Under mild assumptions (see page \pageref{eq: fixedloci}), the GLSM invariants
$$\{\Omega _{g, r, d}\}_{2g-2+r > 0, d\in  \Hom_{\ZZ} (\hat{G}, \QQ)}$$
form a cohomological field theory with unit in the sense of \cite{Pan18}.
\end{Thm}

This paper is, of course, not the first to construct enumerative invariants for GLSMs.  Instead, it is a continuation of the program set forth by Fan--Jarvis--Ruan \cite{FJR, FJR:GLSM} who constructed enumerative invariants for GLSMs in the narrow sector case as well as broad sector invariants for quantum singularities.  

Complementary work by Polishchuk--Vaintrob \cite{PV:MF}  built such invariants algebraically in the case of affine Landau-Ginzburg models.  Therein they constructed a type of fundamental matrix factorization which plays the role of a virtual fundamental class.  Using the associated Fourier-Mukai transform, they were able to build a cohomological field theory whose state space $\State$ is the Hochschild homology of the corresponding category of matrix factorizations.  This was later expanded by Ciocan-Fontanine--Favero--Gu\'er\'e--Kim--Shoemaker \cite{MF} to obtain enumerative invariants for convex hybrid models (see \S\ref{sec: convex}) in the broad sector case.

Recently, Kiem--Li \cite{KieLi} produced GLSM invariants for abelian affine LG models using cosection localization, 
 intersection homology, and Borel-Moore homology (see also \cite{ChKieLi}).  
 These invariants have the advantage of being topological in nature and form a cohomological field theory.

This paper takes a hybrid approach (this time, no pun intended) to the methods mentioned above.  As in \cite{MF, PV:MF}, we construct a fundamental factorization using the GLSM data.  However, rather than passing to Hochschild homology, we directly construct a cycle living in the local cohomology of a certain twisted Hodge complex (see Definition~\ref{def:virclass}).  To obtain this cycle, we construct a localized Chern character map (see Appendix~\ref{sec: B}) using the notion of Atiyah classes for matrix factorizations developed by Kim--Polishchuk \cite{KP}.  The role of our virtual cycle is then provided by the localized Chern character of the fundamental matrix factorization.

The main technical hurdle to overcome to obtain fully general GLSM invariants is the presence of potentially poorly behaved (non-separated and in particular not quasi-projective) moduli spaces.  Namely, in previous incarnations, the fundamental matrix factorization was constructed by choosing a $\Gamma$-acyclic Koszul factorization arising from the GLSM data.  The existence of this choice uses a form of projectivity of the moduli spaces involved.

As the spaces we consider need not even be separated, we require a more robust construction.  First, we observe that Koszul factorizations of a function $W$ arise from the commutative differential graded algebra (cdga) structure on the Koszul complex together with an element $\alpha$ of degree $-1$ satisfying $d\alpha = W$ (see \S\ref{sec:CDGfact}).  Second, we apply the Godement resolution to the Koszul complex to make it $\Gamma$-acyclic.  While this is no longer a sheaf of cdgas, it can be viewed as a cosimplicial sheaf of cdgas.  Hence, using the Thom-Sullivan construction (applied to the cosimplicial Godement resolution of the Koszul complex) we once again obtain a sheaf of cdgas which has the benefit of being $\Gamma$-acyclic.  This provides a fully general definition of the fundamental factorization.

While the components of the fundamental factorization are rather complicated, infinite, and not even quasi-coherent, it is still locally isomorphic to a locally-free factorization in the co-derived category of factorizations i.e.\ it is a perfect object (see Proposition~\ref{prop: perfect}).  Furthermore, the generality of the construction makes it rather flexible to work with, obeying pullback and base change properties.  We employ this flexibility readily in verifying the cohomological field theory axioms.

\subsection{Structure of the paper}
The paper is organized as follows.  To streamline the construction of the GLSM invariants, many of the technicalities in developing a trace map for non-separated spaces, a localized Chern character map for matrix factorizations, and properties of the Thom-Sullivan and Godement resolutions for $\cO_{\tX}$-modules are delegated to the appendix.

In \S\ref{sec: 2} we briefly review the concept of GLSMs and LG quasi-maps.  The new material is mostly Proposition~\ref{prop:U} which describes the construction of the DM stacks $\U_{g,r,d}$ where our fundamental factorizations live.

\S\ref{sec: 3} is concerned with the construction of the fundamental factorization.  For this, we develop the theory of TK factorizations i.e.\ those factorizations which are obtained from the Thom-Sullivan construction applied to the Godement resolution of a Koszul complex.   We define an equivalence relation on TK factorizations and prove that, despite the many choices in the construction of the fundamental factorization, all choices provide the same equivalence class.

\S\ref{sec: 4} is used to define the state space, its pairing, and our GLSM invariants.  The analogue of a virtual fundamental class is obtained as a localized Chern character of the fundamental factorization.  This is well-defined based on the independence of choices proven in the previous section.

In \S\ref{section: cohFT} we define the unit of our cohomological field theory and prove that the GLSM invariants satisfy all the cohomological field theory axioms.  Each axiom is proven in a separate subsection.  In addition, we prove that a certain graded component of our GLSM invariants provide a homogeneous cohomological field theory for convex hybrid models (see Theorem~\ref{thm: homog}) and demonstrate an K\"unneth property for sums of singularities (see \S~\ref{thm: homog}).

Appendix~\ref{sec: trace} describes a general trace map and its basic properties.  It is needed to define our GLSM invariants and the pairing on our state space.  This is essentially the integration map but it is a bit subtle due to the fact that our spaces need not be separated (and hence their analytification need not be Hausdorff).  We make use of the fact that our spaces admit a separated cover (and hence their analytifications admit Hausdorff covers).  This allows us to define a trace map following  the philosophy of \cite{CM2} i.e.\ using the Mayer-Vietoris resolution.

Appendix~\ref{sec: B} discusses many of the properties of factorizations we use.  Importantly, this appendix is where we construct our localized Chern character map.

Appendix~\ref{sec: TS} hashes out the details of the Thom-Sullivan and Godement construction in the context of sheaves on algebraic stacks.

For the reader's convenience, we also included a notational glossary as Appendix~\ref{sec:glossary}.

\subsection{Acknowledgements}

This project was inspired by collaboration and discussions with Ionut Ciocan-Fontanine, Mark Shoemaker, and J\'er\'emy Gu\'er\'e who we thank heartily.
We are also very grateful to Alexander Polishchuk for suggesting the use of the Thom-Sullivan functor, which has become a crucial aspect of the paper.  B. Kim was  supported by KIAS individual grant MG016403.  D. Favero was supported by NSERC through the Discovery Grant and Canada Research Chair programs.  The Fields Institute also supported  the authors for long-term visits where they collaborated on this project.  We thank Fields for a pleasant stay and remarkable work environment.

\subsection{Notation and conventions}

We let the base field $\bf{k}$ be the field  of complex numbers. \label{k}
 We do not require our DM stacks to be separated.  A full notational glossary can be found in Appendix~\ref{sec:glossary}.

\section{Gauged linear sigma models} \label{sec: 2}

This paper is mainly concerned with constructing enumerative invariants for gauged linear sigma models (GLSMs). 
These are roughly GIT quotients of affine space equipped with a function (called the \emph{superpotential}).  
The precise definition is a bit more intricate and we provide it now.  Similar invariants for GLSMs have previously been defined in \cite{MF, FJR:GLSM}.  
We refer the reader to these works for further details.

\subsection{Input data}\label{sec:input} A {\em gauged linear sigma model\/}  \cite{FJR:GLSM} for this paper  is a collection 
\[ (V, \Gamma, \chi, w, \nu) \]
consisting of:
\begin{itemize}
\item  a finite dimensional $\CC$-vector space $V$;
\item a reductive algebraic subgroup $\Gamma \subseteq GL(V)$;
\item a surjective character $\chi : \Gamma \twoheadrightarrow \CC^{\ti}$ of $\Gamma$; 
\item a $\chi$-invariant polynomial function $w: V \ra \mathbb{A}^1$, i.e., $$w\in (\CC_{\chi}\ot \Sym V^{\vee})^{\Gamma};$$ and
\item a $\QQ$-character $\nu$ of $\kG$,
\end{itemize}
subject to the following with $G:= \Ker (\chi )$ and $\theta:=\nu |_{G}$:

\begin{itemize}
\item  $V^{ss}(\nu) = V^{ss}(\theta) = V^{s}(\theta)$.
\item The critical locus $Z(dw)$  of  the function $[V^{ss}(\theta)/G]\ra \mathbb{A}^1$ induced from $w$ is proper over $\Spec \/ \CC$.
Abusing notation, we will denote the induced function also by $w$.
\end{itemize}
The quotient $\cX:=[V^{ss}(\theta)/G]$ is a separated DM stack since the $G$-action on $V^{ss}(\theta) = V^{s}(\theta)$ is proper.  \label{cX}

Since $\chi$ is a surjective character, there is a subgroup isomorphic to $\CC^{\ti}$, in the center of $\Gamma$ such that
the subgroup together with $G$ generates $\Gamma$.
We  make a choice of such a subgroup and denote it by $\CR$. 
\label{rcharge}  Let $d_w$ \label{dw} be the positive weight of $\chi |_{\CR}$
and let $J :=\exp (2\pi i /d_w)$. Then
note that there is a commutative diagram of  exact sequences of groups
\[ \xymatrix{ 1 \ar[r]  &     \lan J \ran \ar[r] \ar[d]   & \CR \ar[d] \ar[r]^{t\mapsto t^{d_w}} & \CC ^{\ti} \ar[r] \ar[d]_{=} &  1 \\
             1 \ar[r]  &       G \ar[r] & \Gamma \ar[r]^{\chi} &    \CC ^{\ti} \ar[r] &  1 .} \]
If $\widehat{\kG}$, $\widehat{\CR}$, $\widehat{G}$ denote the character groups of $\kG$, $\CR$, $G$ respectively,
this is a canonical isomorphism $\widehat{\kG}\ot \QQ \xrightarrow{\cong} (\widehat{\CR} \oplus \widehat{G})\ot \QQ$
between the $\QQ$-character groups.

\subsection{Convex hybrid models} \label{sec: convex}

\begin{Def} \cite[\S 1.4]{MF} 
A GLSM is called a {\em hybrid model } if there is a decomposition
$V=V_1\oplus V_2$ as a $\Gamma$-representation such that:
\begin{enumerate}
\item the $\CR$-action on $V_1$ is trivial and acts with positive weights on $V_2$,
\item the equality $V^{ss}(\theta) = V_1^{ss}(\theta) \times V_2$ holds.
\end{enumerate}
\end{Def}

Let $G_1$ be the quotient group $\Gamma /\CR$.
\begin{Def} \cite[\S 4.1]{MF}  A hybrid model is called {\em convex} if
for any representable morphism $f: C \to [V_1^{ss}/G_1]$ from any genus $0$ prestable orbicurve $C$ with an arbitrary number of markings,
the vector space $H^1(C, f^*T_{ [V_1^{ss}/G_1]/BG_1})$ is zero (where $T_{ [V_1^{ss}/G_1]/BG_1}$ is the relative tangent sheaf of  $[V_1^{ss}/G_1]$ over $BG_1$).
\end{Def}

\begin{Rmk}
When $G$ is abelian and $ [V_1^{ss}/G_1]$ is a variety, this notion of convexity agrees with the usual notion of convexity for the variety $ [V_1^{ss}/G_1]$.  This can be seen using the generalized 
Euler sequence.
\end{Rmk}

\begin{example} \label{ex: LG/CY}
The most famous examples of GLSMs are
\[
\mathbb{X}_{\pm} := (V=\CC^{\oplus 6}, \Gamma=(\CC^{\ti})^2, \chi, w , \nu_{\pm})
\] 
where
$\Gamma$ acts on the coordinates of $V$ by the weight matrix
\[ \left(\begin{array}{cccccc} 1 & 1 & 1 & 1 & 1 & -5 \\
                                      0 & 0 & 0 & 0 & 0 & 1 \end{array} \right),\] 
                                  the superpotential  is $w = x_6f(x_1, ..., x_5)$ for a homogeneous polynomial $f$ of degree 5  which defines a smooth quintic $3$-fold,
and
the characters are defined as  $\chi(s, t)=t$, $\nu_{+} (s, t) = s$, $\nu_{-}(s, t)=s^{-5}t$.  We also make the following choices for the R-charge action
\[ \CC^{\ti}_{R,+}= \{ (s, t) \in \Gamma \ | \ s=1\} , \ \ \CC^{\ti}_{R,-}= \{ (s, t) \in \Gamma \ | \ s^{-5}t=1 \}. \]
Then both GLSMs $\mathbb{X}_{\pm}$ are convex hybrid models with their respective choice of R-charge $\CC^{\ti}_{R,\pm}$.
The GIT quotient for the GLSM $\mathbb{X}_+$ is the total space of $\cO_{\PP^4}(-5)$ and the critical locus of $w$ is the quintic $3$-fold $Z(f)$.  On the other hand, the GIT quotient for the GLSM $\mathbb{X}_-$ is the quotient stack $[\mathbb A^5 / \ZZ_5]$ and the critical locus is the origin.
The GLSM $\mathbb{X}_+$ is called the CY phase.  Its GLSM invariants should correspond to the Gromov-Witten theory of $Z(f)$.  The GLSM $\mathbb{X}_{-}$ is called the LG phase.  Its GLSM invariants should correspond to the FJRW theory of $[\mathbb A^5 / \ZZ_5]$ with function $f$. (For example, these last two statements were proven in \cite{MF} for the similar GLSM invariants constructed therein.)
\end{example}

\subsection{Landau-Ginzburg quasimaps} \label{subsec: def LG qmaps}

\begin{Def}\label{def:LG quasimap}  \cite[Definition 4.2.2]{FJR:GLSM}
An {\em LG quasimap} to $[V^{ss}/G]$ of type $(g, r, d)$ over a scheme $T$ consists of:
\begin{enumerate}

\item a genus $g$ prestable orbicurve $C$ over $T$, which by definition comes with gerbe markings $\sG _i$ and sections of the gerbe markings $T \ra \sG_i$, $i=1, ..., r$;

\item a principal $\Gamma$-bundle $P$ on $C$ 
such that the induced morphism $C \ra B\Gamma$ associated to $P$ is representable;

\item\label{cond:kappa} an isomorphism $\kappa : P \times _{\Gamma} \CC _{\chi} \to \omega_{C}^{\mathrm{log}}$ of line bundles;  and 

\item a section $u: C \to P \ti _{\Gamma} V $ of a vector bundle $ P \ti _{\Gamma} V $ 
such that  for each geometric fiber $C_t$ of a geometric point $t\ra T$, the subset $B_t:= u ^{-1}(P \ti _{\Gamma} V - P \ti _{\Gamma}V^{ss})$ of
$C_t$ is a finite set away from the nodes and markings of $C_t$.

\end{enumerate}

\end{Def}

\medskip

 For simplicity we will often write the LG quasimap data as $(C, P, \kappa, u)$ omitting $\sG_i$, $i=1, ..., r$.
For any $\Gamma$-representation space $Y$ (or more generally a scheme $Y$ with a left $\Gamma$-action), 
we will often write $P(Y)$ instead of the quotient $P\ti _{\Gamma} Y := (P\ti Y )/ \Gamma $ which 
is an algebraic space over $C$ since $P$ is representable over $C$.  
Here the left $\Gamma$-action on $P\ti Y$ is given by $h\cdot (p, y) := (p\cdot h^{-1}, h\cdot y )$
for $h\in \Gamma$, $(p, y) \in P \ti Y$.  \label{pg: P(Y)}
For example, we have a vector bundle $P(V)$ and a line bundle $P (\CC_{\delta})$ 
for a character $\delta$ of $\Gamma$.                   
We denote by 
\[
[u] : C \to [V/\kG]
\]
the map induced by the section $u$. 
Elements of $B_s$ are called  base points.

\medskip

For the following definition, choose a positive integer $l$ such that $l\nu$ becomes a $\ZZ$-valued character,
let ${\bf e}$ be the least common multiple of the set of $l$ times the exponents of the automorphism groups of the geometric points of $[V^{ss}/G]$,
and let $L_\delta$ denote the line bundle $V\ti_{\kG} \CC_{\delta}$ on $[V/\kG]$.

For $b\in B_s$, let $l(b)$ be the length of the base point $b$ with respect to $\nu$.  This is defined as follows.  
For every $s \in H^0 ([V/\kG ] , L_ {m\nu} )$, let $\mathrm{ord}_{b}( [u]^*s)$ be the vanishing order of zero of $[u]^*s$ at $b$; see \cite[Definition 7.1.1]{CKM}
and \cite[Definition 2.4]{CK: bigI}.  Then
\[
l(b) := \mathop{min }\left \{ \frac{\mathrm{ord}_{b}([u]^*s)}{m} \ |\  \forall  m > 0, s  \text{ with } 
 [u]^*s \not\equiv 0 \right \}. \]

\begin{Def} \label{def: stab of LG maps}  
An LG quasimap to $[V^{ss}/G]$ is called {\em $\nu$-stable} (or simply stable) if for every geometric point $s$ of $S$
\begin{enumerate}

\item the $\QQ$-line bundle $\omega _{\underline{C}_s}^{\log} \ot (\rho_*(P(\CC_{\nu})|_{C_s}  ^{\ot {\bf e}} ) ) ^{1/{\bf e}}$ is ample,
where $\rho: C_s\to \underline{C}_s$ is the coarse moduli space  morphism; see \cite[\S 2.3]{CCK}, \label{coarse} and

\item  for $b\in B_s$, the inequality $ l(b) \le 1$ holds.  
\end{enumerate}
We call an LG quasimap {\em $\infty$-stable} (resp. $0+$-stable) if it is $\ke\nu$-stable for every large enough $\ke \in \QQ_{>0}$ (resp. for every 
small enough $\ke \in \QQ_{>0}$).
\end{Def}

Let $Y$ be a closed subscheme of $V$ and let $\mathcal{Y} := [(Y\cap V^{ss})/G]$.
A (stable) LG quasimap to $\cY$ is defined to be a (stable) LG quasimap to $[V^{ss}/G]$ such that the section
$u$ of $\cV$ factors through $P (Y )$.

The degree of an LG quasimap is, by definition the degree of $P$.  This is the morphism $d \in \Hom_{\ZZ} (\widehat{G}, \QQ )$ \label{characters} defined by 
$d(\delta) := \deg P(\CC_{\delta})$ for all $\delta \in \mathrm{Ker} (\widehat{\kG} \to \widehat{\CR})\ot _{\ZZ} \QQ = \widehat{G}\ot_{\ZZ}\QQ $.  
The type $(g,r,d)$ of a (stable) LG quasimap records the genus, number of markings, and degree respectively.

We will make use of the following theorem in what follows.
\begin{Thm} \cite[Theorem 1.1.1 \& Theorem 5.3.1]{FJR:GLSM} \label{FJR proper}
The moduli stack $LG_{g, r, d} (\cY)$ of stable LG quasimaps to $\cY$ of type $(g, r, d)$
 is a separated DM stack of finite type. When $\cY$ is proper over $\Spec \CC$, so is $LG_{g, r, d}(\cY)$.
\end{Thm}
For simplicity we will often write $LG(\cX) := LG_{g, r, d}(\cX )$ when $g,r,d$ are implicit. 
 \label{LG}

\subsection{The construction of $\U$}\label{sec:const U}
 
 In this section, we construct a smooth, finite-type DM ${\bf k}$-stack $\U_{g, r, d}$, containing $LG(\cX)$ as a closed
substack.  This will be the home of the virtual factorization constructed in \S\ref{sec: 3}.

\begin{Def} Let $L$ be a line bundle on a pointed prestable orbicurve $C$. 
We say $C$ is {\em $L$-stable} if
the degree of $L$ restricted to any irreducible component of $C$
is non-negative and
 the degree of $L\ot \omega _C^{\log}$ restricted to any irreducible component of $C$ is strictly positive.
When $L = P(\CC _\nu)$ for a principal $\Gamma$ bundle $P$ on $C$, we simply call such $C$ {\em $\nu$-stable}.
\end{Def}

\begin{Rmk}
For those $((C, \sG_1, ..., \sG_r), P)$ allowing $\nu$-stable LG quasimaps $((C, \sG_1, ..., \sG_r), P, \kappa, u)$,
 the degree of the line bundle $P(\CC _{\nu})$ restricted to any irreducible component of $C$ is non-negative
(see \cite[Propostion 5.1.1]{FJR:GLSM}  or Remark \ref{rmk: degree and triviality}).
Hence, only $\nu$-stable orbicurves admit $\nu$-stable quasimaps.
\end{Rmk}

Fix the numerical data $g \in \ZZ_{\ge 0}$, $r\in \ZZ_{\ge 0}$, $d \in \Hom_{\ZZ} (\widehat{G}, \QQ )$. 
Denote by $\fB_{\Gamma, \chi, \nu}^{g, r, d}$ or simply by $\fB_{\Gamma}$ \label{fB}
 the moduli stack
parametrizing genus-$g$, $r$-pointed $\nu$-stable orbicurves $C$ together with a degree-$d$ principal $\Gamma$-bundle $P$ and an isomorphism 
$\kappa: P (\CC _{\chi}) \to \omega _C^{\rmlog}$ such that
the induced map $C \to B\Gamma$ is  representable.

We often write objects in $\fB_{\Gamma}$ as $(C, P, \kappa)$ omitting $\sG_i$, $i=1, ..., r$.
The stack $\fB_{\Gamma}$ is  a smooth, locally finite type, Artin stack of pure dimension $3g-3 + (g-1)\dim G$.

Let $\pi: \fC \to \fB_{\kG}$ be the universal curve over $\fB_{\kG}$
and let $\cP$ be the universal principal $\Gamma$-bundle on $\fC$. There is an associated vector bundle $\cP (V)$ which we also denote 
by $\cV$.
For a map $S \to \fB_{\kG}$ from an arbitrary algebraic stack $S$ we have the universal curve $\pi_{S} : \fC_{S} \to S$
and the universal bundle $\cP _{\fC_{S}}$ on $\fC_S$ obtained by pullback. 
Abusing notation, we will often write $\pi, \fC, \cP, \cV$ instead of $\pi _S,  \fC_{S}, \cP _{\fC_{S}}$, $\cP_{\fC_{S}}(V)$, respectively.
  \label{pi}  \label{cV}

\begin{Prop}\label{prop:U}   
For every map $S \to \fB_{\kG}$ from an algebraic stack $S$
there is an open substack $\So$ of $S$ satisfying 
the following.

\begin{enumerate} 
\item  On the universal curve $\fC$ over $\So$
there is a $\pi_*$-acyclic 
coherent resolution of $\cV$:
\[ 0   \ra \cV \ra \cA \ra \cB \ra 0  \] for which \label{cAB}
\begin{enumerate}
\item there is a   {\em restriction map} 
\[
{rest}_{\cA}: \cA \ra \cV|_{\sG} := \oplus _i \cV |_{\sG _i} 
\]
 of sheaves of $\cO_{\fC}$-modules which is compatible with the map $\cV \ra \cV |_{\sG}$ and whose kernel is also $\pi_*$-acyclic (this implies
$\pi_*\cA \to \pi_* (\cV|_{\sG})$ is surjective);

\item $\pi_*\cA$ and $\pi_*\cB$ are locally free coherent sheaves.

\end{enumerate}
Furthermore, we can choose $\cA$, $\cB$ to be locally-free sheaves of finite rank. 
\item The formation of $\pi_*\cA$ and $\pi_*\cB$ is functorial under the base change of $\So$. 

\item Denote by $[A \xrightarrow{d_A} B] := [\pi_*\cA \to \pi_*\cB] $ \label{AB} and let $p_A : \tot A \to \So$ be the projection from the total space of the vector bundle
$A$ to $\So$ (here we allow the ranks of $A$, $B$  to be locally constant).
Then $LG(\cX)\ti _{\fB_{\kG}} \So$ is canonically an open substack of the zero locus of the section $p_A^*d_A \circ t_A$ of $p_A^*B$, where $t_A$ is the tautological section of $p_A^*A$.

\item 
Assume $S\to \fB_{\kG}$ is the identity map.  There is an open substack $\U_{g, r, d}$ of $\tot A$ containing $LG(\cX) $ as the zero locus of $\beta := d_A \circ t_A |_{\Ugrd} $.
Moreover $\U_{g, r, d}$ is a DM stack of finite type over $\Spec\/\CC$.
\label{beta}

\item  Assume $S\to \fB_{\kG}$ is the identity map and the GLSM is a convex hybrid model. 
Then, there exists a choice of $\U_{g, r, d}$ which is separated over $\Spec\/\CC$. 
\end{enumerate}
\end{Prop}

\begin{proof}

(1) This is proven in \cite[Lemma 3.4.1, Corollary 3.4.2]{MF}. We spell out some of the details. 
Let  $\sG$ denote the disjoint union $ \coprod_i \sG _i$
of $\sG_i$, let 
\[
\cO_{\fC} (1):= (\omega _{\fC}^{\log} \ot \cP ( \CC _{\nu}))^{\bf e}
\]
 where ${\bf e}$ is a (large enough) positive integer defined in \S \ref{subsec: def LG qmaps}, and let 
\[
\cA_m := \pi^*(\pi_*(\cO_{\fC} (m)))^{\vee} \ot \cV_{\fC} (m).
\]
Then there exists a large enough integer $m$ such that  $\cA _m (-\sG )|_C$ 
is $(\pi|_{C})_*$-acyclic for any $\nu$-stable LG quasimap $C = \fC_s$ of type $(g, r, d)$. 
We want to pick $m \gg 0$ so that the whole of $\cA _m (-\sG )$ is $\pi _*$-acyclic.

 By base change (see \cite[Proposition A.85]{BBH-R}), 
 it is enough to show acyclicity when $S = \fB _{\Gamma}$.
Since  $\pi_*$-acyclicity is an open condition on $\fB_{\Gamma}$, 
we can fix a large enough integer $m$ 
such that $\cA _m ( - \sG ) |_ {\fC|_{\fBo_{\kG} }} $ is $\pi_*$-acyclic over some open substack $\fBo_{\Gamma}$ of $\fB_{\Gamma}$ and hence over a general 
$\So:=S\ti _{\fB_{\kG}} \fBo_{\kG} \to \fB_{\kG}$. 
Now we take $\cA$ to be the subsheaf of $\cA_m|_{  \fC|_{\fBo_{\kG} } } $ for which the restriction map $\cA \to \cA |_{\sG}$ factors through the inclusion
$\cV|_{\sG} \to \cA|_{\sG}$.  To verify acyclicity, see  \cite[Lemma 3.4.1, Corollary 3.4.2]{MF}.

(2) follows from  \cite[Proposition A.85]{BBH-R} which is permissible by  the flatness of $\pi$.

(3)  
By base change (\cite[Proposition A.85]{BBH-R} again), it is enough to handle the case when $S \to \fB_{\Gamma}$ is the identity. 
Let $R:= \RR^1 \pi _* (\cV ^{\vee} \ot \omega_{\fC})$ be the zero-th cohomology sheaf $\cH^0 [B^{\vee} \to A^{\vee}]$ of the complex $[B^{\vee} \to A^{\vee}]$ in amplitude $[-1, 0]$. 
Note that, for any scheme $T$ over $\fBo_{\kG}$,  $$\mathbf{Spec}\, (\Sym R) (T) =  
 \Hom ( R|_T, \cO _T) \cong  \Gamma (T, (\pi |_T)_* (\cV |_T )). $$
 Therefore the cone stack
$\mathbf{Spec}\, (\Sym R)$ coincides with $ LG(\cX)$ after imposing stability conditions.  Since the stability conditions are open conditions,
$LG(\cX)$ is an open substack of $\mathbf{Spec}\, (\Sym R)$.
On the other hand, by \cite[Lemma 3.6.2]{MF}, the cone stack
$\mathbf{Spec}\, (\Sym R)$ is canonically isomorphic to  the zero locus $Z(p_A^*d_A \circ t_A ) $ of $p_A^*d_A \circ t_A$.

(4) The base-length stability condition (2) of Definition \ref{def: stab of LG maps} as well as condition (4) in Definition \ref{def:LG quasimap} are open 
conditions and can be imposed on $\tot A$ to obtain an open substack $\U_{g, r, d}$ of $\tot A$ 
on which the zero locus of section $p^*d_A \circ t_A$ (restricted to $\U_{g, r, d}$) is exactly $LG(\cX)$.  
To show that $\U_{g ,r ,d }$ is a DM stack, 
as $\tot A$ is an Artin stack of locally finite type over ${\bf k}$,  it is enough to show that the diagonal morphism $\Delta _{\U_{g, r, d}} : \U_{g, r, d} \to \U_{g, r, d} \ti _{\bf k} \U_{g, r, d}$ is unramified 
for some open substack containing $LG(\cX)$. We know that the representable morphism $\Delta _{LG(\cX)}$ is unramified. That is, $\Delta_{\tot A}|_{LG(\cX)}$ is 
locally of finite type and $\Omega^1 _{\Delta_{\tot A}} |_{LG(\cX)} = 0$.  By Nakayama's lemma, these conditions hold true for some open substack  containing $LG(\cX)$.
Since $LG(\cX)$ is of finite type and $\tot A$ is of finite type, we may take an open substack $\U_{g, r, d}$  containing $LG(\cX)$ such that $\U_{g, r, d}$
is a DM stack of finite type.

(5) We apply \cite[Theorem 4.3.4]{MF}.  This requires a bit of justification.  There they only considered $\infty$-stability. However, the $\infty$-stabilty was used to ensure that 
$\U_{g, r, d}$ is a  global quotient stack with quasi-projective coarse moduli.  We do not need this here. 
Hence, we have the following (with arbitrary $\nu$-stability):
 \begin{itemize}
\item[--] a $\pi_*$-acyclic, coherent, locally-free resolution 
$\cA_1 \to \cB_1$ of $\cV_1 : = \cP (V_1)$ with a  homomorphism $\cA_1 \to \cV_1|_{\sG} $ which is compatible with $\cV_1 \to   \cV_1|_{\sG} $;
\item[--] a separated finite type DM stack $\U_1$ which is 
an open substack of $\tot A_1$ such that
the zero locus of $\beta _1 \in \Gamma(\U_1, p_1^*\pi_*\cB_1)$ is canonically isomorphic to $LG_{g,r, d}([V^{ss}_1/G])$.
\end{itemize}
 Choose a $\pi_*$-acyclic, coherent,
locally free resolution  $\cA_2 \to \cB_2$ of $\cV_2 :=\cP (V_2)$ with 
a  homomorphism $\cA_2 \to \cV_2|_{\sG} $ compatible with $\cV_2 \to   \cV_2|_{\sG} $.
We take $\U_{g, r, d}:= \U_1 \times_{\fBo_{\Gamma}} \tot A_2 $, which is obviously a separated DM stack over $\bf k$.
\end{proof}

\begin{Rmk}
In the proof of Proposition \ref{prop:U} (3), we demonstrated that
$LG_{g, r, d}$ is an open substack of $\Spec \/ \Sym (\RR ^1 \pi _* \cV ^{\vee} \ot \omega _{\fC})$.
Since 
\[
\cH ^0 (\Sym \RR \pi _* \cV ^{\vee}) \cong \Sym (\RR ^1 \pi _* \cV ^{\vee} \ot \omega _{\fC}),
\] 
$LG_{g, r, d}$ is an open substack of $\pi_0$ of the dg manifold 
$$\RR LG_{g, r, d} := \Spec \/ \Sym \RR \pi _* \cV ^{\vee}$$ over $\fBo _{\Gamma}$ in the sense of \cite{Ci-Ka}. 
\end{Rmk}

\begin{Def}\label{def: U} From now on we will assume that
the morphism $S \to \fB _{\Gamma}$ is of relative DM type and of finite type. We call such a morphism 
a {\em DM finite type morphism}.
We denote $\U_{g, r, d} \ti _{\fB _{\kG}} S$ by $\U_S$ (or simply $\U$ when $S \to \fB _{\kG}$ is implicit). 
It is a finite type DM stack over $\bfk$.
\end{Def}

\subsection{Evaluation maps} \label{sec: ev}

Recall that $\cX = [V^{ss}/G]$.  We denote its inertia stack by $I\cX$. \label{inertia}

Let $G/G$ denote the set of conjugacy classes of $G$, $\rC_{G} (h) $ be the centralizer of $h$ in $G$, and
$(V^{ss})^{h}$ denote the fixed locus of $V^{ss}$ under the action of the cyclic group generated by $h$.
As explained in \cite[\S 3.1]{MF}, the surjective restriction map ${rest}_{\cA}: \cA \to \cV|_{\sG}$ in
 Proposition \ref{prop:U} (1) yields  a smooth evaluation map 
\[ \prod _{i=1}^r ev_i : \U \ra (I\cX = \coprod _{(h) \in G/G} [ (V^{ss})^{h} / \rC_{G} (h) ])^r . \]

\section{Construction of a virtual factorization} \label{sec: 3}

Fix a morphism $S\to \fB_{\Gamma}$ as in Definition \ref{def: U}. 
In this section, we construct a factorization for the function $-\sum _i ev^*_i w$ on the space $\U$ from Proposition~\ref{prop:U}, called a virtual factorization for the
moduli space $LG(\cX)\ti_{\fB_{\kG}} S$.
We show that it is well-defined in a suitable sense.
The most interesting case is when $S= \fB_{\Gamma}$ which allows us to construct GLSM invariants. 
However, we will treat the general case as it is required to check that these invariants form a cohomological field theory.

\subsection{The map $\mathfrak a_{\fullS}$ in derived category}\label{sec:derived step}

Let $\text{tdeg} (w)$ denote the total degree of the polynomial $w$.
When there is at least one marking $r\ge 1$, we consider the following natural maps in the bounded derived category of coherent sheaves on $\fullS$
\begin{eqnarray}
\label{eqn:wkappa} 
\Sym^{\le \text{tdeg}(w)} \RR\pi _* \cV  & \xrightarrow{\text{natural map}} & \RR\pi _* \Sym^{\le \text{tdeg}(w)} \cV \\
\nonumber    & \xrightarrow{\tilde{w}} & \RR \pi _* \cP\times _{\Gamma} \CC _{\chi}   \\
\nonumber  &  \xrightarrow{\tilde{\kappa} }& \pi _* \omega_{\fC}^{\mathrm{log}}  \\
\nonumber  &  \xrightarrow{rest}&      \cO ^{\oplus r}_{\fullS}.
\end{eqnarray}
where $\tilde{w}$, $\tilde{\kappa}$ are induced maps from $w$ and $\kappa$ respectively, and $rest$ is induced from the restriction map $\omega_{\fC}^{\mathrm{log}} 
\to \oplus _i \omega_{\fC}^{\mathrm{log}} |_{\sG_i}$. Starting from the mapping cone of \eqref{eqn:wkappa}, we have
\begin{align*}
 \cone (\Sym^{\le \text{tdeg}(w)} \RR\pi _* \cV \to  \cO ^{\oplus r}_{\fullS} ) [-1]
\to \ & \cone (\RR\pi_*\omega_{\fC}^{\mathrm{log}} \to \cO^{\oplus r}_{\fullS} ) [-1]\\
\to \ & \RR\pi _* \omega_{\fC} \\
\to \ & \RR^1 \pi_* \omega_{\fC} [-1] \\
\to \ &  \cO_{\fullS}  [-1] . \end{align*}

The dual of the composition of the above maps 
yields a map 
\begin{eqnarray}\label{eqn:alphaS}  \mathfrak a_{\fullS}: \cO_{\fullS} [1] \ra  \mathrm{Cone} ( \cO ^{\oplus r}_{\fullS} \ra \Sym^{\le \text{tdeg}(w)}((\RR\pi _* \cV) ^{\vee})) ) \end{eqnarray}
 in the derived category of coherent sheaves such that
  the composition
 \begin{eqnarray}\label{eqn:diag}  \diag _{\fullS} : \cO_{\fullS} [1] \ra \mathrm{Cone} ( \cO ^{\oplus r}_{\fullS} \ra \Sym \RR\pi _* \cV ^{\vee})  \ra \cO ^{\oplus r}_{\fullS} [1] 
 \end{eqnarray} is the diagonal map (which 
 is the dual of the sum map).
  However, as cones are not functorial, the map in \eqref{eqn:alphaS} is not uniquely determined as above.  We now recall how to obtain a unique description of it
  following \cite{MF, PV:MF}.
 
Let
\[
F_m :=\text{Cone}(\text{Sym}^m \mathbb R \pi_* \cV \to \cO_{\fullS}^{\oplus r})[-1]
\]
and consider the $m^{th}$ symmetric power of the universal curve
\[
\text{Sym}^m \mathfrak C := [\overbrace{\mathfrak C \times_{\fullS} ... \times_{\fullS} \mathfrak C}^{m-times} / S_m]
\]
where $S_m$ is the symmetric group on $m$ letters acting by permutation.
The $m^{th}$ diagonal map is $S_m$-equivariant, inducing a map of stacks,
\[
\Delta_m : [\mathfrak C / S_m] \to \text{Sym}^m \mathfrak C.
\]

When the degree $m$ component of $w$ is nonzero, this induces a morphism of exact sequences on $\text{Sym}^m \mathfrak C$
\begin{equation}\label{eq:Alpha}
\begin{tikzcd}
0 \ar[r] & \text{ker}  \ar[r] \ar[d, "f_m"] & \cV^{\boxtimes m} \ar[r] \ar[d] &  (\Delta_m)_*\omega_{\fC}^{\mathrm{log}} |_{\sG}  \ar[d] \ar[r] & 0   \\
0 \ar[r] &   (\Delta_m)_* \nolog  \ar[r] &  (\Delta_m)_* \omega_{\fC}^{\mathrm{log}} \ar[r] &   (\Delta_m)_* \omega_{\fC}^{\mathrm{log}} |_{\sG} \ar[r] &  0 \\
\end{tikzcd}
\end{equation}
where $\sG : = \coprod_i \sG _i$   and $f_m$ is the induced map from the $m$-th degree part of $w$.

 Hence, pushing forward by the projection $\pi_m : \text{Sym}^m \mathfrak C \to {\fullS}$  and composing with the Grothendieck trace map we obtain
\begin{align}  
F_m & \xrightarrow{ (\mathbb R(\pi_m)_*f_m) } \mathbb R\pi_*\nolog \notag \\
& \xrightarrow{\text{H}^1} \mathbb R^1\pi_*\nolog[-1]  \notag \\
& \xrightarrow{\text{trace}}  \cO_{\fullS}[-1]  \label{eq: finalalpha}
\end{align}
(see the remark below for further justification).
Dualizing and summing over $m$ we uniquely define \eqref{eqn:alphaS} for $r \geq 1$.
In the special case when $r=0$, we just consider (the dual of) the map  \eqref{eqn:wkappa} with the last line dropped.

\begin{Rmk}
Notice that we can identify 
\[
(\pi_m)_* \cV^{\boxtimes m} =  ((\pi^{\times m})_* \cV^{\boxtimes m})^{S_m} = \Sym^m \pi_* \cV.
\]
  If we right derive the composition we obtain
 \[
 \mathbb R(\pi_m)_* = (\bullet)^{S_m} \circ \mathbb R(\pi^{\times m})_*
 \]
  (since taking invariants under a finite group is exact in characteristic 0).   This gives
 \begin{align*}
 \mathbb R (\pi_m)_* \cV^{\boxtimes m} & =  (\mathbb R \pi_* \cV^{\boxtimes m})^{S_m}  = ([A \to B]^{\boxtimes m})^{S_m} \\
 &  = \Sym^m [A \to B] = \Sym^m \mathbb R \pi_* \cV.
 \end{align*}
 Hence, applying $\mathbb R (\pi_m)_*$ to the top line of \eqref{eq:Alpha} we obtain an isomorphism of exact triangles,
\begin{equation*}
\begin{tikzcd}
 \mathbb R (\pi_m)_* \text{ker}  \ar[r] \ar[d,  "="] & \mathbb R (\pi_m)_*\cV^{\boxtimes m} \ar[r]   \ar[d,  "="] &  \mathbb R (\pi_m)_*(\Delta_m)_*\omega_{\fC}^{\mathrm{log}} |_{\sG}   \ar[d,  "="]  \\
F_m  \ar[r] & \Sym^m \mathbb R \pi_* \cV \ar[r] &  \cO_{\fullS}^{\oplus r}   \\
\end{tikzcd}
\end{equation*}
which gives the identification $F_m =  \mathbb R (\pi_m)_* \text{ker}$ which we use in \eqref{eq: finalalpha}.
\end{Rmk}

\subsection{Cochain map realizations of $p^*\mathfrak a_{\fullS}$}\label{sec:cochain}
Let $\U$, $A, B, p_A, \beta$ be as in Proposition \ref{prop:U}. 
Consider the diagram of complexes of $\Sym A^\vee$-modules
\[ {\small
\begin{tikzcd}
\Sym A^\vee [1] \ar[r, "p_A^*\mathfrak a_{\fullS}"]  \ar[ddr, bend right, dashed, "\mathfrak a_{\tot A}"] & \mathrm{Cone} ( \cO ^{\oplus r}_{\tot A} \to  \Sym [B^\vee \to A^\vee] \ot_{\cO_S} \Sym A^\vee ) \ar[equal]{d}  \\
  & \ \ \ \mathrm{Cone} ( \cO ^{\oplus r}_{\tot A} \to  [... \to B^{\vee} \ot \Sym A^\vee \to \Sym A^\vee ] \ot_{\cO_S} \Sym A^\vee ) \ar[d, "{\text{Id}_{ \cO ^{\oplus r}_{\tot A}[1]} \oplus (m \otimes \text{Id}_{\wedge^*B^{\vee}})}"]  \\
& \ \ \ \ \ \ \ \ \ \ \ \ \ \ \ \ \ \ \ \ \ \ \ \  \mathrm{Cone} ( \cO ^{\oplus r}_{\tot A} \to [... \to B^{\vee} \ot \Sym A^\vee \to \Sym A^\vee ]  ) 
\end{tikzcd} 
}
\] where $m: \Sym A \ot \Sym A \to \Sym A$ is the multiplication map.  
Let $p := p_A |_{\U}$ and denote by $\mathfrak a_{\U}$ or simply by $\mathfrak a$ the restriction of $\mathfrak a_{\tot A}$ to $\U$: \label{au}
\[ \mathfrak a_{\U} \in \mathbb H ^{-1} (\U, \mathrm{Cone} (\cO _{\U}^{\oplus r} \to  \wedge ^{-\bullet} p^*B ^{\vee} )) . \]

Consider the following vector of functions on $\U$
\[
\underline{W}  : = (ev_1 ^* w, ..., ev_r ^* w).
\]
By \cite[(3.13)]{MF}, the cochain complex
\begin{align} \label{eqn: cochain Wbar}
...   \longrightarrow \wedge ^2 p^*B^{\vee}   \longrightarrow  p^*B^{\vee} \oplus \cO _\U^{\oplus r} \xrightarrow{(\iota_{\beta} , \underline{W})} \cO _\U    \longrightarrow 0    \end{align}
is a realization of the object $\mathrm{Cone} (\cO _{\U}^{\oplus r} \to  \wedge ^{-\bullet} p^*B ^{\vee} )$.
We will seek for a cochain map realization of $\mathfrak a_\U$:
\[ \xymatrix{ ... \ar[r] &  0  \ar[r] \ar[d]  &  \cO_\U   \ar[rr]\ar[d]_{(\ka _{alg} , \diag )}              &  &        0        \ar[r] \ar[d] &    ...                                       \\ 
     ... \ar[r]  &     \wedge ^2 p^*B^{\vee}   \ar[r] & p^*B^{\vee} \oplus \cO _\U^{\oplus r} \ar[rr]_{ \ \ \ (\iota_{\beta} , \underline{W}) \ \ }  && \cO _\U \ar[r]   & ... }  \]  
     
   Note that  
$\diag$ above is a diagonal map since \eqref{eqn:diag} is also a diagonal map.
Hence, when a cochain map realization $(\ka_{alg}, \diag )$ of $\mathfrak a_\U$ exists,  the equality
$$\lan \alpha_{alg}, \beta \ran = - \sum_i ev_i ^* w $$ holds. 
Therefore we obtain a Koszul matrix factorization  \label{KosFact} $\{ \ka_{alg}, \kb \}$
 for $(\U, - \sum_i ev_i ^* w  )$ defined by
\[ \{ \ka_{alg}, \kb \} ^0  = \oplus _i \wedge^{2i} p^*B^{\vee},  \ \{\ka_{alg}, \kb \} ^1 = \oplus_i \wedge^{2i+1} p^*B^{\vee} \]
with differential $ \iota _{\beta} + \ka_{alg} \wedge $ (here $\iota _{\beta}$ denotes the contraction map by $\beta$).

\begin{Rmk}\label{rmk:local}
A cochain map realization of $\mathfrak a_\U$ is always possible \'etale locally since every quasi-coherent sheaf on any noetherian affine scheme has no higher cohomology.
\end{Rmk}

\subsection{A cdga resolution by the Thom-Sullivan functor}\label{sec:TS}
In the next two sections, we give our construction of a virtual factorization.  This uses the Godement and Thom-Sullivan functors.  We would like to point out that in the separated case, there is a simpler construction which we provide for the interested reader in \S\ref{sec: Dolbeault}.

Let $\mathbf{dg}_{\cO_\U}$ be the category of complexes of $\cO_\U$-modules and 
let $\mathbf{cdga}_{\cO_\U}$ be the category of commutative differential-graded $\cO_\U$-algebras.  
Given a category $\cC$ we denote by $\Delta\cC$ the category of cosimplicial objects in $\cC$ (i.e., functors from the simplex category to $\cC$).  \label{cosimplicial objects} 
For example, $\Delta\mathbf{dg}_{\cO_\U}$ denotes the category of cosimplicial complexes of $\cO_\U$-modules.

Let  \[ \Kos (\beta) :=  \Sym^{\rank B }  (p^*B^{\vee} \xrightarrow{\iota_{\beta}}  \cO _{\U} ) \] be the Koszul complex \label{Kos} associated to
the section $\beta$ of $ p^*B $ and consider it as an object in $\mathbf{cdga}_{\cO_\U}$.
Since $\U$ is a DM stack, there is a canonical cosimplicial Godement resolution $\GodKos (\beta)$ (resp.\ $\God(\cO_\U^{\oplus r}))$ of  $\Kos (\beta)$
(resp.\ $\cO_\U^{\oplus r}$)
which is an object of the category $\Delta\mathbf{dg}_{\cO_\U}$ (see \cite[Construction 1.31]{Thomason}).

We apply the Thom-Sullivan functor  $\Thb : \Delta\mathbf{dg}_{\cO_\U} \to \mathbf{dg}_{\cO_\U} $ (see Appendix~\ref{sec: TS}) to get a quasi-isomorphic  
 complex (see Proposition~\ref{prop: derham})
\begin{align}\label{eqn:CTG}  \mathrm{Cone} (\Thb\God  (\cO^{\oplus r}) \to \Thb\GodKos(\beta) ) \end{align} 
in which every component is $\Gamma$-acyclic by Proposition~\ref{prop:TSG acyclic}.  
Hence, $\mathfrak a_{\U}$ is realizable at the cochain level.
We write this cochain realization as a pair $\mathfrak a_\U(1) = (\ka , 1_\diag)$   with:
\begin{eqnarray}\label{eqn:an alpha} \ka  \in \Gamma (\U, \mathrm{Th}^{-1}\GodKos(\beta)) ,  \\ 
 1_\diag \in \Gamma (\U, \mathrm{Th}^{0}\God (\cO_{\U}^{\oplus r} ) ) , \nonumber 
 \end{eqnarray}
    which is unique up to homotopy. 
    
\begin{Prop} \label{prop; TS alphabeta=w}
Let $d = \iota _{\beta} + d_{dR}$ be the total differential on the cdga $\Thb\GodKos(\beta)$ 
where $\iota_{\beta}$ is the Godement-Koszul differential, $d_{dR}$ is the Thom-Sullivan differential (which is just the algebraic de Rham differential). Then
\[
d(\ka) = - \sum ev_i^* w \cdot 1_T.
\]
where $1_T$ is the unit of $\Thb\GodKos(\beta)$ and $\cdot$ denotes the $\cO_\U$-module action.
\end{Prop}
\begin{proof}    

First  note that $\diag$ is completely unique as it remains unchanged by any homotopy.  Hence, using \eqref{eqn:diag} it is realized by 
\begin{equation} \label{eq: diag}
1_{\diag} = 1_T^{\oplus r}.
\end{equation}

Now since $\cO_\U$ is quasi-isomorphic to $\Thb\God \cO_\U$ which is $\Gamma$-acyclic, there is a homomorphism of cdg $\cO_\U$-algebras at the chain level 
    \[
    \epsilon: \cO_{\U} \to \Thb\God \cO_\U.
    \] 
     As there are no possible homotopies this is unique and in fact, it is determined by $\epsilon(1) = 1_T$.  Moreover, it is the structure homomorphism which makes $\Thb\God \cO_\U$ a $\cO_U$-module.
Hence viewing any $f: \cO_\U \to \cO_\U$ simultaneously as an element of $\cO_\U$, we have
\begin{equation} \label{eq: epsilonout}
\epsilon(f) = \Thb\God(f) (1_T).
\end{equation}

Let $d' = \iota_{\beta} +  d_{dR} + \Thb\God(\underline W)$ be the total differential of the complex \eqref{eqn:CTG} where 
$\Thb\God(\underline W)$ is the term in the differential coming from the fact that it is a cone (see \eqref{eqn: cochain Wbar}).  Then,
\small
\begin{align*}
0 & = d'(\mathfrak a_{\U}(1)) & \text{since }\mathfrak a_{\U} \text{ is a cochain map} \\ 
& = d'(\ka, 1_T^{\oplus r}) &  \text{by \eqref{eq: diag}}   \\
& = (\iota_{\beta} +  d_{dR} + \Thb\God(\underline W))(\ka, 1_T^{\oplus r})& \text{by definition} \\
& = (\iota_{\beta} +  d_{dR})(\ka + 1_T^{\oplus r}) +  \Thb\God(\underline W)(1_T^{\oplus r}) 
&  \text{as }\Thb\God(\underline W)|_{\Thb\GodKos\beta} = 0 \\
& = (\iota_{\beta} +  d_{dR})(\ka) + d_{dR}(1_T^{\oplus r}) +  \Thb\God(\underline W)(1_T^{\oplus r}) 
&\text{as }\iota_\beta|_{\cO_\U^{\oplus r}} = 0 \\
& = (\iota_{\beta} +  d_{dR})(\ka) +  \Thb\God(\underline W)(1_T^{\oplus r})& \text{since }\epsilon\text{ is a cochain map} \\
& = (\iota_{\beta} +  d_{dR})(\ka) +  \epsilon( \sum ev_i^* w) & \text{by \eqref{eq: epsilonout}} \\
& = (\iota_{\beta} +  d_{dR})(\ka) +  \sum ev_i^* w  \cdot 1_T & \text{ since }\epsilon \text{ is the structure map}.
\end{align*}
\normalsize
This completes the proof as $(\iota_{\beta} +  d_{dR})$ is the differential on $\Thb\GodKos(\beta)$.
\end{proof}

\subsection{Virtual factorizations}

\subsubsection{Factorizations associated to cdgas}\label{sec:CDGfact}
\label{tX}
Let $\tX$ be a DM stack and $$(\mathbb{A}= \oplus _{i\in \ZZ}  \mathbb{A}^{i}, d)$$ be a sheaf of cdgas over $\cO_{\tX}$ i.e\ a $\ZZ$-graded associative $\cO_{\tX}$-algebra with 
unit $1$ and a degree $+1$ differential $d$. 
Let $\tW \in \Gamma (\tX, \cO_{\tX})$ and 
$a \in \Gamma (\tX, \mathbb{A}^{-1})$ such that $da = \tW\cdot 1$. Consider $d_a$ which is obtained from summing the usual differential with multiplication by $a$,
$$d_a := d + a \cdot.$$
The graded Leibnitz rule implies  that $d_a ^2 = \tW \cdot$. 
Hence we obtain a $\cO_{\tX}$-module factorization  $(\mathbb{A} , d_a )$ of $\tW$.

\subsubsection{Homotopies}\label{ind:a} 
Let $a '$ be another choice such that $da' = \tW \cdot 1$ and $a' - a = d h$ for some degree $(-2)$ element $h$ of  $\Gamma (\tX, \mathbb{A})$. 
It is straightforward to check that the factorizations 
$(\mathbb{A} , d_a) $ and $(\mathbb{A}, d_{a'})$ are isomorphic under the multiplication map by $\exp (-h)$.

\subsubsection{Homomorphisms}\label{sub:homo of curved diff}
Let $(\mathbb{A}', d')$ be another sheaf of cdgas over $\cO_{\tX}$ and let $\varphi:  \mathbb{A} \to \mathbb{A}'$ be a cdga $\cO_{\tX}$-homomorphism of degree $0$.
Then it clearly induces a degree $0$ homomorphism $(\mathbb{A} , d_a) \to (\mathbb{A}' , d'_{\varphi(a)}) $ of $\cO_{\tX}$-modules factorizations of $\tW$.

\subsubsection{Thom-Sullivan Koszul factorizations}\label{sub:Kos-bar}
Let  $\tX$ be a DM stack and let $\tW$ be a regular function on $\tX$.
Consider  a locally free coherent sheaf $F$ on $\tX$ with $\sigma \in \Gamma (\tX, F)$.
Then we have $$\Thb\GodKos (\sigma) $$ which is a lower bounded $\Gamma$-acyclic cdga over $\cO_{\tX}$ (see Proposition~\ref{prop:TSG acyclic}).
Let $\tau \in \Gamma (\tX, \mathrm{Th}^{-1}\GodKos (\sigma))$ be an element satisfying
\begin{eqnarray} \label{eq:stw}  (\iota_{\sigma}  +d_{dR}  ) (\tau ) = \tW \cdot 1 , \end{eqnarray} 
where $1$ is the unit element of $\Gamma (\tX, \cO_{\tX}) \subset \Gamma (\tX, \mathrm{Th}^{0}\GodKos (\sigma)) $.
 As an application of \S  \ref{sec:CDGfact} we have a factorization 
 \[
 \mathbb K(F, \tau, \sigma) := (\Thb\GodKos (\sigma) , \iota_{\sigma} + d_{dR}+  \tau \wedge ).
 \] 
We view it as an object of the co-derived category of factorizations of $\tW$ whose components are sheaves of $\cO_{\tX}$-modules (see \cite{BDFIK12} for a definition in this generality).  We denote this category by $D(\tX, \tW)$. \label{derived}
 \label{derived category}

\begin{Def}
A  factorization 
of the form $ \mathbb K(F, \tau, \sigma)$ is called the {\em Thom-Sullivan Koszul (TK) factorization} \label{TK} associated to $\tau, \sigma$.  To reduce notation, we often simply write $ \mathbb K( \tau, \sigma)$.
\end{Def}

\begin{Rmk} \label{rmk: flat}
TK factorizations always have $\cO_{\tX}$-flat components by \cite[Lemma 5.2.11]{HS} and the flatness of $\God [i] \cO_{\tX}$ for every $i$.
\end{Rmk}

\subsubsection{Virtual factorizations}\label{sub:virtual factorization}
We will now define virtual factorizations.
For clarity, we first recall the procedure above.

For  fixed $g, r, d$, GLSM data, and $S\to \fB_{\Gamma}$, we construct $([\cA \to \cB], rest : \cA  \to \cV |_{\sG}, \U )$ as in Proposition \ref{prop:U}. 
Then, we define an element $\mathfrak a_\U \in \HH ^{-1} (\U, \Cone (\cO_{\U}^{\oplus r} \to \wedge ^{-\bullet} p^*B^{\vee}  ))$ as in \S \ref{sec:cochain}.
We find a degree $-1$ global section $\ka$ of $\ThbGod (\wedge^{-\bullet} p^*B^{\vee} )$ which together with the
diagonal map $\cO_{\U} \to \cO_{\U}^{\oplus r}$ represents the class $\mathfrak a_{\U}$; see \eqref{eqn:an alpha}.
By Proposition~\ref{prop; TS alphabeta=w}, we may apply the Thom-Sullivan Koszul construction to this data.
This yields a factorization for the superpotential $ - \sum ev_i^* w$ on $\U$: 
\begin{equation}\label{eq:fund factor}
 \mathbb{K}_{S\to \fB_{\kG}} := \mathbb K(p^*B, \alpha, \beta) = (\Thb\GodKos (\beta) , \iota_{\beta} + d_{dR}+  \alpha \wedge ),
 \end{equation}
 which is naturally isomorphic to 
\[ (\Kos (\beta) \ot_{\cO _{\U}} \Thb\God \cO_{\U} , \iota_{\beta} + d_{dR}+  \alpha \wedge )  \] by  Lemma~\ref{lem: projGod}.
We call the factorization $\mathbb{K}_{S\to \fB_{\kG}}$ (also written $\mathbb K_S$), 
a {\em virtual factorization} of $LG(\cX)\ti _{\fB_{\kG}} S$.  
For the identity map $S\xrightarrow{=} \fB_{\kG}$ we write $\mathbb{K}_{g, r, d}$ or $\mathbb{K}_{\fB_{\Gamma}}$ instead of $\mathbb{K}_{S}$.

After fixing the data $([\cA \to \cB], rest_{\cA} : \cA  \to \cV |_{\sG}, \U)$, 
as above the virtual factorization is unique up to isomorphism by \S\ref{ind:a}.  
When we need to include the class of $\mathfrak a_\U$ in the notation, we sometimes write $\mathbb K(\mathfrak a_\U, \beta)  := \mathbb K(p^*B, \alpha, \beta)$.

\begin{Rmk}
While a choice $\mathbb{K}_{g, r, d}$ of virtual factorizations lies in the non-canonical LG space $(\U_{g, r, d},- \sum ev_i^* w)$, 
we will show that nevertheless there is an associated cohomology class in 
$\HH^*(I\cX, (\Omega^\bullet_{I\cX}, - dw))^{\ot r} \ot H^*(\overline{M}_{g,r})$ which is independent of all choices, see Theorem~\ref{thm:cind}.
\label{Mgr}
\label{cotangent}
\end{Rmk}

\begin{Rmk}
The virtual factorization $\mathbb{K}_{g, r, d}$ described above has been called the fundamental factorization in the previous literature.
\end{Rmk}

\subsubsection{Pullbacks}\label{sec:coderived}
With the setup as in \S \ref{sub:Kos-bar}, let $f : \tY \to \tX$ be a morphism of smooth DM stacks.  In this section we argue that the formation of $TK$-factorizations ``commutes'' with pullback along $f$.

Let $\cF$ be a $\cO_{\tY}$-module.  By \cite[(1.25)]{Thomason}, there is a natural $\cO_{\tX}$-module homomorphism
\[
 \God f_*\cF \to f_*\God\cF.
\] 
Let $\cF = f^*\cE$ and use the  adjunction $f^* \dashv f_* $ to obtain a natural transformation 
\[
nat_1^f: f^* \God \to \God f^*.
\]

Now, let $\cF_i$ be a collection of $\cO_\tX$-modules.  For any \'etale open $Y \to \tY$, there is a natural map
 \begin{align}
 (f^{-1} \prod_{i \in I} \cF_i)(Y) \otimes_{f^{-1} \cO_\tX(Y)}  \cO_\tY(Y) & \to \prod_{i \in I} (f^{-1} \cF_i(Y) \otimes_{f^{-1} \cO_\tX(Y)}  \cO_\tY(Y))  \label{eq: nat}\\
(a_i) \ot s & \mapsto (a_i \ot s) \notag.
 \end{align}
For an input sheaf $\mathcal F$, we can consider the case  $\cF_i = \God[i] \mathcal F \otimes_{{\bf k}} \Omega[i]$.  This yields a natural transformation
 \begin{align*}
nat_2^f:  f^*\ThbGod & \to \Thb(f^*\God ).
 \end{align*}
 We let
 \[
 \zzeta_f : = \Thb (nat_1^f) \circ nat_2^f 
 \]
 
Now, given $\tau \in \mathrm{Th}^{-1}\GodKos(\sigma)$ satisfying \eqref{eq:stw}, the natural chain map 
\begin{equation} \label{eq: zzeta} \zzeta_f(\Kos(\sigma)) = \Thb (nat_1^f(\Kos(\sigma)) \circ nat_2^f(\Kos(\sigma)) \end{equation} 
 induces a morphism of factorizations
\begin{equation} 
\widetilde{\zzeta}_f : f^*\mathbb K (\tau, \sigma) \to \mathbb K (\zzeta_f(f^*\tau), f^*\sigma).
\label{eq: pullback TK}
\end{equation}

\begin{Lemma} \label{lem: TK pullback}
  Assume that, locally on $\tX$, there exists $\overline \tau \in \Kos^{-1}(\sigma)$ such that $\tau$ is homotopic to $\Thb(\iota)(\overline \tau)$ 
  where $\Thb(\iota)$ is the map from \eqref{eq: Thi}.
  Then
  \begin{enumerate}
   \item The  natural map \eqref{eq: pullback TK} is a co-quasi-isomorphism.
   
   \item On an atlas $g: X\to \tX$ of $\tX$ where $\bar{\tau}$ exists globally, we have natural co-quasi-isomorphisms
    \[\xymatrix{ 
     \{\overline \tau, \sigma \} \ar[r] \ar@/_1.2pc/[rr] &  g^* \mathbb K (\tau, \sigma)  \ar[r] &  \mathbb K ( \zzeta _g (g^*\tau) , g^*\sigma ) .
    } \] 
  \end{enumerate}
  
  \end{Lemma}

\begin{proof}
By the assumption on $\tau$, locally we get a commuting diagram in $D(\tY, f^*\tW)$:
\begin{equation}\label{diag: pullbacks are co-quasi}
\begin{tikzcd}
f^*( \{\overline \tau, \sigma \}   \otimes_{\cO_\tX} \ThbGod\cO_{\tX})  \ar[d, "a"]  \ar[r, "b"] & f^*\mathbb K (\Thb(\iota)(\overline \tau), \sigma)  \ar[r, "d"]  \ar[d]&  f^*\mathbb K (\tau, \sigma) \ar[d, "{\widetilde{\zzeta}_f}"]  \\
 \{f^*\overline \tau, f^*\sigma \}  \otimes_{\cO_\tY} \ThbGod\cO_{\tY} \ar[r, "c"] & \mathbb K (\Thb(\iota)(f^* \overline \tau), f^*\sigma)  \ar[r, "e"]  &\mathbb K (\zzeta _f(f^*\tau), f^*\sigma)
\end{tikzcd}
\end{equation}
where
\begin{itemize}
\item $a$ is the co-quasi-isomorphism obtained from plugging;
\[
\cR = f^*\Thb \God \cO_{\tX} \text{ and } \cR = \Thb \God \cO_{\tY};
\]
into \eqref{eq: fact res} and using Lemma~\ref{lem: Th is qi};
\item $b, c$ are the isomorphisms induced by Lemma~\ref{lem: projGod}; and
\item  $d,e$ are the isomorphisms from \S\ref{ind:a}.
\end{itemize}
Hence $\widetilde{\zzeta}_f = e \circ c  \circ a \circ  b^{-1}  \circ  d^{-1}$ is a co-quasi-isomorphism, which proves (1).
Note that (2) is \eqref{diag: pullbacks are co-quasi} with $f = g$.
\end{proof}

\begin{Cor}\label{cor: pullback of vir fact}
Let $\mathbb K(\mathfrak a_\U, \beta)$ be a virtual factorization on $\U$ and $f: \tX \to \U$ be a morphism of DM stacks.  Then $f^*\mathbb K(\mathfrak a_\U, \beta)$ is co-quasi-isomorphic to $\mathbb K(f^*\mathfrak a_\U, f^*\beta)$.
\end{Cor}

\begin{proof}
By Remark \ref{rmk:local}, for an affine atlas $Y \to \U$ we get a diagram
\begin{equation} 
\begin{tikzcd}
\cO_Y[1] \ar[d, swap, "{(\alpha_{alg}, \diag)}"] \ar[drr, "{\mathfrak a_\U|_{Y}}"]& & \\
\Cone (\cO_Y^{\oplus r} \to \Kos(\beta)) \ar[rr, " \Thb(\iota) \oplus \Thb(\iota)"]& & \Cone (\ThbGod\cO_Y^{\oplus r} \to \Th(\beta)),
\end{tikzcd}
\label{eq: diagram perfect}
\end{equation}
where we abuse notation $\Thb(\iota)$ obviously.
This guarantees any chain level representative $\mathfrak a_\U|_{Y} $ is homotopic to $(\Thb(\iota)\circ \alpha _{alg}, \Thb(\iota)\circ \diag )$.
 Hence, the assumption of Lemma~\ref{lem: TK pullback} is satisfied.
 \end{proof}

This is related to the following proposition.

\begin{Prop} \label{prop: perfect} For an affine atlas $Y \to \U$, let $(\alpha _{alg}, \diag)$ be a realization of 
$\mathfrak a _{\U}|_Y$. The natural map $\{ \alpha _{alg} , \beta |_Y \} \to \mathbb K( \alpha, \beta) |_{Y} $ is a quasi-isomorphism.
In particular, any virtual factorization $\mathbb K( \alpha, \beta)$ is perfect (see Definition~\ref{def: perfect}).
\end{Prop}
\begin{proof}
By \eqref{eq: diagram perfect} $\alpha$ is locally homotopic to $\Thb(\iota)( \alpha_{alg})$.  
Therefore,
\begin{align*}
\{ \alpha _{alg}, \beta |_Y \} & \cong \{ \alpha _{alg}, \beta |_Y \} \ot_{\cO_Y} \Thb \God \cO_Y & \text{ by  \eqref{eq: fact res} and Lemma~\ref{lem: Th is qi}} \\
& \cong \mathbb K (\Thb(\iota)(\alpha _{alg}), \beta |_Y  ) & \text{ by Lemma~\ref{lem: projGod}} \\
& \cong \mathbb K (\alpha, \beta)|_Y &   \text{ from Corollary \ref{cor: pullback of vir fact}.} 
\end{align*}
\end{proof}

We now prove a functorial property of $\zzeta_f$.  Abusing notation, we often write $f^*\tau$ instead of $\zzeta _f (f^*\tau)$. \label{notation: pullback tau}

\begin{Lemma}\label{lemma: functor zzeta}  Consider two maps between smooth DM stacks, $f: \tY \to \tX$ and $g: \tZ \to \tY$.  The following functors agree:
\[
\zzeta _g \circ g^* (\zzeta _f ) = \zzeta _{f\circ g}.
\]
\end{Lemma}
\begin{proof}
We have the following commutative diagram
\[
\begin{tikzcd}
g^*f^* \Thb\God \ar[d, swap, "g^* nat_2^f \God"] \ar[dr, "nat_2^{g \circ f} \God"] \ar[dd, bend right = 100, swap, "{g^*(\zzeta_f)}"] \ar[ddrr, bend left = 80,  "{\zzeta_{g \circ f}}"]& & \\
g^*\Thb f^* \God \ar[d, swap, "g^* \Thb nat_1^f"] \ar[r, "nat_2^g f^* \God "] & \Thb g^*f^* \God \ar[dr, "\Thb nat_1^{g \circ f}"] \ar[d, swap, "\Thb g^* nat_1^f"] & \\
g^* \Thb\God f^* \ar[r, "nat_2^g  \God f^* "] \ar[rr, bend right = 25, swap, "{\zzeta_g}"] & \Thb g^*\God f^* \ar[r, swap, "\Thb nat_1^g f^*"] & \Thb\God g^*f^*  \\
\end{tikzcd}
\]
The upper left triangle commutes since each map is effectively moving a parantheses (which keeps track of the indexing).
The square commutes since $nat_2^g: g^*\Thb \Rightarrow \Thb g^*$ is a natural transformation. 
Commutativity of the bottom right triangle reduces to showing $nat_1^g f^* \circ nat^f_1 = nat_1^{g \circ f}$.  Since $nat_1$ comes from the base change map in topology this reduces to the corresponding standard property of the base change map.
\end{proof}

\subsubsection{Tensor products}  
Let us now study the tensor product of 
two TK factorizations  $ \mathbb K(F_i, \tau_i, \sigma_i)$ on $\tX$ for $\tW_i$, $i=1, 2$.
By Lemma \ref{lem: projGod}  we have cdga homomorphisms 
\begin{eqnarray*} 
\ThbGod\Kos (\sigma _1) \otimes \ThbGod\Kos( \sigma_2)  & \cong & \Kos (\sigma _1) \otimes \Kos (\sigma _2) \otimes  \ThbGod \cO \otimes \ThbGod \cO \\
& \cong & \Kos (\sigma _1 \oplus \sigma _2) \otimes  \ThbGod \cO \otimes \ThbGod \cO \\
& \xrightarrow{qis} & \Kos (\sigma _1 \oplus \sigma _2) \otimes  \ThbGod \cO  \\
& \cong &  \ThbGod\Kos (\sigma _1 \oplus \sigma_2).
\end{eqnarray*}
The left cdga has the degree $-1$ section $\tau _1 \ot 1 + 1 \ot \tau _2$. The above sequence of maps gives rise to a section 
of   ${\rm Th}^{-1}\God\Kos (\sigma _1 \oplus \sigma_2)$ which coincides with 
the composition \[     \cO_{\tX} \xrightarrow{\tau _1 \oplus \tau_2  }
   \mathrm{Th}^{-1}\GodKos (\sigma _1) \oplus \mathrm{Th}^{-1}\GodKos (\sigma_2)  \hookrightarrow  \mathrm{Th}^{-1}\God \Kos (\sigma _1 \oplus \sigma_2).  \]
   This will be denoted  by $\tau _1 \oplus \tau_2 $ as a slight abuse of notation.
Assume that after pulling back to an atlas  $X$ of $\tX$  each of the $\tau _i$ are homotopic respectively to ${\rm Th}^{\bullet} (\iota ) (\bar{\tau}_i)$
for some section $\bar{\tau}_i$ of $F_i^{\vee}|_{X}$. 
Then the map induced from the composite map 
\begin{equation} \label{eqn: inner tensor of TK}
\mathbb K(F_1, \tau_1, \sigma_1) \otimes \mathbb K(F_2, \tau_2, \sigma_2) 
\xrightarrow{cqis} \mathbb K(F_1 \oplus F_2, \tau _1 \oplus \tau_2, \sigma _1 \oplus  \sigma_2) .
\end{equation} is a  co-quasi-isomorphism between factorizations for $\tW_1 + \tW_2$ by Lemma \ref{lem: TK pullback} (2).

Now consider two TK factorizations  $ \mathbb K(F_i, \tau_i, \sigma_i)$ on possibly different DM stacks $\tX_i$ for $i=1, 2$.  Assume the existence of $\bar{\tau}_i$ as above.
Let $p_i : \tX _1 \ti \tX_2 \to \tX_i$ be the projections and set $\tau _1 \boxplus \tau_2  := p_1^* \tau _1 \oplus p_2^*\tau_2$, $\sigma _1 \boxplus  \sigma_2 : =  p_1^* \sigma _1 \oplus  p_2^*\sigma_2$,
and $\tW_1 \boxplus \tW_2 := p_1^*\tW_1 + p_2^*\tW_2 $.
Then by \eqref{eqn: inner tensor of TK} we have 
\begin{eqnarray} \label{eqn: outer tensor of TK}
\mathbb K(F_1, \tau_1, \sigma_1) \boxtimes \mathbb K(F_2, \tau_2, \sigma_2) & := & p_1^* K(F_1, \tau_1, \sigma_1) \otimes p_2^* \mathbb K(F_2, \tau_2, \sigma_2) \nonumber \\
&   \xrightarrow{cqis} & \mathbb K(p_1^*F_1 \oplus p_2^* F_2, p_1^* \tau _1 \oplus p_2^*\tau_2,  p_1^* \sigma _1 \oplus  p_2^*\sigma_2)    \nonumber \\
\end{eqnarray}
as factorizations of $\tW_1 \boxplus \tW_2$ on $\tX_1 \ti \tX_2$.

\subsubsection{Support} \label{ind:support}
Let $LG(Z(dw))$ be the moduli space of $\nu$-stable LG  quasimaps to the critical locus $Z(dw)$ of $w$.
This is a proper DM stack by Theorem~\ref{FJR proper}.
We let $\Z : = LG(Z(dw)) \ti _{\fB_{\Gamma}} S$ \label{Z}  and note that by Proposition \ref{prop: perfect} 
and \cite[Proposition 3.6.1]{MF}, 
$\mathbb K (\alpha, \beta )$ restricted to $\U - \Z$ is locally homotopic to zero and hence coderived quasi-isomorphic to zero, i.e., the virtual factorization $\mathbb K (\alpha, \beta)$ is supported on
the closed substack $\Z$,  which is a proper DM stack whenever $S \to \fB_{\kG}$ is proper.

\subsubsection{The Convex Case}\label{subsubsec: convex} 
For some GLSM input data, it is possible to find a global cochain realization $(\ka _{alg}, 1_{\diag})$  of $\mathfrak a _\U (1)$.
In this case, we obtain a Koszul matrix factorization $\{ \ka _{alg}, \kb \}$ globally, which is obviously
quasi-isomorphic to  the virtual factorization
$\{ \ka_{alg} , \kb \} \ot _{\cO_{\U}} \ThbGod \cO _{\U}$.
When the GLSM input data is convex and $\nu = \infty$ such a Koszul matrix factorization
 $\{ \ka_{alg} , \kb \}$ was constructed in \cite{MF} (for a suitable choice of $\U$).

\subsection{Dolbeault construction:\ the separated case} \label{sec: Dolbeault}
\begin{disclaimer}
This section is not used in the rest of the paper.  We include it only to illustrate that there is a simpler, equivalent construction in the separated case.
\end{disclaimer}

We now consider the case where $\U$ is separated.  With this assumption, we can provide an alternative construction using smooth forms.

Let $\U^{an}$ be the analytification of $\U$ and $\text{\'Et}(\U^{an})$ be the site whose morphisms are local analytic isomorphisms $X \to \U^{an}$ where $X$ is an analytic space; see for example \cite{J.Hall}.

\begin{Lemma}
Let $\cC^\infty_{\U^{an}}$ be the sheaf of smooth functions on $\U^{an}$.  Every sheaf $\cF$ of $\cC^\infty_{\U^{an}}$-modules on $\emph{\'Et}(\U^{an})$ is $\Gamma$-acyclic.
\end{Lemma}

\begin{proof}
Since $\U$ is a separated DM stack, there is a separated, coarse moduli space $\underline{\U}$ of $\U$. Hence the underlying analytic space $\underline{\U}^{an}$ of
$\U ^{an}$ is Hausdorff and paracompact. We recall the natural map $\pi : \U^{an} \to \underline{\U}^{an}$ induces an exact functor $\pi_* $ from the category of sheaves 
of abelian groups on $\U ^{an}$ to  that on $ \underline{\U}^{an}$ (see, e.g., \cite[Theorem 2.4]{J.Hall}). 
Therefore it is enough to show that $\pi_*\cF$ is a fine sheaf 
on the Hausdorff and paracompact topological space  $ \underline{\U}^{an}$. 

 A standard argument (see, e.g., the proof of \cite[Theorem 1.11]{Warner}) provides the existence of a smooth partition of unity subordinate to any cover of $\underline{\U}^{an}$
since $\underline{\U}^{an}$ is locally a quotient of a complex domain by a finite group.  Since $\pi _* \cF$ is a $\cC^{\infty}_{\U ^{an}}$-module, it follows that it is fine.
 \end{proof}

In what follows we abuse notation, denoting all the analytifications of the algebraic constructions from the previous sections the same way.
Let $\Dol$ denote the Dolbeault complex of smooth forms on $\U^{an}$.  

The previous lemma demonstrates that $\Kos(\beta) \otimes_{\cO_{\U^{an}}} \Dol$ is a $\Gamma$-acyclic cdga 
which is quasi-isomorphic to $\Kos(\beta)$. 
Following the construction in \S\ref{sec:TS}, we obtain an element $\alpha_{Dol} \in  \Gamma(\U^{an},  \bigoplus_{p+q = -1} \Kos^p(\beta) \otimes_{\cO_{\U^{an}}} \sA^{(0,q)}_{\bar \partial})$ and a corresponding Dolbeault-Koszul factorization 
\begin{align*}  \mathrm{D}\Kos(\alpha_{Dol}, \beta) := (  K(\beta ) \ot_{\cO_{\U ^{an}}} \Dol     , \iota _{\beta} + \bar{\partial} + \alpha _{Dol}\wedge ) \in D(\U ^{an}, W). \end{align*}

\begin{Prop}
The Dolbeault-Koszul factorization $\mathrm{D}\Kos(\alpha_{Dol}, \beta)$ is co-quasi-isomorphic to the \lq\lq analytification"
$( K(\beta ) \ot_{\cO_{\U ^{an}}} \ThbGod\cO_{\U^{an}}  , \iota _{\beta} + d_{dR} + \alpha \wedge ) $ of the  virtual factorization $\mathbb K(\alpha, \beta)$.
\end{Prop}
\begin{proof}
Since the natural maps $ \cO_{\U ^{an}} \to \Dol$ and $ \cO_{\U ^{an}} \to \ThbGod\cO_{\U^{an}}$ are quasi-isomorphisms,
the natural maps $\mathbb K (\alpha, \beta) \to \mathbb K (\alpha, \beta)  \otimes_{\cO_{\U^{an}}} \Dol$ and
$ \mathrm{D}\Kos(\alpha_{Dol}, \beta) \to \mathrm{D}\Kos(\alpha_{Dol}, \beta) \otimes_{ \cO_{ \U ^{an}}} \ThbGod\cO_{\U^{an}}$ between 
$\cO_{\U^{an}}$-factorizations for $W$ are co-quasi-isomorphisms. On the other hand, since $(\alpha _{Dol}, \diag )$ and $(\alpha , \diag)$ represent the same 
class $\mathfrak a_\U$, 
we have an isomorphism 
$\mathbb K (\alpha, \beta)  \otimes_{\cO_{\U^{an}}} \Dol \cong  \mathrm{D}\Kos(\alpha_{Dol}, \beta) \otimes_{ \cO_{ \U ^{an}}}  \ThbGod\cO_{\U^{an}}$ by \S \ref{ind:a}.
\end{proof}

\subsection{Independence}

In this section, we show that the virtual factorization $\mathbb K_{g,r,d}$ 
is independent of various choices in a suitable sense.
We begin with the following general situation.

Consider a commuting diagram of LG models of {\em smooth} DM stacks $\tX, \tX '$, $\tT$:
\begin{equation} 
   \xymatrix{   (\tX , \tt{W} )  \ar[r]^{\iota}    \ar[rd]_{e} 
   &  (\tX ', \tW ' ) \ar[d]^{e' } \\ 
     &     (\tT, \tt{w}) , }
      \label{LG diagram} 
      \end{equation} 
      i.e., $\tW  = \iota ^* \tW'$, $\tW = e^* {\tt w}$,  $\tW' = e'^* {\tt w}$, and $e = e' \circ \iota$.
      We allow the case where $\tT$ is empty,  i.e., where there is no $\tT, {\tt w}, e, e'$.

Let $F$ and $F'$ be vector bundles on DM  stacks $\tX$, $\tX '$ respectively.
Let $\sigma $, $\sigma '$ be global sections of $F^{\vee}$ and $F'^{\vee}$, respectively.
Let $\tau $, $\tau '$ be degree $-1$ global sections of $\Th (\sigma ) $, $\Th (\sigma ') $, respectively such that $(\iota _{\sigma } + d_{dR}) (\tau ) = \tW$ and
 $(\iota _{\sigma '} + d_{dR}) (\tau ' ) = \tW '$.

         Assume that, locally on $\tX ' $, there exists $\overline \tau ' \in \Kos^{-1}(\sigma ')$ such that 
$\tau '$ is homotopic to $\Thb(\iota)(\overline \tau ')$.
We consider the following two cases for $\iota$.  In these cases, we will demonstrate that there is a natural isomorphism
 $\mathbb{K}(F', \tau '  , \sigma ')  \cong \iota_*\mathbb{K}(F , \tau , \sigma  )$ in $D(\tX ', \tW ' )$. 
 \medskip
 
 {\bf Case (1)} the map $\iota$ is an {\em open} immersion. 
 Suppose that $\iota ^* F' = F, \iota^*\sigma ' = \sigma $, $\iota ^* \tau ' = \tau$, and the support of $\mathbb{K}(F', \tau '  , \sigma ' )$ lies in $\tX$. 
 Then 
 $\mathbb{K}(F', \tau '  , \sigma ' ) \xrightarrow{\cong} \iota _*\iota^* \mathbb{K}(F', \tau '  , \sigma ' ) 
 \xrightarrow{\cong} \iota _*  \mathbb{K}(F, \tau , \sigma  )  $ in $D(\tX ', \tW ')$. The first isomorphism follows from the support condition 
 and the second isomorphism is by Lemma \ref{lem: TK pullback}.

{\bf  Case (2) } the map $\iota$ is a {\em closed} immersion with the following conditions.
There is an epimorphism $\phi: F' \to Q$ with a vector bundle $Q$ on $\tX '$ such that:
\begin{itemize}
\item The smooth DM stack $\tX$ is the zero locus of $\phi\circ \sigma$ via $\iota$; 
\item $\phi \circ \sigma$ is a regular section of $Q$ (i.e., the codimension of $\tX$ in $\tX '$ is $\rank Q$);
\item $F$ fits into a short exact sequence $$0\to F \to F'|_{\tX} \xrightarrow{\phi |_X} Q|_{\tX}\to 0 $$ of vector bundles on $\tX$.
\end{itemize}
Note that the restricted section $\sigma '|_{\tX} : \cO_{\tX} \to F'|_{\tX}$ is factored by a homomorphism   $\cO_{\tX} \to F$.
We require that it is exactly $\sigma $.
    Also we assume that the composition   
     $$\cO_{\tX} \xrightarrow{ \zzeta (\tau'|_{\tX}) } \ThbGod (\wedge^{-\bullet} F'^{\vee}|_{\tX}, \iota _{\sigma '} |_{\tX})  \to \ThbGod (\wedge^{-\bullet} F^{\vee}, \iota _{\sigma } )  $$
     becomes $\tau $; for the definition of $\zzeta$, see \eqref{eq: zzeta}.
         Then by \S \ref{sub:homo of curved diff} we have a natural map 
  \begin{eqnarray}\label{eqn:restriction fact}  \mathbb{K}( F'|_{\tX}, \zzeta (\tau' |_{\tX})  , \sigma' |_{\tX} ) \to \mathbb{K}(F, \tau  , \sigma  ) \end{eqnarray}
         in $D(\tX, \tW )$.

\begin{Prop}\label{prop:PV} 
The natural map $$\psi:  \mathbb{K}(F', \tau '  , \sigma ' ) \to \iota_*\mathbb{K}(F, \tau  , \sigma  ) $$ obtained by the adjunction of \eqref{eq: pullback TK} and 
\eqref{eqn:restriction fact} is an isomorphism in $D( \tX ', \tW ' )$.
\end{Prop} 

\begin{proof} 
 It is enough to show locally that  the map $\psi $ is a quasi-isomorphism (see, e.g., \cite[Corollary 2.2.7]{MF}). 
Hence, by Lemma \ref{lem: TK pullback} (2)  
we may assume that $\tau '$ is a section of $F'^{\vee}$.
Now $\psi$ is a quasi-isomorphism
due to \cite[Proposition 4.3.1]{PV:MF}. 
\end{proof}

\begin{Def}\label{def:equiv}  
When two TK factorizations $\mathbb{K}(F, \tau  , \sigma ) $ on $(\tX, \tt{W})$
and $\mathbb{K}(F', \tau ' , \sigma ') $ on $(\tX ', \tt{W}' )$ are in the situation of either Case (1) or Case (2) above, we say that they are {\em related} (with respect to $ (\tT, \tt{w})$).
Two TK factorizations $\mathbb K_1$ and $\mathbb K_2$ are said to be {\em transitively related}, written $\mathbb K_1 \sim \mathbb K_2$, 
(with respect to $(\tT, \tt w)$) if they lie in the same equivalence class for the equivalence relation generated by related TK factorizations (with respect to $(\tT, \tt w)$). 
\end{Def}

\begin{Rmk}
Two factorizations on the same $\tX$ but with two different superpotentials can, in fact, be in the same equivalence class.  
For example, it is easy to concoct trivial examples where we rescale the potential.  We will explore less trivial cases where two TK factorizations are related below.
\end{Rmk}

\begin{Def}\label{def:reduction}
The TK factorization $\mathbb{K}(F, \tau , \sigma  )$ is called the \emph{reduction} of $ \mathbb{K}(F ', \tau '  , \sigma ' ) $ by $\phi$.
\end{Def}

\medskip

The following theorem is proven in the proceeding sequence of subsections.

\begin{Thm}\label{thm:ind} 
Any two choices of  virtual factorizations of $LG(\cX)\ti _{\fB_{\kG}} S$
are transitively related with respect to $ ([(IV^{ss})^r/G^r] \ti \overline{M}_{g, r}, - \boxplus_i w \boxplus 0)$.
\end{Thm}

\subsubsection{Choices of $\ka$}\label{ind:alpha} 
Let $(\ka ' , \diag )$ be another choice  realizing $\mathfrak a_{\U}$. By the uniqueness of $\ka$ up to homotopy,
$\ka ' - \ka = (\iota_{\beta} + d_{dR}) (h )$ for some degree $(-2)$ element $h$ of $\ThbGod (\wedge^{-\bullet} p^*B^{\vee} )$.
Now as seen in \S\ref{ind:a},
$\mathbb K (p^*B, \ka , \kb )$ and $\mathbb K (p^*B, \ka' , \kb )$ are isomorphic under the multiplication map by $\exp (-h)$.

\subsubsection{}\label{ind:resol}
 {\em Choices of resolutions of $\cV$ related by an injective cochain map with compatible evaluation maps.}
Let $[\cA' \to \cB']$ be another resolution choice of $\cV$ in the procedure. Assume that
two quasi-isomorphic complexes $[\cA \to \cB]$ and $[\cA' \to \cB']$ are realizable by a quasi-isomorphic cochain map $[\cA \to \cB] \to [\cA' \to \cB']$ with $\cA \to \cA' $ being a monomorphism. We assume also that  restriction maps $rest: \cA \to \cV |_{\sG}$ and $rest': \cA ' \to \cV |_{\sG}$
are compatible. In this setting, we will argue that $\mathbb{K} (p^*B, \ka, \kb )$ is transitively related to $\mathbb{K} (p'^*B', \ka', \kb')$ in the sense of Definition \ref{def:equiv}.

Consider the exact sequence
\[
0 \to A \to A' \xrightarrow{\delta} \pi_* \cQ \to 0,
\] where $\cQ$ is defined to be the cokernel of $\cA \to \cA '$.
Notice that $\tot A$ is the zero locus in $\tot A'$ of the tautological section followed by the pullback of $\delta$.  This section is obviously regular with smooth zero locus (which can be seen locally).  
Using the snake lemma and Case (1) of  Definition \ref{def:equiv}, 
we may assume that $\U$ is the zero locus in $\U'$ of this regular section. 
Let $\iota : \U \to \U '$ be the closed immersion and 
let $$q: \HH  ^{-1} (\U, \Cone (\cO^{\oplus r}_\U \to \wedge ^{-\bullet}  \iota^*p'^*B'^{\vee}   )) \to \HH^{-1} (\U, \Cone (\cO^{\oplus r}_\U \to \wedge ^{-\bullet} p^*B^{\vee} ))$$ denote
the map induced from the projection $\iota^*p'^*B'^{\vee} = p^* B'^{\vee} \to p^*B^{\vee}$.

\begin{Lemma}\label{functor:alpha} We have
\begin{eqnarray}\label{eq:RestAlpha} q \circ \iota ^* (\mathfrak a_{\U'} ) = \mathfrak a_\U . \end{eqnarray}
\end{Lemma} 
\begin{proof}
The class $\mathfrak a_{\U'}$ can be realized as a cochain map in an injective replacement of $\omega _{\fC}$ in 
\eqref{eq:Alpha} followed by $p'^* \RR(\pi_m)_*$ and the multiplication. In this cochain level realization of $\mathfrak a_{\U'}$, it is straightforward
to check \eqref{eq:RestAlpha}.  
\end{proof}

Hence we may assume that $q\circ \iota ^*(\ka ') =  \ka$ by \S\ref{ind:alpha}.
Now, we are  in the situation of Proposition~\ref{prop:PV}, which completes the proof.

\begin{example} \label{ex:rest}
 For $[\cA \xrightarrow{d_{\cA}} \cB ]$ in Proposition \ref{prop:U}, 
we consider  $\cA ' := \cA ' \oplus \cV |_{\sG} , \cB ':=\cB \oplus \cV |_{\sG}$ with $d_{\cA '} = d_{\cA}\oplus  \mathrm{id} $.
We may take, for example, (a) $rest ' = rest + \mathrm{id}$ or (b) $rest ' = rest$.
\end{example}

\subsubsection{Choices of evaluation maps}\label{ind:ev}

We use an argument parallel to \cite[\S 5.2.3]{MF}.
By the construction in Proposition \ref{prop:U}, the evaluation map $ev$ originates from a cochain map realization 
$$[\cA \rightarrow \cB ] \ra [\cV |_{\sG} \ra 0 ]  $$
of $\cV \to  \cV |_{\sG} $.   Given two choices ${rest}_{\cA}, {rest'}_{\cA} : \cA \to \cV|_{\sG}$,
consider two virtual factorizations $\mathbb K (\ka , \kb )$ for $-\sum_i w \circ ev_i$ and $\mathbb K ( \ka' , \kb )$ for $ -\sum_i w \circ ev'_i$.
We claim that these two TK factorizations are transitively related.

To prove it, first note that the difference ${rest}_{\cA} - {rest'}_{\cA}$ descends to a map $g: \cB \to \cV|_{\sG}$.
Set $Q = \pi _* \cV |_{\sG}$.  Pushing forward $g$ via $\pi$
we obtain a map $f  \in \Hom (B, Q)$ such that $f \circ d_A = ev_A- ev_A'$ where $ev_A = \pi_* rest _{\cA}$, $ev'_{A} = \pi _* rest' _{\cA}$.

Consider the sheaf $\mathcal{Q} = \cV |_{\sG} $  on the universal curve.
We replace $[\cA\xrightarrow{d_{\cA}} \cB]$ by $$[\cA\oplus \mathcal{Q} \xrightarrow{(d_{\cA}, \mathrm{id}_{\mathcal{Q}})} \cB\oplus \mathcal{Q}]$$
a new resolution of $\cV$ (via $\cV \ra \cA \xrightarrow{(\mathrm{id}, 0)} \cA \oplus \mathcal{Q}$) and replace $rest_{\cA}$ (resp.  $rest_{\cA}'$) 
by $rest _{\cA} +\mathrm{id} _{\cQ}$
(resp.  $rest_{\cA}'$) as in Example \ref{ex:rest}.
Hence $ev_A, ev_A', g, f$ are replaced by $ev_A + \mathrm{id}_Q$, $ev_A'$, $g+\mathrm{id}_{\cQ}$, $f+\mathrm{id}_Q$, respectively.
Note that $g+\mathrm{id}_{\cQ}$, $f+\mathrm{id}_Q$ are surjective.

Thus, using this construction of $\cQ$ and \S\ref{ind:resol}, we may assume that $g$ and $f\circ d_A$ are surjective. Note also that 
$p^*f\circ \beta $ is a regular section with smooth zero locus. 
Let $\cK_\cB$, $\cK_{\cA}$ be the kernel of $\cB \to \cQ$, $\cA \to \cQ$, respectively.  We obtain another $\pi_*$-acyclic resolution
$[\cK_{\cA} \to \cK_{\cB}]$ of $\cV$ satisfying item (1) of Proposition \ref{prop:U}
and $rest_{\cA} | _{\cK _{\cA} } = rest'_{\cA} | _{\cK _{\cA} }$. Thus again by  \S\ref{ind:resol}
we conclude that $\mathbb K( {\ka}, {\kb})$ and  $\mathbb K( \ka' , \kb )$ are transitively related.

\subsubsection{Proof of Theorem \ref{thm:ind}}\label{ind:result}
Consider two resolutions $[\cA_i \xrightarrow{d_i} \cB_i]$ for $i=1, 2$ together with restriction map $rest _{\cA_i}: \cA_i  \to \cV|_{\sG}$ 
satisfying Proposition \ref{prop:U}, respectively; and virtual factorizations $\mathbb K_i :=\mathbb K(p^*B_i, \ka_i , \kb_i )$ on $\U_i$, $i=1, 2$.

Define $\cB_3$ as the cokernel of the diagonal map $\cV \to \cA_3 := \cA_1 \oplus \cA_2$.
We obtain a $\pi_*$-acyclic resolution $[\cA_3 \xrightarrow{d_3} \cB_3]$ with quasi-isomorphic cochain maps 
$(\phi_{i, 0} , \phi _{i, 1}), : [\cA_3 \to \cB_3 ] \to [\cA_i \to \cB_i]$, and term-wise surjections $\phi _{i, j}$. 
Furthermore $\pi _*\cA_3$, $\pi _* \cB_3$ are locally free.
We have two $rest _{\cA_3}^i: \cA_3 \to \cV |_{\sG}$ each from $rest_{\cA_i}$ such that
$\pi _* (rest _{\cA_3}^i)$ is also surjection.
We apply the Thom-Sullivan-Koszul construction for this resolution with $rest_{\cA_3}^i$
to obtain two factorizations $\mathbb K_3^i$, $i=1, 2$, which are transitively related
\begin{equation}\label{eq: two evs}  \mathbb K_3^1 \sim \mathbb K_3 ^2 \end{equation} by
\S \ref{ind:ev}.

On the other hand, the map $(\phi_{i, 0} , \phi _{i, 1})$ is factored into 
$$[\cA_3 \xrightarrow{d_3} \cB_3 ] \xrightarrow{(\phi _{i,0} \oplus d_3, \phi _{i,1}\oplus \mathrm{id}) } 
[\cA_i \oplus \cB_3 \xrightarrow{d_i\oplus \mathrm{id}} \cB_i \oplus \cB_3] \xrightarrow{proj}  [\cA_i \to \cB_i] .$$
Note that the first map is term-wise injective and the second map has the obvious homotopy inverse map which is also term-wise injective.
Let $\cA_{4,i}:=\cA_i \oplus \cB_3$, $\cB_{4,i}:= \cB_i \oplus \cB_3$. We have maps $rest_{\cA_{4, i}}: \cA_{4, i} \to \cV |_{\sG}$ coming from $rest _{\cA_i}$.
Then $[\cA_{4, i} \to \cB_{4, i}]$ with $rest_{\cA_{4, i}}$  satisfy the condition of Proposition \ref{prop:U} (1). 
Hence we obtain virtual factorizations $\mathbb K _{4, i}$, fitting into a sequence of
transitive relations
\begin{equation} \label{eq: mono relations}  \mathbb K _{i} \sim  \mathbb K _{4, i} \sim \mathbb K_{3}^i \end{equation}
by \S \ref{ind:resol}. Combining \eqref{eq: two evs} and \eqref{eq: mono relations}, we conclude that $\mathbb K_1 \sim \mathbb K_2$.

\subsection{Gluings and Forgetting Tails}

In this section, $\fB_1$ will be a moduli stack for which we allow 3 possible cases which are specified below.   We set $\fB_2 := \fB _{\Gamma}^{g, r, d}$.  

Let $S$ be an algebraic stack equipped with a DM finite type  morphism to the algebraic stack $\fB_1 \times \fB_2$. 
Over $S$, we have two families of curves  $\pi_i : \fC_i\to S$ together with $\Gamma$-bundles $\cP _i$, pullbacked from the 
corresponding data on $\fB_i$ for or $i=1, 2$. 
We are also given some additional data depending on the case.  Here are the three cases we consider.

\begin{itemize}
\item  Case (1)
\begin{itemize}
\item  Fix $r_i$, $d_i$ such that $r_1 + r_2 = r$ and $d_1 + d_2 = d$.
\item $\fB_1 := \prod _{i=1}^2 \fB_{\Gamma}^{g_i, r_i+1, d_i}$.            
\item          In this case, $\fC_1$ is the disjoint union of the universal curves from the universal curve on each factor $\fB_{\Gamma}^{g_i, r_i+1, d_i}$.
          We  rename the markings of $\fC_1$ so that the last two markings $\sG_{r+1}$, $\sG_{r+2}$ correspond to last markings $\sG_{r_1+1}$, $\sG_{r_2+1}$ of the universal curves over 
          $\fB_{\Gamma}^{g_1, r_1+1, d_i}$ and $\fB_{\Gamma}^{g_2, r_2+1, d_i}$ respectively.
          
\end{itemize}

 \item         Case (2)
\begin{itemize}
\item $\fB_1 := \fB^{g-1, r+2, d}_{\Gamma}$  
\item We are given a canonical isomorphism $\sG_{r+1} \cong \sG_{r+2}$ inverting the band  and a morphism $\epsilon : \fC _1 \to \fC_2 $ over $S$ such that 
         $\epsilon$ is the pushout of the diagram $[\sG_{r+1} \cong \sG_{r+2} \rightrightarrows  \fC _1]$.
         (This forces $\fC_2$ to be the quotient of the curve $\fC_1$ coming from identifying the markings $\sG_{r+1}$, $\sG_{r+2}$ through the inversion of the band).
 \item We are given a canonical isomorphism $\cV _1 |_{\sG_{r+1}} \cong \cV _1 |_{\sG_{r+2}}$. Here $\cV_1 := \cP_1 (V)$.
\item We assume that $\cV_2 : = \cP_2(V)$  fits into an exact sequence 
 \begin{equation} \label{eq: exact for glueing case}
 0 \to \cV_2 \to \epsilon _* \cV_1 \xrightarrow{\epsilon_*( rest _{r+1} - \zeta \circ rest _{r+2}) } \epsilon_* \cV_1|_{\sG_{r+1}} \cong  \epsilon_* \cV_1|_{\sG_{r+2}} \to 0 , 
 \end{equation}
 where $rest _i$ denotes the restriction map at the $i$-th marking and 
 $\zeta : = \exp (\pi \sqrt{-1} / d_w) \in \CR$ acts on $V$ (and hence on $\cV_1$)  via $\CR \to \Gamma \subset GL(V)$.
 \end{itemize}

 \item Case (3)
 \begin{itemize}
 \item  $\fB_1 =   \fB^{g, r+1, d}_{\Gamma}$.
 \item We are given a map  $\epsilon : \fC_1 \to \fC_2$ preserving the first $r$ markings and their types $\cP_1$, $\cP_2$ at the markings.
 \item We are given  a $\Gamma$-representation decomposition $V= V^f \oplus V^m$. Denote $\cP_1 ( V^f ) $, $\cP_1 ( V^m)$, 
 by $\cV_1 ^{f}, \cV_1^{m}$, respectively. Here $\cP_1$ is the universal  principal $\Gamma$-bundle on $\fC _1$.
 We assume that the natural map 
 \begin{equation}\label{eq: van} \pi _* (\Sym (\cV_{1}^f|_{\sG_{r+1}}) )  \xrightarrow{\text{via sum}, w, \kappa} \cO _S 
 \end{equation}  vanishes. 
 \item We assume that $\cV_2$ fits into an exact sequence 
  \begin{equation} \label{eq: exact for forgetting tails case} 0 \to \cV_2 \to \epsilon _* \cV_1 \xrightarrow{\epsilon _*(pr\circ rest_{r+1})} \epsilon _* \cV_{1}^{m}|_{\sG_{r+1}} \to 0 . \end{equation}
  \end{itemize}
 \end{itemize}

Using $S\to \fB_i$ 
we apply the Thom-Sullivan Koszul construction to obtain perfect TK factorizations 
$\mathbb K (p_i^* B_i, \ka _i, \kb _i)$ on open substacks $\U_i$ of the $\So$-stack $\tot A_i$. 
Here $p_i$ denotes the restriction of the projection map $p_{A_i} |_{\U _i} \to \So$,
$[\cA_i \to \cB_i]$ are the resolutions of $\cV_i$ satisfying Proposition \ref{prop:U},
and $A_i := \pi _{i *} \cA_i$, $B_i := \pi_{i*}\cB_i$.
Provided with the above conditions, we claim that they are related in a suitable sense.

Before stating their relation, we need some notation. First in order to unify  \eqref{eq: exact for glueing case} and
 \eqref{eq: exact for forgetting tails case}, we set $V^f= 0$, $V^m = V$ in Cases (1) and (2) 
 and set $rest^{\mathfrak{m}} := rest _{r+1} - \zeta\circ rest _{r+2}$  in Cases (1) and (2) and $rest^{\mathfrak{m}} :=  pr\circ rest_{r+1}$ in Case (3).  
Hence in all cases we have the exact sequence
 \[  \xymatrix{ 0 \ar[r]   &  \cV_2 \ar[r]    &   \epsilon _* \cV _1 \ar[r]^{\epsilon _* rest^{\mathfrak{m}} \ \ \  \ }          & \epsilon _* \cV_1 ^m |_{\sG_{r+1}}  \ar[r] & 0 .} \]
Let $\cA_0$ be the kernel of $\cA_1 \xrightarrow{rest^{\mathfrak m}_{\cA_1}}  \cV_{1}^m |_{\sG_{r+1}}$, where $rest^{\mathfrak m}_{\cA_1}$ is the extension
of $rest^{\mathfrak{m}}$ corresponding to the given extensions of restriction maps.
 Denote by $\iota$ the closed inclusion $\U_0:=\tot A_0 \cap \U_1 \to \U_1$, where 
 $A_0 := \pi _{1 *} \cA_0$ and $\U_1$ is a suitable open substack of $\tot (A_1 := \pi _{1 *} \cA_1)$. 

\begin{Thm}\label{thm: cat glue and tail}
 In cases (1), (2), and (3), the TK factorizations, $ \mathbb K (\iota^*\ka _1, \iota^* \kb _1)$ and $ \mathbb K ( \ka _2, \kb _2)$, are  transitively related with respect to 
  $( (I\cX)^{r}, -\boxplus w) \ti (\prod _i \overline{M}_{g_i, r_i+1} , 0 )$, 
    $( (I\cX)^{r}, -\boxplus w) \ti \overline{M}_{g-1, r+2} , 0 )$, 
$ ((I\cX)^{r} , -\boxplus w ) \ti (\overline{M}_{g, r+1}, 0)$, respectively.
\end{Thm}

\begin{proof} 
We begin by choosing $[\cA_1 \to \cB_1]$ and $rest _{\cA_1, i} : \cA_1 \to \cV_1|_{\sG_i}$ for all $i$
satisfying Proposition \ref{prop:U} for $S \to \fB_1$.  Using this data, we construct a resolution $[\cA_2 \to \cB_2]$ of $\cV_2$ as follows.

We take $\cA_2 = \epsilon _* \cA_0$,
 which is not necessarily locally free. However $\pi _* \cA_2$ is locally free.
Let $\cB_2$ be the cokernel of the map $\cV _2 \to \cA_2$ induced from $\epsilon_* \cV_1 \to \epsilon _* \cA_1$ to get the
commuting diagrams of exact sequences 
\[     \xymatrix{ 0 \ar[r]   &  \cV_2 \ar[r] \ar[d]    & \ar[d]  \epsilon _* \cV _1 \ar[r]^{\epsilon _* rest^{\mathfrak{m}} \ \ \  \ }          & \ar[d]^{=} \epsilon _* \cV_1 ^m |_{\sG_{r+1}}  \ar[r] & 0 \\
                        0 \ar[r]   & \cA_2 \ar[r]  \ar[d]    & \ar[d]  \epsilon _* \cA_1  \ar[r]^{\epsilon _* rest^{\mathfrak m}_{\cA_1} \ \  } & \ar[d] \epsilon _* \cV_1 ^m |_{\sG_{r+1}}   \ar[r] & 0 \\
                        0 \ar[r]   & \cB_2  \ar[r]^{\cong}& \epsilon_* \cB_1                           \ar[r]                                            & 0           &  } \]  
The exact sequence $0 \to \cV_2 \to \cA_2 \to \cB_2 \to 0$ provides a resolution of $\cV_2$. Also, for $i=1, ..., r$,
we have restriction maps $rest _{\cA_2, i}: \cA_2 \to \cV_2|_{\sG_i}$  induced from $res _{\cA_1, i} : \cA_1 \to \cV_1 |_{\sG_i}$. The resolution $[\cA_2 \to \cB_2]$ together with the restriction maps  
satisfies Proposition \ref{prop:U} (1) for $S \to \fB_2$.

Now consider the virtual factorization $\mathbb K (\mathfrak a _{\U_1}, \kb _1)$ from $[\cA_1 \to \cB_1]$ with $rest_{\cA_1, i}: \cA_1 \to \cV_1|_{\sG_i}$, $i=1, ..., r+2$
in Case (1) and (2)
or $i=1, ..., r+1$ in Case (3)
and the virtual factorization $ \mathbb K ( \mathfrak a_{\U _2}, \kb _2)$ from $[\cA_2 \to \cB_2]$ with $rest_{\cA_2, i}: \cA_2 \to \cV_2|_{\sG_i}$, $i=1, ..., r$.
Note that $A_2 = A_0 \subset A_1$, $B_2 = B_1$, and $d_{A_1}|_{A_2} = d_{A_2}$ where $d_{A_i}: A_i \to B_i$ are the differentials.
These, together with compatibility of restriction maps and $w(\zeta \cdot x ) = - w (x)$ for Case (1) and (2) and the vanishing of \eqref{eq: van}, show
that  $\iota ^* \mathfrak a_{\U_1} =  \mathfrak a_{\U _2}$ (up to restrictions to appropriate open substacks)
and $\iota ^* \kb _1 = \kb _2$. Hence by \S  \ref{ind:a}  we have an isomorphism 
$\mathbb K (\iota ^* p_1^* B_1, \iota ^* \mathfrak a _{\U_1}, \iota ^* \kb _1) \cong  \mathbb K ( p_2^*B_2, \mathfrak a_{\U _2}, \kb _2)$.  
The result follows from Theorem~\ref{thm:ind}.
 \end{proof}

\section{GLSM cohomological invariants}  \label{sec: 4}

In this section we will define GLSM cohomological invariants 
\[ \Omega _{g, r, d} : \State ^{\ot r} \to \overline{M}_{g, r} \]
using the virtual factorizations.
Here $\State$ is a complex vector space with a nondegenerate pairing associated to the GLSM input data;  see \eqref{eq: state}.
Throughout this section unless otherwise specified we consider the case where $S\to \fB_{\kG}$ is the identity map, which
ensures that $\Z$ is proper over $\Spec\/\CC$; see \S \ref{ind:support} for the notation $\Z$.

\subsection{State space and pairing}\label{subsection:State}

\subsubsection{Integration maps and Pairings} \label{subsec: pairings}
Let $\tZ_1$, $\tZ_2$ be closed DM substacks of a smooth DM stack $\tX$ and suppose that $\tZ_1 \cap \tZ_2$ is proper over $\Spec\/\CC$.
For $a \in  \HH ^* _{\tZ_1} (\tX, (\Omega ^{\bullet}_{\tX}, d\tW ))$, $b \in  \HH ^* _{\tZ_2} (\tX , (\Omega^{\bullet}_{\tX} , -d\tW ))$, we define
 \[ \lan a, b\ran := \int _{\tX} a \wedge b . \]
Here the wedge product is defined in Appendix \ref{sec: wedge product} 
and the integration map $\int _{\tX }$ is defined in Appendix  \eqref{eqn: def of integration map} by passing from the algebraic Hodge cohomology 
$\HH ^* _{\tZ_1\cap \tZ_2} (\tX, (\Omega ^{\bullet}_{\tX}, 0))$ to the $(\dim \tX, \dim \tX)$-degree part of its holomorphic Hodge cohomology.

\subsubsection{Nondegeneracy}\label{Serre:GAGA}
For a separated quasicompact DM ${\bf k}$-stack $\cY$ with the structure map $p: \cY \to \Spec \,{\bf k} $,
there exists a functor $p^! : D^+(Qcoh ({\bf k})) \to D^+(Qcoh (\cY)) $ right adjoint to $\RR p_*$ by \cite[Theorem 1.16]{Nironi}.  
Furthermore if $\cY$ is smooth and compactifiable with dualizing sheaf $\omega_{\cY}$ then by \cite[Theorem 2.11, Theorem 2.22]{Nironi}, we have an isomorphism $p^! {\bf k} \cong \omega _{\cY}$.

In the case where $\cY = [V^{ss}/G]$, there exists an open immersion into a smooth proper DM stack.  Indeed, 
$[V^{ss}/G] \subset [\PP (V\oplus \CC) /G]$ and we may take a sequence of blow-ups of $\PP (V\oplus \CC)^{ss}$ to remove the strictly semistable locus;
see \cite{Kirwan}.

For $\cY= \cX_{h}:= [(V^{ss})^h/\rC_G(h)]$, let $n_{\cX_h}$ be the dimension of $\cX_h$. Abusing notation, we denote the restriction  $w|_{\cX_h}$ under the 
natural map $\cX_h \to \cX$ also by $w$.
As above, there is a smooth proper DM stack $\overline{\cX}_h$ admitting an open inclusion $j: \cX_h \to \overline{\cX}_h$.
Then
we have
\begin{align*} 
       & \RR\Hom(\RR\Gamma(\Omega^{\bullet}_{\cX_h}, -dw), {\bf k}) &  \\
=    &  \RR\Hom(\RR\Gamma \RR j_* (\Omega^{\bullet}_{\cX_h}, -dw), {\bf k} ) &  \\
 \cong &  \RR\Hom(\RR j_* (\Omega^{\bullet}_{\cX_h}, -dw),\omega_{\overline{\cX}_h}[n_{ \cX_h} ] ) & \text{ by Serre duality } \\
=  & \RR\Gamma \RR j_* \RR\cH om((\Omega^{\bullet}_{\cX_h}, -dw),\omega_{\cX_h}[n_{\cX_h}] ) 
&\text{since $\RR j_*(\Omega^{\bullet}_{\cX_h}, -dw) |_{\overline{\cX}_h-{\cX_h}} = 0 $}   \\
=  &\RR\Gamma ((\wedge ^{-\bullet} \cT_{\cX_h}, \iota _{dw} ) \otimes \omega_{\cX_h} [n_{\cX_h]}) &  \\
=   & \RR\Gamma(\Omega^{\bullet}_{\cX_h}, dw)  [2n_{\cX_h}]   . &
\end{align*}
Here we considered $(\Omega^{\bullet}_{\cX_h}, \pm dw)$ as a bounded complex and the vanishing 
$\RR j_*(\Omega^{\bullet}_{\cX_h}, -dw) |_{\overline{\cX}_h-{\cX_h}} = 0$ follows from the fact that
the support of $(\Omega^{\bullet}_{\cX_h}, -dw)$ lies on $Z(dw)$, which is proper over $\Spec\/\bfk$ by Corollary \ref{cor: zdw}.

This shows that
the pairing defined by the composition of natural maps
\begin{eqnarray*}  \HH^*({\cX_h}, \Omega^{\bullet}_{\cX_h}, dw) \ot \HH^{\star}({\cX_h}, \Omega^{\bullet}_{\cX_h}, -dw)  & \xrightarrow{} 
& \HH^{*+\star} _{Z(dw)} ({\cX_h}, (\Omega ^{\bullet}_{\cX_h}, 0)) \\
           & \xrightarrow{\mathrm{proj}}  & H^{n_{\cX_h}}_{Z(dw)} ({\cX_h}, \omega_{\cX_h} )  \\
& \xrightarrow{}&  H^{n_{\cX_h}} (\overline{\cX}_h, \omega_{\overline{\cX}_h} )   \xrightarrow{\tr}   {\bf k} \end{eqnarray*}
is nondegenerate. 
Here the third arrow requires the properness of $Z(dw)$.

By applying GAGA \cite{Toen} to $\overline{\cX}_h$, there is a canonical isomorphism
\begin{align*}
 \HH ^* ({\cX_h}, ( \Omega^{\bullet}_{\cX_h}, dw) ) & \cong \HH ^* ({\overline{\cX}_h}, \RR j_*  ( \Omega^{\bullet}_{\cX_h}, dw) ) \\
& \cong \HH ^*( {\overline{\cX}_h}^{an},  \RR j_* (\Omega^{\bullet}_{{\cX_h}^{an}}, dw)) \\
& \cong \HH ^*( {\cX_h}^{an}, (\Omega^{\bullet}_{{\cX_h}^{an}}, dw)) \\
\end{align*}
To simplify notation, we use this isomorphism implicitly in what follows.

\begin{Lemma}\label{lemma: nond pairing} The pairing 
\begin{align*}
\HH^*({\cX_h}, \Omega^{\bullet}_{\cX_h}, dw) \ot \HH^{\star}({\cX_h}, \Omega^{\bullet}_{\cX_h}, -dw) & \to  
 \bf{k} \\
a \otimes b & \mapsto \int _{\cX_h} a \wedge b  
\end{align*}
 is nondegenerate.
\end{Lemma}
\begin{proof}  This follows from the fact that the composition of the last three maps above coincides with $\int _{\cX_h}$ up to a nonzero constant multiplication; 
see \cite[Theorem 2.1]{Conrad: more}.
\end{proof}

\subsubsection{Critical locus of restricted functions}

Let $H$ be an abelian group acting diagonally on $\CC ^n$ and let $D$ be an analytic open domain $D$ of $\CC ^n$ which is 
$H$-invariant. Let $f$ be an complex holomorphic function on $D$ which is $H$-invariant. Denote by $D^H$ the $H$-fixed locus 
of $D$. 

\begin{Lemma}
$Z(df) \cap D^H = Z(d(f|_{D^H}))$
\end{Lemma}

\begin{proof}
Write $f = \sum_i c_i \prod_j x_j ^{e_{ij}}$ as a sum of (possibly infinite) monomials, with nonzero coefficients $c_i$.
Without loss of generality we may assume that the $H$-fixed locus $(\CC^n)^H$ is defined by the zero set of the coordinates $\{ x_i : i \ge k_0\}$ for some $k_0$.

For $k\ge k_0$  consider the expression $\frac{\partial}{\partial _k}   \prod _j x_j ^{e_{ij}}$.  We claim it is always contained in the ideal  
generated by $x_m $, $m=k_0, ... , n$. 
In fact this expression vanishes when $e_{ik}=0$ and is contained in the ideal generated by $x_k$ when $e_{ik} > 1$.
If $e_{ik} = 1$, then $e_{im}\ne 0$ for some $m \ge k_0$ since $\prod _j x_j ^{e_{ij}}$ is $H$-invariant. Therefore $\frac{\partial}{\partial _k}   \prod _j x_j ^{e_{ij}}$ lies in the ideal generated by $x_m$.
Hence
\[
\langle \partial_k f,  x_i \ | \ i \geq k_0 \text{ and } k < k_0  \rangle =  \langle \partial_k f,  x_i \ | \ i \geq k_0 \rangle.
\]
This is just the algebraic expression for the statement of the lemma.
\end{proof}

\begin{Cor}\label{cor: zdw}
For any separated DM stack $\cY$ with a regular function $f$, 
$Z(d (f|_{I\cY}))$ maps into $Z(df)$ under the natural finite map $I\cY \to \cY$.
\end{Cor}

\subsubsection{State space} \label{sec: state space def}
There is a natural forgetful map from the inertia stack to the target $I\cX \ra \cX$ (recall $\cX := [V^{ss}/G ] $).  
Pulling back $w$ we get a regular function on $I\cX$ which,  by abusing notation, we also denote by $w$.

Consider $\zeta := \exp (\pi \sqrt{-1} /\bfdeg ) \in \CC ^\ti _R$. 
If we let $q_j:= c_j / \bfdeg $, then $\zeta$ is $(\exp \pi \sqrt{-1} q_j)_j$, via $\CR \hookrightarrow GL(V)$,   so that $\zeta ^2 = J$.
By the $\CC^\ti_R$-weight of $w$, we see that $w (\zeta \cdot x) = - w (x)$.
Consider the automorphism of the inertia stack of $\cX$ defined by 
\label{inv}
\begin{align*}
\inv : I\cX& \to I\cX  \\
 (x, h) & \mapsto (\zeta \cdot x, h^{-1}) . 
\end{align*}
We define the {\em state space} for the GLSM input data as the $\CC$-vector space
 \begin{equation} \label{eq: state}
 \State        
 :=  \HH ^* (I\cX , (\Omega ^{\bullet}_{I\cX}, dw ))
 \end{equation}
with pairing $\eta$ given by
\begin{align} \label{def pairing} \eta :  \State \ot \State & \to \CC \notag \\
 \gamma_1 \ot \gamma _2 & \mapsto 
\int _{I\cX} \gamma _1 \wedge \inv^* \gamma _2 . 
\end{align} 

\begin{example}\label{ex: LG/CY 2}
We continue our running example (see Example~\ref{ex: LG/CY}).  Then $I\cX_+ = \cX_+ = \cO_{\PP^4}(-5)$. There is a commutative diagram
\[
\begin{tikzcd}
&  & \cO \ar[d, "{(taut, f)}"] \ar[ddl, dashed, swap, "s"]&  & \\
&  &  \cO(-5) \oplus \cO(5) \ar[d, "{(\partial_i f, - 1)}"] \ar[dl, dashed]&  & \\
0 \ar[r] & \Omega^1  \ar[r] & \cO(-1)^{\oplus 5} \oplus \cO(5) \ar[r, "{(x_i, taut)}"] & \cO \ar[r] & 0 \\
\end{tikzcd}
\]
where the dashed arrows follow from Euler's homogeneous function theorem.
This leads to an exact sequence
\[
\begin{tikzcd}
&  & \cO \ar[d, "s"] \ar[dl, swap, "{(taut, f)}"] \ar[dr, "0"]&  & \\
0 \ar[r] &  \cO(-5) \oplus \cO(5) \ar[r] & \Omega^1 \ar[r] & \text{coker} \ar[r] & 0 \\
\end{tikzcd}
\]
and $\text{coker}|_{Z(f)} = \Omega^1_{Z(f)}$ by the conormal bundle sequence.  Let $i : Z(f) \to \cO_{\PP^4}(-5)$ be the inclusion along the zero section.  The identifications
\[
\Omega^i_{ \cO_{\PP^4}(-5)} = \Omega^{5-i}_{ \cO_{\PP^4}(-5)} \ \ \ \ \text{ and } \ \ \ \ \Omega^i_{Z(f)} = \Omega^{3-i}_{Z(f)} 
\]
yield a natural morphism of complexes
\[
(\Omega ^{\bullet}_{\cO_{\PP^4}(-5)}, dw )[2] \to i_* (\Omega_{Z(f)}^{\bullet}, 0).
\]
One can check locally that this natural map is a quasi-isomorphism 
since the section $(taut, f)$ is regular.
Taking cohomology gives an isomorphism
\[
\HH ^* (\cO_{\PP^4}(-5), (\Omega ^{\bullet}_{\cO_{\PP^4}(-5)}, dw )) \cong \HH ^{*-2} (Z(f), \Omega_{Z(f)}^{\bullet}).
\]
Hence,
\begin{align*}
\State_+ & = \HH ^* (\cO_{\PP^4}(-5), (\Omega ^{\bullet}_{\cO_{\PP^4}(-5)}, dw )) \\
& = \HH ^{*-2} (Z(f), \Omega_{Z(f)}^{\bullet}) \\
& = \bigoplus_{p,q} \text{H}^{p,q}(Z(f)).
\end{align*} 
Furthermore, the pairing can be identified with the cup product pairing.  

On the other hand, $I\cX_- = \coprod_{i=1}^{4} B\ZZ_5 \coprod [\mathbb A^5 / \ZZ_5]$ and
\begin{align*}
\State_- & = \bigoplus_{i=1}^{4} \text{H}^* (B\ZZ_5, \CC ) \oplus \text{H}^* ([\mathbb A^5 / \ZZ_5], (\Omega ^{\bullet}_{[\mathbb A^5 / \ZZ_5]}, df ))  \\
& = \CC^4 \oplus (\CC[x_1, ..., x_5] / df)^{\ZZ_5}  \hspace{80pt} \text{ since }df \text{ is regular} \\
& = \CC^4 \oplus \{\text{polynomials in the Jacobian of degree 0,5,10,15}\} \\
& = \CC^4 \oplus (\CC \oplus \CC^{101} \oplus \CC^{101} \oplus \CC)
\end{align*}
and the pairing is the Grothendieck residue pairing up to a normalization constant (see \cite[Proposition 2.1]{CIR} and  \cite[Theorem 2.1]{Conrad: more}). 
\end{example}

\begin{Rmk}
The calculation
\[
\State_+  = \bigoplus_{p,q} \text{H}^{p,q}(Z(f))
\]
in the above example works for any geometric phase of a GLSM with only minor modification.
\end{Rmk}

We investigate the pairing using a decomposition of $I\cX$ and check its nondegeneracy as follows.
Let 
\begin{equation}\label{eqn: def of X_h} \cX_h:= [(V^{ss})^h/\rC_G(h)] . \end{equation}
 We have $I\cX = \coprod _{[h]\in G/G} \cX_h$.
Note that $(V^{ss})^h = (V^{ss})^{h^{-1}}$ and $\rC_G(h) = \rC_G(h^{-1})$.  Hence $\cX_h = \cX_{h^{-1}}$.
We define $\State_{h} := \HH^* (\cX_h , ( \Omega ^{\bullet}_{\cX_h},    dw ) )$, i.e., $\State_{h}$ is the summand of the state subspace supported on the component $\cX_h $ with conjugacy class
$h$.  Let $\State_{h}^- :=  \HH^* (\cX_h , ( \Omega ^{\bullet}_{\cX_h},   - dw ) )$.
 Note that $\State _h \cong \State ^{-}_{h^{-1}} $ by $\inv$.
 The pairing
 is decomposed into
 \[ \State_h \ot \State_h = \State _h \ot \State _{h^{-1}} \xrightarrow[\simeq]{\mathrm{id}\ot \inv^*} \State_h \ot \State _{h}^- \to \HH^*_{Z_h(dw)} (\cX_h, (\Omega ^{\bullet}_{\cX_h}, 0))  \xrightarrow{\int_{\cX_h}} \bfk , \]
where $Z_h(dw)$ is the critical locus of $w|_{\cX_h}$.
In Lemma \ref{lemma: nond pairing} we saw that the pairing 
$\State_h \ot \State _{h}^- \to  \bfk$
 is nondegenerate.

\subsection{Virtual class} \label{subsec: vir class}
Let  $\tX$ be a DM stack and let $\tW$ be a regular function on $\tX$ and $\mathbb K(F, \tau, \sigma)$ be a perfect TK factorization on $(\tX, \tW)$. 
 Assume that $\mathbb K(F, \tau, \sigma)$ is supported on a closed substack $\tZ$.  

We may define a cohomology class of a perfect TK factorization using the localized Chern character map \eqref{eq: local Chern} 
and the wedge product \S\ref{sec: wedge product} as follows.
Using  \eqref{eq: local Chern}, we can define $\ch^{\tX}_{\tZ} (\mathbb K(F, \tau, \sigma)) \in  \HH^{even}_{\tZ} (\tX, ( \Omega ^{\bullet}_{\tX},  d\tW) )$. 
The bigrading on $\HH^* (\tX, ( \Omega ^{\bullet}_{\tX}, 0) )$ allows us to define Chern classes as polynomials in the homogeneous pieces 
of $\ch (F)$ and in turn Chern roots in a universal extension. 
Here $\ch (F)$ is the Chern character of $F$ regarded as the matrix factorization for $0$-function.
Now define $\tda (F)$ using the usual characteristic polynomial in the Chern roots. 
In particular, $\tda (F)$ is an element of $ \HH^{even} (\tX, ( \Omega ^{\bullet}_{\tX}, 0) )$.

\begin{Def}\label{def:tdch}
The \emph{Todd-Chern class} of a perfect TK factorization $\mathbb K(F, \tau, \sigma)$ is the element
\begin{eqnarray*} \label{eqn: todd chern} \tdch ^{\tX}_{\tZ} \mathbb K(F, \tau, \sigma) := (\frac{i}{2\pi })^{\rank F} \tda (F)  \wedge \ch ^{\tX}_{\tZ} (\mathbb K(F, \tau, \sigma))
 \end{eqnarray*}
 in $\HH^{even} _{\tZ}(\tX, (\Omega ^{\bullet}_{\tX}, d\tW) )$. 
 When $\tZ = \tX$, the notation is abbreviated as 
 \[
 \tdch  \mathbb K(F, \tau, \sigma) :=  \tdch ^{\tX}_{\tX} \mathbb K(F, \tau, \sigma).
 \]
When $\tZ$ is proper over $k$, we define the \emph{properly supported Todd-Chern class} as the image of $\tdch^{\tX}_{\tZ}$ in properly supported cohomology i.e.
\[
 \tdchc  \mathbb K(F, \tau, \sigma) := \mathrm{nat} (     \tdch ^{\tX}_{\tZ} \mathbb K(F, \tau, \sigma)   )  \in \HH^{even} _{c}(\tX, (\Omega ^{\bullet}_{\tX}, d\tW) ) .
 \]
 where $\mathrm{nat}: \HH^{even} _{\tZ}(\tX, (\Omega ^{\bullet}_{\tX}, d\tW) ) \to \HH^{even} _{c}(\tX, (\Omega ^{\bullet}_{\tX}, d\tW) )$ is the natural map.
\end{Def}

\begin{Rmk} The definition above is motivated by Remark \ref{rmk: top c1 and cat c1} and the Borel-Serre identity 
\[ \sum_p (-1)^p \ch (\wedge ^p F^{\vee}) = c_{\rank F}(F) \cdot \td (F) ^{-1}. \]
\end{Rmk} 

Recall from \S\ref{ind:support} that $\Z$ denotes the moduli space of $\epsilon$-stable LG quasimaps to $Z(dw)$, that $\Z$ is a proper DM stack, and that $\mathbb K_{g,r,d}$ is supported on $\Z$.

We now define a cohomology class associated to our data.  
For this, let ${\bf r}_i$ be the locally constant function on $\Ugrd$ defined as
the order of the automorphism group of the $i$-th marking on each connected component; see \cite[\S 2.1]{Tseng}.

\begin{Def}\label{def:virclass} Let $W := \sum _i ev_i^* w$.
We define the \emph{virtual fundamental class} of the LG-moduli space $LG_{g,r,d}(\cX)$ to be
\[ [\Ugrd]_W^{\vir} : = (\prod _{i=1}^r {\bf r}_i) \tdch^{\Ugrd}_{\Z} \mathbb K_{g,r,d} \ \ \  \in  \ \HH^{even} _{\Z}(\Ugrd, (\Omega ^{\bullet}_{\Ugrd}, -d W) ).  \]
\end{Def}

\subsection{GLSM invariants}\label{sec glsm inv}

There is a forgetful map $fgt: \Ugrd \ra \overline{M}_{g,r}$. \label{fgt}
Using the pairing and the natural pullback map $fgt ^*$ we define a pushforward 
\[ fgt _* : \HH^* _\Z(\Ugrd, (\Omega^{\bullet} _{\Ugrd}, 0)) \ra \HH^* (\overline{M}_{g,r}^{an}, (\Omega^{\bullet} _{\overline{M}^{an}_{g,r}}, 0)) \]
by the requirement that
\[ \lan fgt _* a, b \ran = \lan a , fgt ^* b \ran \] where the left pairing is nondegenerate by Serre duality for smooth proper DM stacks.

For all $g,r, d$ such that $2g-2 + r > 0$, the virtual class $[\Ugrd]^{\vir}_W$ defines a homomorphism,
\begin{eqnarray*} 
\Omega _{g, r , d}: \HH ^* (I\cX, (\Omega ^{\bullet}_{I\cX }, dw))^{\ot r}  & \xrightarrow[\cong]{\text{K\" unneth}} & \HH ^* ((I\cX)^r, (\Omega ^{\bullet}_{(I\cX)^r}, dw^{ \boxplus r} ))  \\
& \xrightarrow {ev^*} & \HH^*(\Ugrd, (\Omega^{\bullet}_{\Ugrd}, dW )) \\
& \xrightarrow{ \wedge [\Ugrd]^{\vir}_W } & \HH^* _{\Z}(\Ugrd, (\Omega^{\bullet} _{\Ugrd}, 0))  \\
& \xrightarrow{fgt_*} & \HH^* (\overline{M}_{g,r}^{an} , (\Omega^{\bullet} _{\overline{M}_{g,r}^{an}}, 0))  \\
& \cong & H^*(\overline{M}_{g,r}^{an}, \CC) .
\end{eqnarray*}
 Here the last isomorphism is possible since $\overline{M}_{g,r}$ is a compact K\"ahler orbifold. 
 This isomorphism is independent of the choice of K\"ahler structure (see, e.g., the proof of \cite[Corollary 3.2.12]{Huyb}).
From now on we will write $H^*(\overline{M}_{g,r}^{an}, \CC)$ using the abbreviated notation $H^*(\overline{M}_{g,r})$.

 Our convention is that when $r =0$ with $g\ge 2$, 
 \begin{align*}
 \Omega _{g, 0 , d} : \bf k & \to H^*(\overline{M}_{g,0}) \\
 1 & \mapsto  fgt_* [\Ugrd]^{\vir}_W.
 \end{align*}

\subsection{Independence}
Let
\[
(ev, fgt)_* : \HH^* _\Z(\Ugrd, (\Omega^{\bullet} _{\Ugrd}, W)) \to \HH ^* (I\cX, (\Omega ^{\bullet}_{I\cX }, -dw))^{\ot r}  \ot H^*(\overline{M}_{g,r} )
\]
be defined by 
\begin{align*}
 \lan (ev, fgt)_* a, \gamma \ot \delta \ran = \lan a , (ev, fgt)^* (\gamma \ot \delta) \ran \\
  \text{ for } a \in \HH^* _\Z(\Ugrd, (\Omega^{\bullet} _{\Ugrd}, W)) , \ \gamma \in  \HH ^* (I\cX, (\Omega ^{\bullet}_{I\cX }, dw))^{\ot r}  , \  \delta \in H^*(\overline{M}_{g,r} ) ,
\end{align*}
where the left pairing is nondegenerate by Serre duality for DM stacks and \S\ref{Serre:GAGA}.
Then note that $\lan (ev, fgt)_* [\Ugrd]_W^{\vir} , \gamma \ot \delta \ran = \lan \Omega _{g, r , d} (\gamma ), \delta \ran $.

\begin{Thm}\label{thm:cind}
The class
$$(ev, fgt)_* [\Ugrd]_W^{\vir} \in \HH ^* (I\cX, (\Omega ^{\bullet}_{I\cX }, -dw))^{\ot r}  \ot H^*(\overline{M}_{g,r})
$$
 is independent of all choices in the construction of  $[\Ugrd]^{\vir}_{W}$.
Equivalently, the GLSM invariants $\Omega_{g,r,d}$ are  independent of all choices in the construction of  $[\U]^{\vir}_{W}$.
\end{Thm}
\begin{proof}
This is immediate from  Theorem~\ref{thm:ind} and Corollary~\ref{cor: trans related} below.
 \end{proof}

To prove Theorem \ref{thm:cind} we will use a general deformation result for a related pair of virtual factorizations.

\begin{Prop}\label{prop:GRR} Consider two related perfect TK factorizations $\mathbb{K}(F, \tau  , \sigma ) $, $\mathbb{K}(F', \tau '  , \sigma ') $
as in Definition \ref{def:equiv}. 
Let $\tZ_1, \tZ _2$ be two closed subloci of $\tX$ such that their intersection $\tZ_1 \cap \tZ_2$ is proper over $\Spec\/\CC$. 
Suppose that $\mathbb K(F, \tau, \sigma)$  is supported on $\tt{Z}_2$, then for all
 $a \in \HH^*_{\tZ_1} (\tX ', (\Omega ^{\bullet}_{\tX '}, - d \tt{W}' ))$ we have
  \begin{eqnarray}\label{GRR1} 
  \int _{\tX '} a \wedge  \tdch^{\tX '}_{\tZ_2} \mathbb{K}(F', \tau ' , \sigma ')     
  & = &  \int _{\tX } \iota ^* a  \wedge \td \ch^{\tX}_{\tZ_2} \mathbb{K}(F, \tau , \sigma  ) \end{eqnarray}
  where $\iota ^*a$ is considered as an element of $\HH^*_{\tZ_1} (\tX, (\Omega ^{\bullet}_{\tX}, - d \tt{W} ))$.
\end{Prop}

\begin{proof} If $\iota$ is an open immersion, then the statement is trivial.  Consider the case where $\iota$ is a closed immersion.
By Lemma \ref{lemma: deform fact}, 
there exists a perfect TK factorization $\mathbb K (F_{\tM}, \tau _{\tM}, \sigma _{\tM})$, supported on $\tZ_2 \ti \PP ^1$ inside the deformation to the normal cone  $\tM $ of $\tX$ in $\tX '$,   whose restriction to
$0$ and $\infty$ are co-quasi-isomorphic to 
$\mathbb{K} (F', \tau ', \sigma ')$ and
 $(p_{Q|_{\tX}}^*\mathbb{K}  (F, \tau , \sigma  ) ) \ot \mathbb{K} (p_{Q|_{\tX}}^*Q, 0, taut_Q )$, respectively; see Lemma \ref{lemma: deform fact}
 for the notation $p_{Q|_{\tX}}$, $taut _{Q}$, $\iota _t$.

Let $\pi _{\tM} :  \tM \to \tX '$ be the projection
and consider the integral
\begin{equation}\label{eq: deform cone}  \int_{\tM_t } \iota _{t}^*\left(  (\pi^*_{\tM } a ) \wedge 
\tdch ^{\tM}_{\tZ_2 \times \PP ^1 } \mathbb K (F_{\tM}, \tau _{\tM}, \sigma _{\tM})  \right) ,  \end{equation}
which is independent of $t$ by Lemma \ref{lemma:def inv}.

At $t=0$  \eqref{eq: deform cone} 
becomes the left hand side of \eqref{GRR1} by the functoriality of $\tdch$ (Lemma \ref{lemma: functor tdch}).

On the other hand, at $t=\infty$, \eqref{eq: deform cone} becomes 
the right side of \eqref{GRR1}.  This uses the functoriality of $\tdch$ (Lemma  \ref{lemma: functor tdch}), 
Lemma \ref{lemma: at mult}, 
and 
Lemma \ref{lemma: projection for Q}, together with the standard fact that for every point closed $z$ of  ${\tX}$
$$\int _{[\tot Q|_z] } \td\ch^{\tot Q|_{z}}_{z} \{ 0, taut_{Q|_z} \} = 1. $$ 
\end{proof}

As a corollary, we have a cohomological projection formula for matrix factorizations.
\begin{Cor}\label{cor: mf projection for}
Let $\tX '$ be a smooth DM stack of finite type, containing a smooth closed substack $\tX$ as the  zero locus of a regular section
of a vector bundle  $Q$ on $\tX '$. Let $\tW '$ be a regular function on $\tX$ which vanishes on $\tX$ and let $\iota : \tX  \to \tX'$ be the inclusion.
 For $a \in \HH^*_{\tZ} (\tX', (\Omega ^{\bullet}_{\tX'}, - d \tW ' ))$ with $\tZ$ a proper closed substack of $\tX$,
  \begin{eqnarray*}
  \int _{\tX '} a \wedge  \tdch (\iota _*\cO_{\tX} )  & = &  \int _{\tX} \iota ^* a, \end{eqnarray*}
  where $\iota _*\cO_{\tX}$ is taken as a Koszul factorization for $(\tX ' , \tW ')$ using $Q$ and the regular section.
\end{Cor}
\begin{proof}
In this case the Koszul factorization associated to $\iota _*\cO_{\tX}$ is  transitively related to $\cO_{\tX}$ which has trivial Todd-Chern class.
\end{proof}

\begin{Cor}

Consider maps $e_i: (\tX_i, \tW_i) \to  (\tT, {\tt w})$ between LG models, $i=1, 2$.
Let $\cE_1, \cE_2$ be two perfect TK factorizations supported on closed proper substack $\tZ \subseteq \tX _1 \subseteq \tX_2$. 
Suppose that  $\cE_1, \cE_2$ are transitively related with respect to $ (\tT, {\tt w})$.
Then, for all $a \in \HH^* (\tT, (\Omega ^{\bullet}_{\tT}, - d \tt{w}))$,
  \begin{eqnarray}
  \int _{\tX_1} e_1^* a \wedge   \tdch^{\tX _1}_{\tZ} \cE_1    & = &  \int _{\tX_2} e_2 ^* a  \wedge  \tdch^{\tX_2}_{\tZ} \cE_2.
   \end{eqnarray}
   \label{cor: trans related}
\end{Cor}

Consider the setup of Case (2) of Definition \ref{def:equiv} with $\mathbb K(F, \tau, \sigma)$  supported on a closed 
substack $\tZ$ of $\tX$.  Denote by $\tM $ the deformation to the normal cone of $\tX$ in $\tX '$, by $\pi _{\tM} :  \tM \to \tX '$ the projection, by $\iota _t :  \tM_t \to \tM$  the inclusion of the fiber $\tM_t$ at $t\in \PP ^1$, by
 $p_{Q|_{\tX}}$ the projection $Q|_{\tX} \to \tX$, and by $taut _Q$ the tautological section.
\begin{Lemma}\label{lemma: deform fact} 
There exists a perfect TK factorization  $\mathbb K (F_{\tM}, \tau _{\tM}, \sigma _{\tM})$ on the LG model $(  \tM,   \pi _{\tM}^* \tW ')$
 supported on the closed substack $\tZ \ti \PP ^1$ of $\tM$
such that 
\[ \begin{array}{lll} \iota _0^* F_{\tM} \cong F' & \text{ under which } & 
 \begin{array}{l}  
  \iota _0^* \tau _{\tM} = \tau ' , \\   \iota _0^* \sigma _{\tM} = \sigma ' ; \end{array} \end{array} \]
 and 
 \[ \begin{array}{lll} 
 \iota _{\infty}^* F_{\tM} \cong p_{Q|_{\tX}}^* (F \oplus Q|_{\tX}) &  \text{ under which } & 
 \begin{array}{l}  
  \iota _0^* \tau _{\tM} =  p_{Q|_{\tX}}^*\tau  \oplus 0  ,  \\
 \iota _{\infty}^* \sigma _{\tM} =  p_{Q|_{\tX}}^*\sigma \oplus taut _{Q|_{\tX}} ; \end{array}
\end{array} \]
i.e.,  (by \eqref{eqn:  inner tensor of TK} and Corollary \ref{cor: pullback of vir fact}) there exists a perfect TK factorization $\mathbb K (F_{\tM}, \tau _{\tM}, \sigma _{\tM})$ which deforms
$\mathbb{K} (F', \tau ', \sigma ')$ to $(p_{Q|_{\tX}}^*\mathbb{K}  (F, \tau , \sigma  ) ) \ot \mathbb{K} (p_{Q|_{\tX}}^*Q, 0, taut_Q )$.

\end{Lemma}

\begin{proof}
Let  $\pi _{\PP^1}: {\tX '} \times \PP^1 \ra {\tX '}$ be the projection.
On ${\tX '} \times \PP^1$, we consider a short exact sequence of vector bundles
 \[ \begin{array}{ccccc}  K  & \ra  &  \pi _{\PP^1}^* (F' \oplus Q)      & \ra  &          Q \boxtimes \cO _{\PP ^1} (1)
 \\                               &          &      (v, q)               & \mapsto &     t_0 \phi(v) - t_1 q    
                                \end{array} , \]
                                where $K$ is defined to be the kernel and $t_0, t_1$ are homogeneous coordinates of $\PP ^1$.
                                Let $t= t_0/ t_1$.
                                Note that $K|_{t}$ for $t=\infty$ is isomorphic to $F \oplus Q$
                                and for $t\ne \infty$ is isomorphic to $F'$ by  via the map $v \mapsto (v, t \phi (v))$.
Consider the section $(\sigma ', t \phi \circ \sigma ')$ which is the composition of 
\[ X'\ti \mathbb{A}^1 \xrightarrow{(t\phi \circ \sigma , t)} Q \times \PP ^1 \xrightarrow[\text{closed immersion}]{( \sigma ' \circ p_Q, \mathrm{id}_Q, \mathrm{id}_{\PP ^1})} 
(F'\oplus Q ) \times \PP ^1  \] where $p_Q$ denotes the projection $Q \to X'$.
The closure of the image of the section $(\sigma ', t \phi \circ \sigma ')$ in $\pi _{\PP^1}^* (F' \oplus Q) $ is isomorphic to $\tM$.
In particular,  the fiber of $\tM$ at $t=\infty$  is canonically isomorphic to 
the normal vector bundle $Q|_{\tX}$ of $\tX$ to ${\tX '}$ and the generic fiber  $\tM|_{t\ne \infty}$ is canonically isomorphic to $\tX '$. 
Therefore we have a commutative diagram
 \[ \xymatrix{  
    ( F' \oplus 0 )\ti \{ 0\}            \ar[r]               & (F' \oplus Q ) \ti \PP ^1 & \ar[l] (F \oplus Q |_{\tX}) \ti \{\infty\}  \\
      \tX ' \times \{ 0 \} \ar[r] \ar[d]  \ar[u]_{\sigma '} & \tM\  \ar[u] \ar[d]  &  \ar[l]
  \tM|_{\infty} = Q|_{\tX}  \ar[u]^{(\sigma \circ p_{Q|_{\tX}}, \mathrm{id}_{Q|_{\tX}} ) } \ar[d]  \\
                            {\tX '}\ti   \{ 0 \} \ar[r]  &    {\tX '} \times \PP ^1 \ar[d]_{\pi _{\PP^1}} & \ar[l] \  {\tX} \ti \{ \infty \}  \\
                                                          &      {\tX '} & & 
                                                             } \]

We now construct the sections $\sigma_{\tM}$ of $ K |_{\tM}$ 
and $\tau _{\tM}$ of $\mathrm{Th}^{-1}\God ( \wedge ^{-\bullet} K^\vee |_{\tM} ) $ as follows.
Consider the tautological section  of the vector bundle $(F' \oplus Q)|_{(F' \oplus Q)\ti \PP ^1} $ on the total space $(F' \oplus Q)\ti \PP ^1 $.
Its restriction to $\tM$ factors through a section of $K|_{\tM } $, which we denote by $\sigma_{\tM}$.
The section $ (\pi _{\PP^1}^*  (\tau ', 0)) |_{\tM } $ 
of $\mathrm{Th}^{-1}\God ( \wedge ^{\bullet}  (F'^{\vee}\oplus Q^{\vee}) |_{\tM})$ naturally yields a section of  
$\mathrm{Th}^{-1}\God ( \wedge ^{-\bullet} K^\vee |_{\tM} ) $ via the map
$ \pi _{\PP^1}^* (F'^{\vee} \oplus Q^{\vee}) \to  K^\vee$. Denote this section by $\tau _{\tM}$.
Take $F_{\tM} = K|_{\tM}$. It is straightforward to check the required properties using Lemma \ref{lemma: functor zzeta}.
\end{proof}

\section{The cohomological field theory} \label{section: cohFT}

In \S\ref{sec: state space def}, we defined a $\ZZ_2$-graded vector space  $\State$ together with a nondegenerate pairing  \eqref{def pairing}.    In \S\ref{section:tails}, we will also define a distinguished element $\mathbbm{1} \in \State$.
Furthermore, in \S\ref{sec glsm inv} we defined a homomorphism 
\[  \Omega _{g, r, d} : \State ^{\ot r} \to H^* (\overline{M}_{g, r}) \] 
 for $2g - 2 + r > 0$ and all $d \in \Hom (\widehat{G}, \QQ )$.
Consider the formal power series ring $\CC[[q]]$ and define the following $\CC[[q]]$-linear maps,
 \[ \Omega _{g, r} = \sum_{d} q^d \Omega _{g, r, d}  : \State [[q]] \to H^* (\overline{M}_{g, r}) [[q]].  \]

\begin{Rmk} Let $t$ be another formal parameter.
We may replace $\tdch \mathbf K (\ka, \kb)$ by $\sum t^i \tdch _i \mathbf K (\ka, \kb)$ where 
$\tdch _i \mathbf K (\ka, \kb)$ is the degree $i$-th part of $\tdch \mathbf K (\ka, \kb) \in \oplus_j \HH^j_{\tZ} (\Ugrd, (\Omega ^{\bullet}_{\Ugrd}, -dW))$.
Therefore, we can also defined refined GLSM invariants $\Omega _{g, r} : \State [[q, t]] \to H^* (\overline{M}_{g, r}) [[q, t]]$.
\end{Rmk}

The goal of this section will be to verify that, under mild conditions, 
the data $(\State, \eta, \mathbbm{1}, \{ \Omega _{g, r} \})$ forms a cohomological field theory with unit in the sense of \cite{Pan18}.  

We proceed subsection by subsection, where we describe and verify the $S_r$-covariance, tree gluing, loop gluing, forgetting tails, and the metric axiom respectively.
In fact, the $S_r$-covariance, tree gluing, and loop gluing axioms always hold.

Recall that $d_w$ is the pairing of $\mathbb C^\times_R$ and $\chi$ and define
$c_i$ to be the weights of the $\mathbb C^\times_R$-action i.e.,
   \[ \CR \to GL(V); \  \lambda \mapsto (\lambda ^{c_1} , ..., \lambda ^{c_{\dim V}}) \label{pg: ci} . \]
For the rest of the axioms, we will assume that $\nu$ is large enough, $d_w \ge c_i \ge 0$,
and that the fixed locus of $\CR$ on $\cX$ is  expressible
as $[(V^{ss})^{\CR}/G]$ i.e.,
\begin{equation} \label{eq: fixedloci}
\cX^{\CR} = [(V^{ss})^{\CR}/G] \tag{$\dagger$}.
\end{equation}

We will also consider the following map.  Recall that 
$\zeta := \exp (\pi \sqrt{-1} /\bfdeg ) \in \CC ^\ti _R$.
We define the \emph{twisted diagonal map} by
\[ \Delta_{\zeta} = ( \Delta_{\zeta, 1}, \Delta_{\zeta, 2} ): I\cX  \to (I\cX )^2 , \ (x, h) \mapsto ((x, h) , (\zeta \cdot x, h^{-1})) . \] 
The ``Poincar\'e dual" class of the twisted diagonal is the element
\[ \eta _{\Delta} \in \HH^*(I\cX, (\Omega_{I\cX}, -dw )) \ot \HH^*(I\cX, (\Omega_{I\cX}, -dw )) \]
defined by the pairing \eqref{def pairing}  considered as an element  of $\State ^{\vee}\ot \State^{\vee}$.
Hence, by definition 
\begin{eqnarray}\label{eq:def eta delta} \int _{I\cX \ti I\cX} a \ot b \wedge \eta _{\Delta} = \eta ( a, b ) . \end{eqnarray}
This property is preserved under smooth pullback by Corollary \ref{cor: base change Poincare}.
Finally, we fix a basis $\{ T^h_j \}_j$ of $\State_h$ and denote the dual basis with respect to the pairing $\eta |_{\State_h}$ by $\{ T_h^j \}_j$.

For fixed $h_i \in G$, $i=1, ..., r$,
we denote by $\U(h_1, ..., h_r)$ the open and closed substack of $\U_{g,r,d} := \U$ such that the $i$-th marking lands on $\cX_{h_i}$.
Then, we denote by
$
\mathbb{K}_{g, r, d}(h_1, ..., h_r)
$
the restriction of $\mathbb{K}_{g, r, d}$ to $\U(h_1, ..., h_r)$.

Finally, we also sometimes write $\fB_{\Gamma}^{g, r, d}$ instead of $\fB_{\Gamma}$ when specification is required; see \S\ref{sec:const U}.

\subsection{$S_r$-covariance axiom}
The symmetric group $S_r$ acts on $\State ^{\ot r}$ by permutation and on $H^* (\overline{M}_{g, r})$ by permuting the markings.
Let $\sigma_i \in S_r$ be the permutation which exchanges $i, i+1$.  
The $S_r$-covariance axiom is the equality
\[
(-1)^{|\gamma_i| |\gamma_{i+1}|}  \Omega_{g,r,d}(\gamma_1 \ot \cdots \ot \gamma_r) =  \sigma_i^*\Omega_{g,r,d}(\gamma_1 \ot \cdots \ot \gamma_{i+1} \ot \gamma_i \ot \cdots \ot \gamma_r)
\]
for all $i$, where $\gamma_j$ are homogeneous of degree $|\gamma_j|$ with respect to the $\ZZ_2$-grading on $\State$.

\subsubsection{Proof}\label{proof:covariance}
Let $K$ denote the K\"unneth isomorphism  \eqref{kunneth}.
We have the following facts:
\begin{enumerate}
\item the action of $S_r$ on $\U_{g,r,d}$ induced by permuting the markings makes the evaluation and forgetful maps equivariant;
\item the virtual fundamental class $[\Ugrd]_W^{\vir} $ does not depend on the ordering of the markings, i.e., $\sigma_i^*[\Ugrd]_W^{\vir} = [\Ugrd]_W^{\vir}$;
\item the K\"unneth isomorphism introduces the sign $(-1)^{|\gamma_i| |\gamma_{i+1}|}$ when compairing $\State^{\ot r}$ to $\HH^*((I\cX)^r, \Omega^\bullet_{(I\cX)^r}, dw^{\boxplus r})$.
\end{enumerate}
Hence,
\begin{align*}
& \ (-1)^{|\gamma_i| |\gamma_{i+1}|}  \Omega_{g,r,d}(\gamma_1 \ot \cdots \ot \gamma_r)  &\\
=  & \  (-1)^{|\gamma_i| |\gamma_{i+1}|}  \Omega_{g,r,d}\sigma_i^*(\gamma_1 \ot \cdots \ot \gamma_{i+1} \ot \gamma_i \ot \cdots \ot \gamma_r) & \text{by definition}  \\
 = & \ fgt_*(ev^* \sigma_i^* K ( \gamma_1 \ot \cdots \ot \gamma_{i+1} \ot \gamma_i \ot \cdots \ot \gamma_r) \wedge [\Ugrd]_W^{\vir} ) & \text{by (3)} \\
 = & \ fgt_*(\sigma_i^* ev^*  K ( \gamma_1 \ot \cdots \ot \gamma_{i+1} \ot \gamma_i \ot \cdots \ot \gamma_r) \wedge [\Ugrd]_W^{\vir} ) & \text{by (1)} \\
 = & \ fgt_*(\sigma_i^* ev^*  K ( \gamma_1 \ot \cdots \ot \gamma_{i+1} \ot \gamma_i \ot \cdots \ot \gamma_r) \wedge \sigma_i^* [\Ugrd]_W^{\vir} ) & \text{by (2)} \\
 = & \ \sigma_i^* fgt_*( ev^*  K ( \gamma_1 \ot \cdots \ot \gamma_{i+1} \ot \gamma_i \ot \cdots \ot \gamma_r) \wedge  [\Ugrd]_W^{\vir} )& \text{by (1)} \\
 =  & \  \sigma_i^*\Omega_{g,r,d}(\gamma_1 \ot \cdots \ot \gamma_{i+1} \ot \gamma_i \ot \cdots \ot \gamma_r) & \text{by definition}  \\
\end{align*}

\subsection{Tree gluing axiom}  For $i = 1, 2$, fix $g_i, r_i, g, r, d$ such that $g_1+g_2 =g$, $r_1 + r_2 = r$.
Let $\rho _t :   \overline{M}_{g_1, r_1+1} \times \overline{M}_{g_2, r_2+1}   \to \overline{M}_{g,r}  $ be the gluing map, i.e.,
 the map corresponding to gluing along the final marking of each curve.
We now prove the tree gluing axiom,
\begin{eqnarray}\label{eqn:tree}
 & &  \rho _t ^* \Omega _{g,r, d } (\gamma _1, ..., \gamma _r) =
 \\ \notag    & &    \sum _{ h,  j, d_1 + d_2 = d }  \Omega _{g_1, r_1 +1, d_1} (\gamma _1, ..., \gamma _{r_1}, T^h_j) \ot
 \Omega _{g_2, r_2 +1, d_2} (\gamma _1, ..., \gamma _{r_2}, T_h^j ),
\end{eqnarray}
of our cohomological field theory.

\subsubsection{Setup for the proof}\label{proof:tree}
 Fix $d_1, d_2$ such that $d_1 + d_2 =d$. Choose virtual factorizations
$(\U_{g_i, r_i+1, d_i}, \mathbb{K}_{g_i, r_i+1, d_i}) $  and $(\U_{g,r, d}, \mathbb{K}_{g, r, d})$ as in \S\ref{sub:virtual factorization}.

Let $ \prod '_{i} \fB_{\Gamma}^{g_i, r_i+1, d_i} $ denote the stack parameterizing 
triples 
\begin{equation}\label{eq: disjoint pair}
((C^1, P^1, \kappa ^1), (C^2, P^2, \kappa ^2), (\phi, \tilde{\phi}_{\Gamma}))
\end{equation}
 of objects in $\fB_{\Gamma}^{g_1, r_1+1, d_1}, \fB_{\Gamma}^{g_2, r_2+1, d_2}$ together with a pair of isomorphisms $\phi : \sG_{r_1+1} \to \sG_{r_2+1}$ inverting the band,  
 $\tilde{\phi}_{\Gamma}: P^1|_{\sG_{r_1+1}} \to P^2|_{\sG_{r_2+1}}$ as principal bundles on $\sG_{r_1+2} = \phi(\sG_{r_1+1})$.  By an argument similar  to \cite[\S 5.1]{AGV}, 
  the stack  $ \prod '_{i} \fB_{\Gamma}^{g_i, r_i+1, d_i} $ is a smooth algebraic stack which is locally of finite type over $\Spec\/\CC$.

We use the following shorthand, $\U_{d_i} := \U_{g_i, r_i+1, d_i}$ and $\U :=\U_{g, r, d}$.  Let $\U_{d_1}\ti \U_{d_2} \to (I\cX)^2$ be the product $ev_{r_1+1}\ti ev_{r_2+1}$ of evaluation maps,
let $$\rho _{t, d_1, d_2} :    \prod'_i \fB_{\Gamma}^{g_i, r_i+1, d_i}  \to \fB_{\Gamma}^{g, r, d}$$ be the gluing map, which is a DM type morphism, 
and let $LG_{d_i} :=LG_{g_i, r_i+1, d_i} (\cX)$, $LG_{g, r, d} := LG_{g, r, d}(\cX )$. Then
we have the following commuting diagram 
\begin{equation} \label{eq: fiber ev}\xymatrix{    
& LG_{d_1}\ti LG_{d_2} \ar@{^{(}->}[r] & \U_{d_1}\ti \U_{d_2}  \ar[r] & (I\cX)^2 \\
&  \Delta^*_{\zeta} (LG_{d_1} \ti LG_{d_2})  \ar@{^{(}->}[r]  \ar[ld]^{\cong} \ar[u]  &  \Delta ^*_{\zeta} (\U_{d_1}\ti \U_{d_2})   \ar[u]_{\tilde{\Delta}_{\zeta}}  \ar[r]  
                                                    & I\cX \ar[u]_{\Delta _{\zeta}} \\
\rho^*_{t, d_1, d_2}LG_{g, r, d} \ar@{^{(}->}[r]  &      \rho^*_{t, d_1, d_2}  \U \ar[r]_{\mathrm{bl}} \ar[d]     \ar@{..}[ru]_{\sim}  \ar@/_1.5pc/[rr]_{  \tilde{\rho}_{t, d_1, d_2} \ \ \ \ \ \ \ \ \ \  }
&               \rho_t^* \U  \ar[r]_{\tilde{\rho}_t} \ar[d]    &  \ar[d] \U  \ar@/^1.5pc/[dd]^{fgt}\\
 &  \prod'_i \fB_{\Gamma}^{g_i, r_i+1, d_i}  \ar[r]  \ar@/_1.5pc/[rr]_{  \rho_{t, d_1, d_2} \ \ \ \ \ \ \ \ \ \  }
  &            \rho_t^*\fB_{\Gamma}^{g, r, d}   \ar[r] \ar[d] & \fB_{\Gamma}^{g, r, d} \ar[d] \\
 &   &  \prod _i  \overline{M}_{g_i, r_i+1}  \ar[r]_{\ \   \ \rho_{t}} & \overline{M}_{g,r}      }
\end{equation}
where all squares are defined to be fiber products and
 \[
 \rho ^*_{t, d_1, d_2}LG_{g, r, d}  := \Pi '_i\fB_{\Gamma}^{g_i, r_i+1, d_i} \ti _{\fB_{\Gamma}^{g, r, d}} LG_{g, r, d}.
\]

We will show in Lemma~\ref{lem: tree related} 
that the  stacks $ \Delta^*_{\zeta} (LG_{d_1} \ti LG_{d_2}),   \rho^*_{t, d_1, d_2}LG_{g, r, d}$ are isomorphic  and 
that the stacks  $\Delta ^*_{\zeta} (\U_{d_1}\ti \U_{d_2})$ and $\rho^*_{t, d_1, d_2}  \U  $ carry transitively related factorizations  (indicated by the symbol $\sim$) with respect to 
\[
( (I\cX)^{r_1+r_2}, -\boxplus w) \ti (\prod _i \overline{M}_{g_i, r_i+1} , 0 ).
\]

Consider a generic point of $\rho_t^* \U$. Its domain curve has two irreducible components and the principal $\Gamma$-bundle restricted to each component carries a degree.
Hence we can decompose 
\begin{equation} \label{eq: U decomp}
\rho_t ^*\U  = \bigcup _{d_1+d_2 = d} (\rho_t ^*\U)_{d_1, d_2}
\end{equation}
 as a union of closed substacks.

\subsubsection{Proof of \eqref{eqn:tree}}
Now, let $\gamma := \gamma_1 \ot ... \ot \gamma _r \in \State ^{\ot r}$ and 
$\delta \in \ot _i H^*(\overline{M}_{g_i, r_i+1})$.  Abusing notation for evaluation and forgetful maps from various domains, we claim  the following sequence of equalities
\footnotesize
\begin{align*} 
   \nonumber     &   \langle \rho _t ^* \Omega _{g,r, d } (\gamma _1, ..., \gamma _r), \delta  \rangle \\
 \stackrel{(1)}{=}  &
\int _{\rho _t^* \U}  \prod {\bf r}_i\cdot \tilde{\rho}_t^* \left(\tdchc \mathbb{K}_{g, r, d}  \wedge ev^*( \gamma )\wedge fgt ^*(\delta)  \right) \\ 
   \stackrel{(2)}{=}  &  \sum _{d_1 + d_2 =d}  \int _{(\rho _t^* \U)_{d_1, d_2} } \prod {\bf r}_i\cdot 
    \tilde{\rho}_t ^* \left( \tdchc  \mathbb{K}_{g, r, d}   \wedge ev^*( \gamma)\wedge fgt ^*(\delta) \right)  |_{(\rho _t^* \U)_{d_1, d_2}}   \\
     \stackrel{(3)}{=}  &  \sum _{d_1 + d_2 = d}   \int _{\rho ^*_{t, d_1, d_2 } \U }  \prod {\bf r}_i\cdot  {\bf r}^2 \cdot \tilde{\rho}^*_{t, d_1, d_2 } (\tdchc  \mathbb{K}_{g, r, d} )    \wedge ev^*( \gamma)\wedge fgt ^*(\delta) \\
    \stackrel{(4)}{=}    &  \sum _{d_1 + d_2 = d}    \int _{\Delta ^*_{\zeta} (\U_{d_1}\ti \U_{d_2})}   \prod {\bf r}_i\cdot  {\bf r}^2 \cdot \tilde{\Delta}^*_{\zeta} \tdchc ( \boxtimes_i \mathbb{K}_{g_i, r_i+1, d_i} )   \wedge ev^*( \gamma)\wedge fgt ^*(\delta) \\
    \stackrel{(5)}{=}  & \sum _{d_1 + d_2 = d}   \int _{\prod _i \U_{g_i, r_i+1, d_i}  } \prod {\bf r}_i\cdot {\bf r}^2 \cdot \tdchc ( \boxtimes_i \mathbb{K}_{g_i, r_i+1, d_i} ) \wedge ev^*( \gamma)\wedge fgt ^*(\delta) \wedge ev^* \eta _{\Delta} \\
  \stackrel{(6)}{=}   &    \sum _{ h,  j, d_1 + d_2 = d }  \langle \Omega _{g_1, r_1 +1, d_1} (\gamma _1, ..., \gamma _{r_1}, T^h_j) \ot
 \Omega _{g_2, r_2 +1, d_2} (\gamma _1, ..., \gamma _{r_2}, T_h^j ) , \delta \rangle.
   \end{align*}
   \normalsize
  
   In the first equality, 
  note that $\rho _t^* \U$ and $(\rho _t^* \U)_{d_1, d_2}$ 
  could be singular. However we have the topological trace map, given in \S \ref{sub: top def of trace}, which we are using for the definition 
   of $\int _{\rho _t^* \U}$ and the integrands are considered as elements in $H^*_{ \rho _t^* \Z} (    \rho _t^* \U     )$ and 
   $H^*_{ (\rho _t^* \Z )_{d_1, d_2}}  ( (\rho _t^* \U)_{d_1, d_2}) $,  respectively. 
   The topological trace map is compatible with the integration map in the smooth case; see \eqref{eqn: com of integration maps}. 
   The map $fgt$ in the figure is a flat map (see \cite[Proposition 3]{Beh}) and hence we can apply the base change  Lemma \ref{lemma: base change of Gysin maps}
   to obtain the first equality.   

The second equality follows from the cycle level decomposition induced by \eqref{eq: U decomp}.
  The third equality is due to the projection formula \eqref{eqn:ProjFor} 
  and  the fact that the map $ \mathrm{bl}:  \rho_{t, d_1, d_2}^*\U   \xrightarrow{} (\rho_t^*\U)_{d_1, d_2}$ has degree $1/{\bf r}^2$.
    The fourth equality follows from Lemma  \ref{lemma: functor tdch}, 
    Lemma~\ref{lem: tree related} (2), Corollary \ref{cor: pullback of vir fact}, and Corollary~\ref{cor: trans related}.
   The fifth equality is by \eqref{eq:def eta delta} and Corollary \ref{cor: base change Poincare}.  The final line  is by Corollary \ref{cor: multi of ch}, Lemma \ref{lemma: functor tdch}, and Fubini's theorem.

\subsubsection{A lemma}

\begin{Lemma} \label{lem: tree related} The following hold.
\begin{enumerate}
\item
The stacks $ \Delta^*_{\zeta} (LG_{d_1} \ti LG_{d_2}) $ and $ \rho ^*_{t, d_1, d_2}LG_{g, r, d}$ are canonically isomorphic.
\item
The perfect TK factorizations corresponding to
 \[
 \tilde{\Delta}_{\zeta}^* ( \mathbb K _{g_1, r_1+1, d_1} \boxtimes \mathbb K _{g_2, r_2 +1, d_2} )  \text{ and } 
 \tilde{\rho}_{t, d_1, d_2} ^* \mathbb{K}_{g, r, d}
 \]
  are transitively related with respect to 
  \[
  ((I\cX)^{r_1+r_2}, -\boxplus w) \ti (\prod _i \overline{M}_{g_i, r_i+1} , 0 ).
  \]
\end{enumerate}
\end{Lemma} 

\begin{proof} 
(1) The argument is very similar to \cite[Proposition 5.2.2]{AGV}.

 Let $T\to  \Delta^*_{\zeta} (LG_{d_1} \ti LG_{d_2}) $ be a map from a  test scheme $T$. The induced map $T \to \prod _i LG_{g_i, r_i+1, d_i}$ 
 amounts to choosing a pair of families of stable LG quasimaps $((C_1, P_1, \kappa _1, u _1)$ over $T$ (allowing the notion of disconnected stable maps) 
 where $(C_1, P_1, \kappa _1)$ is the disjoint union of the corresponding part of the pair \eqref{eq: disjoint pair} (so $C_1 = C^1 \coprod C^2$, $P_1 = P^1 \coprod P^2$).
 After  reordering the markings so that the last two markings are the special ones, let $\sG_j$ be the  $j$-th gerbe marking of $C_1$ for $j=1, ..., r+2$.
For $i=1, 2$, let $\bar{P}_{r+i}$ be the canonical $G$-bundle reduction 
$1^*P_1|_{\sG_{r+i}}$ of $P_1|_{\sG_{r+i}}$ by the canonical section $1$ of  the $\CC^{\ti}$-bundle 
$$P_1 / G  |_{\sG_{r+i}} \stackrel{\kappa_1}{\cong} \omega ^{\ti}_{C_1}  |_{\sG_{r+i}} = \CC^{\ti} \ti _{\Spec \CC} \sG_{r+i}$$ on $\sG_{r+i}$.
Here  $\omega ^{\ti}_{C_1}$ indicates the principal bundle attached to the line bundle $\omega_{C_1}$ by taking nonvanishing sections.

Unwinding the definition of the fiber product, we see that the map $T$ to $ \Delta^*_{\zeta} (LG_{d_1} \ti LG_{d_2})$  exactly amounts to a pair 
\begin{equation}
(C_1, P_1, \kappa _1, u_1) \ti (\phi , \tilde{\phi}_G)
\label{eq: fiber data}
\end{equation}
where
\begin{itemize}
\item $\phi: \sG_{r+1} \xrightarrow{\cong} \sG_{r+2}$ inverts the band;
\item $\tilde{\phi}_G:  \bar{P}_{r+1} \xrightarrow{\cong} \phi^* \bar{P}_{r+2}$ such that
 $$\zeta \tilde{\phi}_G (u_1|_{\sG_{r+1}} )  = \phi^*( u_1|_{\sG_{r+2}} ) $$ as elements in 
 $ \Gamma (\sG_{r+1}, \phi^* \bar{P}_{r+2}(V))$. 
\end{itemize}

Note that  $P _1|_{\sG_{r+1}} \cong \bar{P}_{r+1}  \ti _G \Gamma  
\cong \bar{P}_{r+2} \ti _G \Gamma  \cong P _{1}|_{\sG_{r+2}} $ via $\kappa _1$ and $\tilde{\phi}_G$ (see \S \ref{pg: P(Y)} for the notation).
Denote by $\kappa _1\diamond \tilde{\phi}_G \diamond \kappa _1$ the resulting isomorphism $P _1|_{\sG_{r+1}} \cong  P _1|_{\sG_{r+2}} $
and let $\tilde{\phi}_{\Gamma} :=\zeta ^{-1} \circ \kappa _1\diamond \tilde{\phi}_G \diamond \kappa _1$.
This gives a diagram
\begin{equation}  \label{diag: gluing}      \xymatrix{ P_1|_{\sG_{r+1}} 
\cong_{ \tilde{\phi}_{\Gamma} } P_1|_{\sG_{r+2}} \ar@<.5ex>[r] \ar@<-.5ex>[r] \ar[d] & P_1  \ar[r]  \ar[d]&  P_2 \ar[d] \\
\sG_{r+1} \cong_{\phi} \sG_{r+2} \ar@<.5ex>[r] \ar@<-.5ex>[r]
& C_1  \ar[r]^{\epsilon}   \ar[d]& C_{2} \ar[dl] \\
         &   S & }  \end{equation}
where $C_{2}, P_{2}$ are pushouts, i.e.,   
$C_{2}, P_{2}$ are the curve (resp.\ principal $\Gamma$-bundle) obtained from gluing the two curves (resp.\ two principal bundles) along the last two markings via $\Delta _{\zeta}$.
 Furthermore, there is also a pushout diagram 
\begin{equation*}
 \xymatrix{ \omega^{\log}_{C_1}|_{\sG_{r+1}} = \omega^{\log}_{C_1}|_{\sG_{r+2}} \ar@<.5ex>[r] \ar@<-.5ex>[r]  & \omega^{\log}_{C_1} 
 \ar[r]  &  \omega^{\log}_{C_{2}} }
 \end{equation*}
 as total spaces of vector bundles.
     Hence, we may glue $\kappa_1$ at $\sG_{r+i}$, yielding 
         \[
         \kappa_{2} : P_{2} (\CC_{\chi}) \cong \omega ^{\log}_{C_{2}}.
         \]     
         Finally we glue $u_1$ along $\sG_{r+1} = \sG_{r+2}$ to obtain  a unique global section $u_{2}$ of $\cV_{2} := P_2 (V)$ whose
         restriction to $C_i$ coincides with $u_i$ for $i=1, 2$.        
    Hence from $T\to  \Delta _{\zeta}^* (LG_{d_1} \ti LG_{d_2} )$  
    we obtain the pair 
    \begin{equation}\label{eq:obj rho LG} 
    \left( (C_1, P_1, \kappa _1) \ti (\phi , \tilde{\phi}_{\Gamma} )\right) \ti
    (C_{2}, P_{2}, \kappa_{2}, u_{2} ), 
    \end{equation} 
    which exactly amounts to a map $T \to  \rho ^*_{t, d_1, d_2}LG_{g, r, d}$. 
    
      Conversely starting from a map $T \to  \rho ^*_{t, d_1, d_2}LG_{g, r, d}$, i.e., \eqref{eq:obj rho LG} with the diagram \eqref{diag: gluing}, 
    one can simply take the data 
    $$(C_1, P_1, \kappa_1, u_1:=\epsilon  ^* u_2) \ti (\phi , \tilde{\phi}_G := 1^* (\zeta  \circ \tilde{\phi}_{\Gamma}) ) ,$$
     i.e., a map
       $T \to  \Delta^*_{\zeta} (LG_{d_1} \ti LG_{d_2}) $ (see \eqref{eq: fiber data}).
      
      It is straightforward to check that these constructions give rise to equivalences of groupoids yielding an isomorphism of stacks 
      \[
      \Delta^*_{\zeta} (LG_{d_1} \ti LG_{d_2}) \cong  \rho ^*_{t, d_1, d_2}LG_{g, r, d}.
      \]
      
      (2) By Theorem \ref{thm: cat glue and tail} applied to the case where $S := \prod' \fB_{\Gamma}^{g_i, r_i+1, d_i}$ 
      with the projection $q: S \to \fB_1 := \prod \fB_{\Gamma}^{g_i, r_i+1, d_i} $ and the gluing map $\rho_{t, d_1, d_2}: S \to \fB_2:=\fB ^{g, r, d}_{\Gamma}$, 
     the following factorizations are transitively related $ \iota ^* \mathbb{K}_{S \to \fB_1} \sim   \mathbb{K}_{S \to \fB_2}$ (see the notation from \S \ref{sub:virtual factorization}
      and Theorem \ref{thm: cat glue and tail}). 
      Hence it is enough to 
      show that  
       \begin{equation}\label{eqn: outer tensor app} 
      \mathbb{K}_{\fB_1} \underset{cqis}{\simeq}  \mathbb{K}_{\fB_{\Gamma}^{g_1, r_1+1, d_1}} \boxtimes \mathbb{K}_{\fB_{\Gamma}^{g_2, r_2+1, d_2}} ;
      \end{equation} 
      \begin{equation}\label{eqn: comparison constr-pullback and pullback-constr} 
      \tilde{\Delta} ^*_{\zeta } \mathbb{K}_{\fB_1} 
      \sim  \iota ^* \mathbb{K}_{S \to \fB_1} ;
      \end{equation} 
      and 
       \begin{equation}\label{eqn: comparison constr-pullback and pullback-constr 2} 
       \mathbb{K}_{S\to \fB_2} 
      \sim   \tilde{\rho}_{t, d_1, d_2} ^* \mathbb{K}_{\fB_2} .
      \end{equation}

      The equation \eqref{eqn: outer tensor app}  is  from \eqref{eqn: outer tensor of TK}.
      The equation \eqref{eqn: comparison constr-pullback and pullback-constr 2} is from Corollary \ref{cor: pullback of vir fact}.
      To show 
      \eqref{eqn: comparison constr-pullback and pullback-constr}, 
      we consider Proposition \ref{prop:U} (1) over $\fB_1$ to get $\cA_{\fB_1} \to \cB_{\fB_1}$ and restriction maps
      $\pi_* (rest_{r+i}) : A_{\fB_1} \to \pi_* (\cV|_{\sG_{r+i}})$. 
      We use those data to get virtual factorizations $\mathbb{K}_{\fB_1}$,  $\mathbb{K}_{S \to \fB_1}$ over $S$ and  $S\to \fB_1$, respectively. Note that
       the total space $A_{\fB _1, 0}$  of the kernel of $\pi_* (rest_{r+1}) - \zeta \pi_* (rest_{r+2})$ is isomorphic to the total space $A_{S, 0}$
       of the kernel of $q^* (    \pi_* (rest_{r+1}) - \zeta \pi_* (rest_{r+2})     )$ under the map $ \tot (q^* A_{\fB_1}) \to \tot A_{\fB_1}$. 
       We have fiber products
                                \[ \xymatrix{  
       \fB _1 & \ar[l] A_{\fB_1}           
       &   A_{\fB_1, 0} \ar[l]_{\tilde{\Delta}_{\zeta}}        \\
                             S \ar[u]^{q} & \ar[l] A_{S} \ar[u] & \ar[l]^{\iota  }  A_{S, 0}  \ar[u]_{\cong} . } \]
      Letting 
      $B_S : = q^* B_{\fB_1}$,  under the isomorphism $ A_{S, 0}  \cong A_{\fB_1, 0} $  we have
        $\tilde{\Delta} ^*_{\zeta } \mathfrak a _{A_{\fB_1}} \cong \iota ^* \mathfrak a _{A_S} $, $\tilde{\Delta} ^*_{\zeta } \beta _{A_{\fB_1}} \cong \iota ^* \beta _{A_{S}}$,
        which prove \eqref{eqn: comparison constr-pullback and pullback-constr} by Corollary \ref{cor: pullback of vir fact}.
    \end{proof}

\subsection{Loop gluing axiom}
Let $\rho _l :   \overline{M}_{g-1, r+2}  \to \overline{M}_{g,r}  $ be the gluing map of the last two markings.
In this subsection we show the following loop gluing axiom:
\begin{eqnarray}\label{eqn:loop}
   \rho _l ^* \Omega _{g,r, d } (\gamma _1, ..., \gamma _r) =   \sum _{[h]\in G/G, j}   \Omega _{g-1, r +2, d} (\gamma _1, ..., \gamma _{r_1}, T^h_j, T_h^j ) .
\end{eqnarray}

\subsubsection{Proof}

Starting with 
$(\U_{g, r, d}, \mathbb{K}_{g, r, d}) $ and $(\U_{g-1,r+2, d}, \mathbb{K}_{g-2, r+2, d})$ satisfying Lemma \ref{prop:U} 
let us
consider the following diagram
\begin{equation} \label{eq: loop fiber ev}\xymatrix{    
& LG_{g-1, r + 2, d} (\cX) \ar@{^{(}->}[r] & \U_{g-1, r+2, d}  \ar[r]  & (I\cX)^2 \\
 & \Delta _{\zeta}^* LG_{g-1, r + 2, d} (\cX) \ar@{^{(}->}[r]  \ar[u] \ar[ld]^{\cong} &   \Delta _{\zeta}^* \U_{g-1, r+2, d} \ar[u]_{\tilde{\Delta} _{\zeta}}  \ar[r]  & I\cX \ar[u]_{\Delta _{\zeta}} \\
 \rho_{l, \tilde{\fB} }^{* }LG_{g, r, d} (\cX) \ar@{^{(}->}[r] 
  &      \rho_{l, \tilde{\fB} }^{* } \U_{g, r, d}  \ar[r]_{\mathrm{bl}} \ar[d]     \ar@{..}[ru]_{\sim} \ar@/_1.5pc/[rr]_{  \tilde{\rho}_{l, \tilde{\fB} } \ \ \ \ \ \ \ \ \ \ } 
                                &              \rho _l ^* \U_{g, r, d}  \ar[r]_{\tilde{\rho}_{l}}\ar[d]    &  \ar[d] \U_{g,r, d}   \\
   & \tilde{\fB}_{\Gamma}^{g-1, r+2, d} \ar[r] \ar@/_1.5pc/[rr]_{  \rho_{l, \tilde{\fB} } \ \ \ \ \ \ \ \ \ \ } &            \rho_l^*\fB_{\Gamma}^{g, r, d}  \ar[r]\ar[d] & \fB_{\Gamma}^{g, r, d} \ar[d] \\
   &  & \overline{M}_{g-1, r+2}  \ar[r]_{\ \    \rho_{l}} & \overline{M}_{g,r}      }
\end{equation}
where:
\begin{itemize}
\item all squares are defined to be fiber products;
\item $\tilde{\fB}_{\Gamma}^{g-1, r+2, d}$ is the smooth stack parameterizing 
$(C, P, \kappa)$ objects of $\fB_{\Gamma}^{g-1, r+2, d}$ together with pairs $(\phi, \tilde{\phi}_{\Gamma})$ where $\phi : \sG_{r+1} \to \sG_{r+2}$ is an isomorphism which inverts the band and
 $\tilde{\phi}_{\Gamma}: P|_{\sG_{r+1}} \to P|_{\sG_{r+2}}$ is an isomorphism of principal bundles on $\sG_{r+2} = \phi(\sG_{r+1})$;
\item  $\rho_{l, \tilde{\fB} }$ denotes the DM type morphism $\tilde{\fB}_{\Gamma}^{g-1, r+2, d} \to  \fB_{\Gamma}^{g, r, d} $ which glues the principal bundles along the last two markings;
\item and  $\rho_{l, \tilde{\fB} }^{* }LG_{g, r, d} (\cX) :=  LG_{g, r, d} (\cX)  \ti _{\fB_{\Gamma}^{g, r, d} } \tilde{\fB}_{\Gamma}^{g-1, r+2, d}$.
\end{itemize}

Since the rest of proof for  \eqref{eqn:loop}  is parallel to the proof in \S\ref{proof:tree}, we omit this.

\subsection{Forgetting tails axiom}\label{section:tails}
For $2g-2+r >0$, let $\rho_f: \overline{M}_{g, r+1} \to \overline{M}_{g, r}$ be the map which forgets the last marking.
We will define a unit $\mathbbm{1} \in \State$
and prove the forgetting tails axiom
\begin{eqnarray}\label{axiom:tails}  \rho_f ^* \Omega _{g , r, d} (\gamma _1, ..., \gamma _r) = \Omega _{g, r+1, d} (\gamma _1, ..., \gamma _r, \one )  \end{eqnarray}
for $\infty$-stability (to make sure that the forgetting tail map $\U_{g, r+1, d} \to \U_{g, ,r, d}$ exists).  
Furthermore, as in the introduction to the section, we assume that 
$$d_w\ge c_i \ge 0$$ for every $i$ and that the fixed locus of $\CR$ on $\cX$ is expressible
as $[(V^{ss})^{\CR}/G]$, i.e.,
\begin{equation}
\cX^{\CR} = [(V^{ss})^{\CR}/G] \tag{$\dagger$}.
\end{equation}

\begin{example} \label{ex: Rchargechoice}
Returning to Example~\ref{ex: LG/CY},
the $\CC^{\ti}_R$-fixed loci of $\cX$ are
 $\cX_+ ^{\CC^{\ti}_{R,\pm}} = \PP^4$ and  $\cX_- ^{\CC^{\ti}_{R,\pm}} = B\ZZ_5$.
However $V^{\CC^{\ti}_{R,-}} \cap V^{ss} (\nu_+)$ and  $V^{\CC^{\ti}_{R,+}} \cap V^{ss} (\nu_-)$ are empty
while $V^{\CC^{\ti}_{R,+}} \cap V^{ss} (\nu_+)$ and  $V^{\CC^{\ti}_{R,-}} \cap V^{ss} (\nu_-)$ are nonempty.  Hence $\mathbb{X}_{\pm}$ satisfies \eqref{eq: fixedloci} with $\CC^{\ti}_{R, \pm}$ but does not satisfy \eqref{eq: fixedloci} with the reversed R-charge.  This is consistent with the choices which make these examples into a convex hybrid models.
\end{example}

 \begin{Rmk}
On the other hand, the GLSM invariants themselves are independent of the choice of R-charge in the following sense.  Namely,
two choices of $\CR$ provide two splitting maps 
\[
\begin{tikzcd}
0 \ar[r] &  \widehat{\CC^{\ti}} \ot \QQ \ar[r, "{\chi ^{\vee}}"] & \widehat{\kG}\ot \QQ  \ar[r, "p"] \ar[l, dashed, bend right, swap, "{s_1, s_2}"]& \widehat{G}\ot \QQ \ar[r] & 0.
\end{tikzcd}
\]
Let $\fB_{\kG, s_i}^{g, r, d_i}$ denote the corresponding moduli space with the corresponding R-charge and degree $d_i$ also depending on $i$. Then for a geometric point $P$ of  $\fB_{\kG, s_i}^{g, r, d_i}$
and arbitrary $\delta \in \widehat{\kG}\ot \QQ $, we can calculate the degree as follows 
\[
\deg P (\CC_{\delta})= d_i (p(\delta)) + s_i (\delta ) (2g - 2 + r).
\] 
Hence if 
 \begin{equation}\label{eq: degree change} d_2 (p(\delta )) = d _1 (p(\delta )) + (s_1 (\delta ) - s_2 (\delta ) ) (2g - 2 +r ), \forall \delta \in \widehat{\Gamma}\ot \QQ, \end{equation}then
  $\fB_{\kG, s_1}^{g, r, d_1} = \fB_{\kG, s_2}^{g, r, d_2}$ and vice versa.
Therefore the corresponding GLSM invariants satisfy
\[
\Omega^{s_1}_{g,r,d_1} = \Omega^{s_2}_{g,r,d_2}.
\]
\end{Rmk}

\subsubsection{Unit} Recall that $\cX_J : = [(V^{ss})^J / G ]$; see  \eqref{eqn: def of X_h}.  Let $\cXo \subseteq \cX_J$ denote the fixed locus of $\CR$, i.e.,
$\cXo = [(V^{ss})^{\CR} / G ]$.  
The unit $\mathbbm{1}$ of our cohomological field theory is essentially the Poincar\'e dual of the class of $\cXo$. The precise definition is as follows.  
The inclusion map $\iota : (\cXo, 0) \to (\cX_J, w_J := w|_{\cX_J} )$ is an LG map since $w_J |_{\cXo} = 0$ (because $w_J$ is homogeneous of positive degree with respect to the $\CR$-action).
  We define $$S_{\mathbf{1}} :=  \iota_* \cO_{\cX _0}$$ 
  as a factorization of $(\cX_J, w_J)$.

Note that by Euler's Homogeneous Function Theorem
         \begin{equation}
         \label{eq: Euler}
          w_J = \frac{1}{d_w} \sum _{\{  i \, | \, \frac{ c_i }{  d_w } \in \ZZ, \, c_i \ne 0 \} } c_i x_i \partial_i w_J 
          =   \sum _{\{  i \, | \, \frac{ c_i }{  d_w } \in \ZZ, \, c_i \ne 0 \} } x_i \partial_i w_J .
          \end{equation}
Letting  $V_J$ be the fixed locus of $V$ under the action of $\lan J\ran$,
  we have a decomposition 
$V_J= V^{\CR} \oplus V_J^m$, where $V_J^m$  is the moving part of the $\CR$-space $V_J$ (i.e., the sum of the nontrivial eigenspaces).

Consider the section $ dw_J$ of $\Omega^1_{\cX _J } =  [(V^{ss}_J \ti V^{\vee}_J) / G] $. 
By \eqref{eq: Euler}, it factors through a section $s$ of $[(V^{ss}_J \ti (V_J^m)^{\vee}) / G]$.
Also by~\eqref{eq: Euler},  the tautological section $q \circ \text{taut}$ of $[(V^{ss}_J \ti V_J^m)/G]$ on $\cX_J$ pairs with $s$ to $w_J$. 
Therefore we get a locally-free coherent resolution of $S_{\mathbf 1}$ by the Koszul factorization $\{ s, q \circ \text{taut} \}$, i.e.,
\[
S_{\mathbf 1} \cong \{ s, q \circ \text{taut} \}.
\]
Now using this expression we define the unit of our cohomological field theory as:
\begin{eqnarray}\label{eqn: one} \mathbbm{1} : = \td \ch ( S_{\mathbf{1}} ) \in \State _{J} \subset \State . \end{eqnarray}

\subsubsection{Setup for the proof} 
For this proof we abbreviate the notation $LG_{g, r, d} (\cX)$ and simply use $LG_{g, r, d}$.  
Let $\U_{g,r+1, d} (J)$ (resp. $LG_{g, r+1, d}(J)$) denote the open and closed substack of $\U_{g,r+1, d}$ (resp. 
$LG_{g, r+1, d}$)  such that the last marking has type $J$.
We begin by choosing
$(\U_{g, r, d} ,  \mathbb{K}_{g, r, d} ) $ and $(\U_{g,r+1, d} ( J), \mathbb{K}_{g-2, r+2, d} (J) ) $ satisfying Proposition~\ref{prop:U}. 

In this section we will construct the following diagram
\begin{equation} \label{eq: forgetting tails fiber ev}\xymatrix{    
& LG _{g, r+1, d}(J) \ar@{^{(}->}[r] & \U_{g, r+1, d} (J)  \ar[r]_{\ \ \ ev_{r+1}}  & \cX_J \\
 & LG_{g, r+1, d}(J)_{0} \ar@{^{(}->}[r]  \ar[u] \ar[ld]^{\cong}_{\rho_{f, \fB, LG}} &   \U_{g, r+1, d}(J)_0 \ar[u]_{\tilde{\iota} }  \ar[r]  & \cX_0 \ar[u]_{ \iota   } \\
 \rho_{f, \fB }^{* }LG_{g, r, d} \ar@{^{(}->}[r]  &      \rho_{f, \fB }^{* }\U_{g, r, d}  \ar[r]_{\mathrm{bl}} \ar[d]     \ar@{..}[ru]_{\sim} \ar@/_1.5pc/[rr]_{\tilde{\rho} _{f, \fB} \ \ \ \ \ \ \ \ \ \ }
                                &              \rho _f ^* \U_{g, r, d}  \ar[r]_{\tilde{\rho}_{f}}\ar[d]    &  \ar[d] \U_{g,r, d}   \\
   &  \fB_{\Gamma}^{g, r+1, d} (J) \ar[r] \ar@/_1.5pc/[rr]_{\rho_{f, \fB } \ \ \ \ \ \ \ \ \ } &            \rho_f^*\fB_{\Gamma}^{g, r, d}  \ar[r]\ar[d] & \fB_{\Gamma}^{g, r, d} \ar[d] \\
   &  & \overline{M}_{g, r+1}  \ar[r]_{\ \    \rho_{f}} & \overline{M}_{g,r}      }
\end{equation}
where:
\begin{itemize} 
\item all squares are defined to be fiber products;
\item $\fB_{\Gamma}^{g, r+1, d} (J) $ is the smooth stack parameterizing 
objects $(C, P, \kappa)$  of $\fB_{\Gamma}^{g, r+1, d}$ such that the last marking is of type $J$;
\item  the DM type morphism $\rho_{f, \fB } :  \fB_{\Gamma}^{g, r+1, d} (J) \to \fB_{\Gamma}^{g, r, d}$ will be constructed in the lemma below using Hecke modifications;
\item and $\rho_{f, \fB }^{* }LG_{g, r, d} := LG_{g, r, d} \ti _{\fB_{\Gamma}^{g, r, d}} \fB_{\Gamma}^{g, r+1, d} (J)$.
\end{itemize}

\subsubsection{Proof of \eqref{axiom:tails}}\label{sub: proof of axiom tails}
For any $\delta \in H^*(\overline{M}_{g,r+1})$ we claim that there is  a sequence of equalities,
 \begin{align*}
&\  \langle  \Omega _{g, r+1, d} (\gamma _1, ..., \gamma _r, \one ), \delta \rangle \\
 \stackrel{(1)}{=} & \ \int _{ \U_{g, r+1, d} (J) } \prod_{i=1}^{r+1}  {\bf r}_i 
 (fgt, \prod _{i=1}^r ev_i)^* (\gamma\ot \delta) \wedge \tdchc \mathbb{K}_{g, r+1, d} (J)  \wedge ev_{r+1}^*\tdch S_{\mathbf{1}}  \\
   \stackrel{(2)}{=}  & \  \int _{ \U_{g, r+1, d}(J)_0 }   \prod_{i=1}^{r+1}  {\bf r}_i  (fgt, \prod _{i=1}^r ev_i)^* (\gamma\ot\delta) \wedge
    \tilde{\iota}^* \tdchc \mathbb{K}_{g, r+1, d} (J) \\
   \stackrel{(3)}{=}  & \ d_w \int _{\rho ^*_{f, \fB} \U_{g, r, d} }   \prod_{i=1}^{r}  {\bf r}_i  (fgt, \prod _{i=1}^r ev_i)^* (\gamma \ot \delta) 
   \wedge  \tilde{\rho}_{f, \fB}^* \tdchc \mathbb{K}_{g, r, d}  \\
     \stackrel{(4)}{=}   & \  \int _{\rho ^*_{f} \U_{g, r, d} }   \prod_{i=1}^{r}  {\bf r}_i  (fgt, \prod _{i=1}^r ev_i)^* (\gamma \ot \delta) \wedge \tilde{\rho}_f^*  \tdchc\mathbb{K}_{g, r, d}  \\
   \stackrel{(5)}{=}   & \  \langle \rho_f ^* \Omega _{g , r, d} (\gamma _1, ..., \gamma _r), \delta \rangle. \\
 \end{align*}
 Line one is by definition.   Line two is by Corollary \ref{cor: mf projection for} and the smoothness of $ev_{r+1}$ (see \S \ref{sec: ev}). 
   Line three is by Lemma \ref{lem:tail} (3) and Lemma \ref{lemma: functor tdch},
   and Corollary~\ref{cor: trans related}.
 Line four uses that the degree of the map $\mathrm{bl}$ is $1/d_w$ (Lemma \ref{lem:tail} (1)). 
 Line five is by  the base change identity $fgt^* (\rho _f)_* = (\tilde{\rho}_f)_*fgt^*$ proven in Lemma \ref{lemma: integration along the fiber} (here 
 $\rho _f$ is a smooth map of relative dimension $1$ and the right $fgt$ is the forgetful map $\rho_f^* \fU _{g, r, d} \to \overline{M}_{g, r+1}$).

\subsubsection{A lemma}\label{sec:Hecke}
\begin{Lemma}\label{lem:tail}  Suppose either $2g +r - 2 > 0$ or $d \neq 0$.  Let $d_w\ge c_i \ge0$ and $(\dagger)$ holds. 
The following statements hold.
\begin{enumerate}
\item There is a natural forgetful map 
$$\rho_{f, \fB } : \fB_{\Gamma}^{g, r+1, d} (J) \to  \fB_{\Gamma}^{g, r, d} $$
satisfying that the induced map  $\fB_{\Gamma}^{g, r+1, d} (J) \to \rho^*_f \fB_{\Gamma}^{g, r, d}$  is of degree $1/d_w$.
\item There is an isomorphism 
$$ \rho _{f, \fB, LG}: LG_{g, r+1, d} (J)_0 \to \rho_{f, \fB}^* LG_{g, r, d}.$$
\item  The  perfect TK factorizations associated to
\begin{equation*}\label{eqn:tail vir pull}    \tilde{\iota}^*  \mathbb{K}_{g, r+1, d} (J) 
 \text{ and }  \tilde{\rho} _{f, \fB} ^* \mathbb{K}_{g, r, d}  \end{equation*}
are transitively related 
with respect to $ ( (I\cX)^{r} , \boxplus w ) \ti (\overline{M}_{g, r+1}, 0)$.
\end{enumerate}
\end{Lemma}

\begin{proof}
 (1) For every $T$-point $(C_1, P_1, \kappa_1)$ of $\fB_{\Gamma}^{g, r+1, d} (J)$ we will construct a $T$-point $(C_2 , P_2, \kappa _2)$ of $\fB_{\Gamma}^{g, r, d}$.  This construction will commute with the base change of $T$, yielding the map  $\rho_{f, \fB }$.

Let $\epsilon ' : C_1 \ra C_2'$ be the natural map where $C_2'$ is the partial coarse moduli space of $C_1$ obtained by 
forgetting the marking $\sG_{r+1}$ and then removing the stacky structure at $\sG_{r+1}$.

 Consider the natural exact sequence of coherent sheaves
\begin{equation} \label{eq:exact:hecke}
 0 \to  \mathcal{V}_2' \to \epsilon '_* (\cV_1:=P_1(V) ) \to \epsilon_* ( P_1(V_J^m ) |_{\sG_{r+1}}) \to 0
 \end{equation}
where $\mathcal{V}_2'$ is defined to be the kernel 
(with the convention that if $c_i < d_w$ for every $i$, i.e., $V_J^m = \{0\}$, then $\mathcal V '_2$ is just  $\epsilon '_* \cV_1$.)
Locally near the gerby marking $\sG _{r+1}$ we have
\begin{equation}\label{eq: local desc of V} \cV_1 \cong \oplus_i \cO _{C_1} ( \frac{c_i}{d_w}  [\underline{\sG _{r+1}} ] )  \end{equation}
and hence near $\underline{\sG_{r+1}} \subset C_2'$ the exact sequence \eqref{eq:exact:hecke} is isomorphic to
   \[ 
0 \to  \oplus _{i} \cO _{C_2'} \to  \oplus_i \cO _{C_2'} ( \lfloor \frac{c_i}{d_w} \rfloor [\underline{\sG _{r+1}} ] ) \to
                   \oplus_{i : 1 \le \frac{c_i}{d_w} \in \ZZ}  \cO _{C_2'} ( \lfloor \frac{c_i}{d_w} \rfloor [\underline{\sG _{r+1}} ] ) |_{\underline{\sG}_{r+1}} \to 0.  \] 
Thus  $\mathcal V_2'$ is a vector bundle. In fact, $\mathcal V_2'$ is the associated vector bundle of
a principal $\Gamma$-bundle $P_2'$ constructed by a Hecke modification as follows.

 \medskip
{\em Hecke Modification.} Consider the $\CC ^{\ti}$-bundle $L$ obtained by
removing the zero section of  $\cO _{C_1}(\frac{1}{d_w} [ \underline{\sG} _{r+1} ])$.
Then $P_1 \ti_{C_1} L $ is a $\Gamma \ti \CC^{\ti}$-bundle on $C_1$. Now using the multiplication map
$\Gamma \ti \CC^{\ti} = \Gamma \ti \CR \to \Gamma$, we construct a new $\Gamma$-bundle $$P_1' :=  (P_1 \ti_{C_1} L) \ti _{\Gamma \ti \CC^{\ti}} \Gamma$$ on $C_1$.
(Using the inverse multiplication map $\Gamma \ti \CC^{\ti} \to \Gamma ; (h_1, h_2) \mapsto h_1 h_2^{-1}$, we may reverse the 
construction, i.e., recover $P_1$ starting from $(P_1', L)$. This reversing will be used in the proof of Lemma. \ref{lemma:LGM0}.) 
By a local computation, we see that there exists a unique pair $(P_2', t) $ of a $\Gamma$-bundle $P_2'$ on $C_2'$ 
and an isomorphism  $t: \epsilon'^* P_2' \cong P_1'$. 
This yields a pair $(P_2', s)$ of a $\Gamma$-bundle $P_2'$ on $C_2'$ and an isomorphism 
$s: P_2'|_{C_2' \setminus \sG_{r+1} } \to P_1 |_{C_1 \setminus \sG_{r+1}}$ making
 $ \mathcal{V}_2' \cong P_2'\ti_{\Gamma} V $ compatible with \eqref{eq:exact:hecke}.

 \medskip

 Since $\epsilon'$ is a degree $1$ map, 
the degree of $P_2'$ is equal to  the degree of $(\epsilon')^* P_2'$. On the other hand, we have $(\epsilon')^* P_2' (\CC^{\ti}_{\delta})  
= P_1 (\CC^{\ti}_{\delta}) $
for every $\delta \in \widehat{G}\ot_{\ZZ} {\QQ} = \mathrm{Ker} (\widehat{\Gamma} \to \widehat{\CR}) \ot _{\ZZ}\QQ$.
Thus the degree of $P_2'$ is equal to  the degree $d$ of $P_1$.

Using $s$, we have a canonical isomorphism
\[ \bar{s}_{\chi}: (\epsilon ')^* P_2' (\CC_{\chi})  \cong P_1(\CC_{\chi}) \ot \cO (-[\underline{\sG_{r+1}}]) \notag  \]
which is isomorphic to $  (\epsilon ')^* \omega _{C_2'}^{\log}$ via $\kappa_1$.
Denote by $\kappa _2'$  the isomorphism $\epsilon'_* (\kappa_1 \circ \bar{s}_{\chi}) : P_2' (\CC_{\chi}) \to  \omega _{C_2'}^{\log}$.

\medskip

{\em Stabilization. } We will also require the use of stabilization for nonstable trees (see \cite{AV}). For a moment let $T$ be a geometric point. 
Then the stabilization is 
a sequence of contractions of the genus zero components of $C_2'$ which carry only two special points
and whose restriction to $P_2' (\CC_{\nu})$ has degree zero.
When one of the two special points is a marking, say $\sG_i$, then 
the node attached to the contracted component becomes the new marking $\sG_i$.

By Lemma \ref{lemma:metric} (2),
the new marking $\sG_i$ has the same type (orbifold structure) as the old $\sG_i$.
The stabilization procedure produces a sequence of these contractions until the new contracted curve $C_2$ with degree $P_2' (\CC_{\nu})$
becomes $\nu$-stable for $\nu >\!\!> 0$. 

Let $\epsilon '' : C_2 ' \to C_2$ be the contraction described above and let $E$ be the exceptional locus of $\epsilon ''$. 
Then the natural map $E \to V/\!\!/_{\theta}G$ induced from $[u]_2' |_{E} \to [V/\Gamma]$ is
a constant map to a point $p$ of $V/\!\!/_{\theta}G$ (see Remark \ref{rmk: degree and triviality} (3)).
The triple $(E, P_2'|_{E} , \epsilon '_*u_1|_{E})$ amounts to a map $E\to B\mu_a$ where $B\mu_a$ is the reduction of the inverse image of $p$.
In other words, $E \to [V/G]$ factors as  $E \to B\mu_a \to [V/G]$. This implies that
the map $[\epsilon '_*u_1]:  C \to [V/\Gamma]$ factors as $C_2' \to C_2 \to [V/\Gamma]$ for some unique map  $C_2 \to [V/\Gamma]$.
Hence we obtain the data $(C_2, P_2, u_2)$ such that $(\epsilon '')^* P_2 \cong P_2'$.

This whole procedure can be upgraded to a $T$-family version  as 
follows. We view $(C_2', P_2', \epsilon'_* u_1, \kappa ')$ as a representable map 
\[ P_2' (\CC ^{\ti}_{\chi}) \underset{\kappa_2'}{\cong} (\omega^{\log} _{C_2'})^{\ti} \to [V^{s}(\theta) /G] . \]
Note that $P_2' (\CC ^{\ti}_{\chi})$ restricted to the contracted locus is a trivial $\CC ^{\ti}$-bundle on the locus. 
Hence, locally $P_2' (\CC ^{\ti}_{\chi})$ is the trivial $\CC^{\ti}$-bundle $C_2' \ti \CC ^{\ti}$. Now we can handle the contraction
as in the proof of  \cite[Proposition 9.1.1]{AV} using an \'etale presentation of $[V^{s}(\theta) /G]$.

\medskip

Let $\epsilon = \epsilon '' \circ \epsilon '$. 
Then by the pushforward of $\epsilon$ we obtain the isomorphism
\[
 \kappa_2:=\epsilon_* (\kappa_1 \circ \bar{s}_{\chi}):  P_2 (\CC_{\chi})   \xrightarrow{\cong} \omega _{C_2}^{\log}.
 \] 
Hence we finally obtain the data $(C_2, P_2, \kappa _2)$, i.e., we get a $T$-object in $\fB_{\kG}^{g, r, d}$, as desired.
The assertion on the degree of $1/d_w$ follows from counting automorphisms at the generic point.

(2) Let $(C_1, P_1, \kappa _1, u_1)$ be an object in $LG_{g, r+1, d}(J)_0$. 
Then  $\epsilon _* u_1 : \cO _{C_2} \to \epsilon_*\cV_1$ is an element of $\Gamma(C_2, \cV_2)$.
Hence we obtain an object of $\rho_{f, \fB}^* LG_{g, r, d}$:
$$((C_1, P_1, \kappa ),  (C_2, P_2, \kappa _2, u_2=\epsilon _* u_1), (\mathrm{id}: C_2 = C_2, \mathrm{id}: P_2 = P_2 ) ) $$
It is easy to see that the stability condition of $(C_2, P_2, \kappa _2, u_2 )$ holds for $\nu >\!\!> 0$. 
Also, we may check the functoriality of the assignment so that it yields a map $LG_{g, r+1, d}(J)_0 \to \rho_{f, \fB}^* LG_{g, r, d}$.
It is also straightforward to construct its inverse map.

(3) This follows from by Theorem \ref{thm: cat glue and tail} applied to the case where 
$S := \fB_{\Gamma}^{g, r+1, d}(J)$ with the identity map $S \to \fB_1 := S $ and $\rho _{f, \fB}: S \to \fB_2:=\fB _{\Gamma}^{g, r, d}$
(there $\iota$, $V^m$, $V^f$ are  here $\tilde{\iota}$, $V^m_J$, $V^{\CR} \oplus V'$, respectively where $V'$ is the moving part of the $\lan J \ran$-space $V$).
\end{proof}

\begin{Rmk}\label{rmk: degree and triviality} Let  $(C, P, \kappa, u)$  be an LG quasimap over  $T=\Spec \CC $.

(1) \cite[Propostion 5.1.1]{FJR:GLSM}  
For each component $C'$ of $C$ we have a non-base point $p \in C'$.
By the base point condition of the LG quasimap, there is a positive integer $m$ and a $\Gamma$-invariant polynomial function $h: V\to \CC_{m\nu}$ such that $h(p)\ne 0$.
Thus we have the induced section $h(u)$ of the line bundle $P(\CC_{m\nu})$ on $C$, which is nonzero on the component $C'$. 
Hence  $\deg P(\CC _{\nu})|_{C'} \ge 0$ and $\deg P(\CC _{\nu})|_{C'} = 0$ exactly when $P(\CC_{\nu})|_{C'}$ is trivial.

(2) \cite[Propostion 5.1.1]{FJR:GLSM}  Suppose that $\deg P(\CC _{\nu}) = 0$. Then there is no base point of $u$ at all. Hence the map 
$C \to [V/\Gamma]$ factored through $[V^{ss}(\nu )/\Gamma]$. Using the natural map from $[V^{ss}/\Gamma ]$ to the GIT quotient $V/\!\!/_{\nu}\Gamma$,
we obtain a map $[u]_{git}: C \to V/\!\!/_{\nu}\Gamma$. Note that $P(\CC _{\nu})$ is the pullback of the  ample line bundle of the GIT quotient 
over the affine quotient $\Spec \CC [V]^{\Gamma}$. Hence $[u]_{git}$ is a constant map.

(3) Suppose that $P(\CC_{\chi})$ is trivial. Then $P$ has a reduction $P'$ to  $G$, and hence we may view $u$ 
as a section of $P'(V)$ so that we obtain a map $[u]_G: C\to [V/G]$. Suppose furthermore that $\deg P(\CC_{\nu}) = 0$. Then
the map $[u]_G$ lands on $[V^{s}(\theta)/G]$ and hence we get a map $C \to  V/\!\!/_{\theta}G$, which is a constant map since  
$P'(\CC_{\theta}) = P(\CC_{\nu})$ is trivial.  This implies that the map $[u]_G$ factored through a point $B\mu _a$ of $[V^{s}(\theta)/G]$ for some integer $a$.
In turn, this implies that $P'$ has a reduction $P''$ to $\mu _a$.
\end{Rmk}

\subsection{Metric axiom}
In this section, as in \S \ref{section:tails} we assume again that $\nu$ is large enough,
$d_w \ge c_i \ge 0$, $\forall i$, and $(\dagger)$ holds.
The metric axiom is as follows
\begin{eqnarray}\label{eq:metric axiom}  \Omega _{0, 3} (\gamma _1, \gamma _2, \one ) = \eta (\gamma _1, \gamma _2 ) ; \end{eqnarray} 
see \eqref{def pairing} for the definition of the pairing $\eta$.
The proof will follow from a further analysis of what proven in \S \ref{sub: proof of axiom tails} for the special case when $g=0, r=2$.
Thus we will reuse the notation in \S \ref{section:tails}. Only the difference is that
we cannot use $\overline{M}_{0,2}$, $LG_{0,2, 0}$ (and hence $\U_{0,2, 0}$) which do not exist since
the ample stability condition cannot be satisfied at all.

\begin{Rmk}
Note that the metric axiom uniquely determines a unit if one exists.
\end{Rmk}

\subsubsection{Proof of \eqref{eq:metric axiom}}
By Lemma  \ref{lemma:metric} (2), it is enough to consider the case where  type is $(h, h^{-1}, J)$ .
 Consider the commuting diagram\footnote{This commutes using the twisted diagonal map $\Delta_\zeta$ since the log differentials at $0$ and $\infty$ are identified with opposite sign.}
\begin{equation}
 \xymatrix{    LG_{0, 3, 0} (\cX)(h, h^{-1}, J) _0 \ar@{^{(}->}[r]^{i} \ar[rd]_{\Delta _{\zeta}  \circ ev_1 }  &    \U_{0, 3, 0} (h, h^{-1}, J)_0 \ar[d]^{(ev_1 , ev_2 )}  \\
                    & \cX _h \times \cX_{h^{-1}}  .  } 
                    \label{eq: inversion diagram}
                    \end{equation}
            For $\gamma _1 \in \State_h,  \gamma _2 \in \State _{h^{-1}} $ we have the following sequence of equalities,
             \begin{align*}
          & \ \ \   \  \Omega_{0, 3} (h, h^{-1}, J) (\gamma _1, \gamma _2, \one) \\
           & \stackrel{(1)}{=}    \Omega_{0, 3, 0} (h, h^{-1}, J) (\gamma _1, \gamma _2, \one) \\ 
             & \stackrel{(2)}{=}    \int _{\U_{0,3, 0 } (h, h^{-1}, J)_0 } \td\ch (\tilde{\iota}^* \mathbb{K}(h, h^{-1}, J)) \wedge ev_1^*(\gamma _1) \wedge ev_2^* (\gamma _2)  \\
             & \stackrel{(3)}{=}    \int _{LG_{0, 3, 0} (\cX) (h, h^{-1}, J)_0} ev_1^*(\gamma _1) \wedge ev_2^*( \gamma_2 )  \\
             & \stackrel{(4)}{=}    \eta (\gamma _1, \gamma _2)  .   \end{align*} 
Line one is   by the vanishing    by Lemma \ref{lemma:metric} (1).                
Line two is by line (2) of \S \ref{sub: proof of axiom tails}.
Line three is by Lemma \ref{lemma:metric} (4) and Corollary \ref{cor: mf projection for}.
Line four is  by Lemma \ref{lemma:LGM0},  \eqref{def pairing}, and \eqref{eq: inversion diagram}.

\subsubsection{Two lemmas}
Let  $\gamma_i \in \sH_{h_i}$ for $i=1, 2$
and let $\cV_0$ be the kernel of the map $\cV \to \cP (V_{J}^m) |_{\sG_3}$ on the universal curve for $LG_{0, 3, d} (h_1, h_2, J)$.

\begin{Lemma}  \label{lemma:metric}  The following statements hold.
\begin{enumerate} 
\item If $d \in \Hom (\widehat{G}, \QQ)$ is nonzero, then $\Omega _{0, 3, d} (\gamma _1, \gamma _2, \one ) = 0$.
\item If $LG(\cX)_{0,3,0}(h_1, h_2, J)_0$ is nonempty, then  
$h_1h_2=1$.
\item  If $LG(\cX)_{0,3,0}(h_1, h_2, J)_0$ is nonempty, then $\RR^1\pi _* \cV_0 = 0$. 
\item Let $d=0$. Then $\tilde{\iota} ^*\beta$ is a regular section yielding the smooth zero locus $LG _{0, 3, 0}(h, h^{-1}, J)_0$.
\end{enumerate}
\end{Lemma}

\begin{proof}
(1) We use diagram \eqref{eq: forgetting tails fiber ev} without the bottom line since the forgetful map $\rho_f : \overline{M}_{0,3} \to \overline{M}_{0,2}$ does not exist.  
By line (3) of the display in \S \ref{sub: proof of axiom tails}, we have
\[ \lan \Omega _{0, 3, d} (\gamma _1, \gamma _2, \one ), 1\ran =  \ d_w \int _{\rho ^*_{f, \fB} \U_{0, 2, d} }   \prod_{i=1}^{2}  {\bf r}_i   \prod _{i=1}^r( ev_i)^* (\gamma) 
   \wedge  \tilde{\rho}_{f, \fB}^* \tdchc \mathbb{K}_{0, 2, d} .\]
 The right hand side vanishes since its integrand is a pullback of a class in $\U_{0, 2, d}$
  under $\tilde{\rho}^*_{f, \fB}$ whose relative dimension is strictly positive.

(2) When $G$ is abelian, this is a special case of the selection rule  \cite[Proposition 2.2.8]{FJR}, \cite[Equation (3.14)]{PV:MF}.
For general $G$, we argue as follows. 

Let $((C_1, \sG_1, \sG_2, \sG_3), P_1, \kappa _1, u_1)$ be a $\bfk$-point of $LG_{0, 3, 0} (\cX)(h_1, h_2, J)_0$.
Perform the modification procedure described in  \S\ref{sec:Hecke} without the stabilization to obtain a new tuple $((C_2, \sG_1, \sG_2), P_2, \kappa_2, \epsilon _* u_1)$.  This tuple gives the data of a  point in the moduli stack $LG_{0,3,0}(\cX )(h_1, h_2, J)_0$ {\em except that it need not obey the ample stability requirement}.

Now, since there is a map $P_2 (\CC_{\chi})  \xrightarrow{\kappa _2} \omega _{C_2}^{\log}  \cong \cO_{C_2}$, 
$P_2$ has a reduction to a $G$-principal bundle $\bar{P}_2$.
Furthermore, since $\epsilon _* u_1 \in \Gamma (C_2 , P_2(V) \cong \bar{P}_2(V))$, the triple  
$((C_2, \sG_1, \sG_2), \bar{P}_2, \epsilon _* u_1)$ yields a map to $\cX$.  As this is a constant map
at the coarse moduli space level, it  factors through a point $B\mu _r$ of $\cX$ for some $r\ge 1$. This in turn gives a further reduction of $\bar{P}_2$ to
a $\mu_r$-bundle $\bar{P}_{2, \mu_r}$. This shows that $h_1h_2 = 1$ and $C_1$ is irreducible. 

(3) It is enough to prove the statement for every $\bfk$-point of $LG(\cX)_{0,3,0}(h_1, h_2, J)_0$. 
Note first that 
$$\epsilon _* \cV_0   = P_2  (V) \cong \bar{P}_{2, \mu_r} (V) \cong \bigoplus \cO_{C_2}(\frac{[\underline{\sG_1}] } {r} - \frac{[\underline{\sG_2}]}{r}).$$
Thus we have $H^1(C_1, \cV_0)=H^1(C_2, \cV_2) =0$.

(4) By $\RR^1\pi _* \cV_0= 0$  proven in (3), we have a commuting diagram of short exact sequences
\begin{eqnarray}\label{diag:metric} \xymatrix{ 0 \ar[r] & \pi_* \cV \ar[r] & \pi_*\cA \ar[r]  & \pi_*\cB  \ar[r] & 0 \\
0 \ar[r] & \pi_* \cV_0 \ar[r] \ar[u] & \pi_*\cA _0 \ar[r] \ar[u] & \pi_*\cB  \ar[r] \ar[u]_{=} & 0 , } \end{eqnarray}
where $\cA _0 $ is the kernel of the map $\cA \to \cP (V^{J}_{m}) |_{\sG_3}$. The diagram clearly shows the proof.
\end{proof}

\begin{Lemma}\label{lemma:LGM0}
The evaluation map $ev_1: LG_{0, 3, 0}(\cX)(h, h^{-1}, J)_0 \to \cX_h $ is an isomorphism of stacks.
\end{Lemma}
\begin{proof}
 We construct an inverse morphism  $\varphi: \cX_h \to LG(\cX)_0 $ as follows.
Let $\mu_r \subset \CC^{\ti}$ denote the group of $r$-th roots of  the unity and ${\bf r}_1$ be the order of $h$.   Now, let $\mu_r$ act on $\PP ^1$ by
  $[\xi x,  y]$ for $\xi \in \mu _r \subset \CC ^{\ti}$.

Let us recall that there is an isomorphism of stacks,
 \[
\cX_h \cong \Hom^{rep}_0([\PP ^1/\mu _{{\bf r}_1}], \cX)
 \]
where $\Hom^{rep}_0([\PP ^1/\mu _r], \cX)$ parameterizes
  representable 1-morphisms from $ [\PP ^1/\mu _r] $ to $\cX$ such that the induced map on coarse moduli is constant.
   Indeed, in general, the cyclotomic inertia stack $I\cX$ is
 equivalent to the stack $\coprod _{r=1}^{\infty} \Hom^{rep} (B\mu _r, \cX)$
 parameterizing representable 1-morphisms from $B\mu _r$ to $\cX$ (see \cite{AGV})
and the stack $\Hom^{rep} (B\mu _r, \cX)$ parametrizing representable 1-morphisms from $B\mu_r$ to $\cX$ is equivalent to the stack 
$\Hom^{rep}_0([\PP ^1/\mu _r], \cX)$ (see \cite[Lemma 3.3]{CCK}).
  
Hence it remains to show that   $LG(\cX)_0$ is equivalent to $\Hom^{rep}_0([\PP ^1/\mu _r], \cX)$.
  Fix $C_2 = [\PP ^1/\mu _r]$ with two markings at $0$ and $\infty$. We will define a 2-morphism $\varphi$
           \begin{eqnarray*} \Hom^{rep}_0 (C_2 , \cX) (h, h^{-1}) & \to &  LG_{0, 3, 0} (\cX) (h, h^{-1}, J) 
           \\ ( C_2 , \bar{P}_2, u_2)   & \mapsto & (C_1, P_1, \kappa _1,  u_1 ) \end{eqnarray*}
            as follows.
            First note that every object $(C_2, \bar{P}_2, u_2)$ of  $\Hom_0 (C_2, \cX) (h, h^{-1})$ over $\Spec \CC$
            satisfies that  $\bar{P}_2 = [(\PP ^1 \ti G ) / \mu _{{\bf r}_1}]$ and $u_2$ is a section of $\bar{P}_2 (V) = [(\PP ^1 \ti V )/  \mu _{{\bf r}_1} ]$
            landing on the stable locus $[(\PP ^1 \ti V^{ss} )/  \mu _{{\bf r}_1}]$, where $\mu _{{\bf r}_1}$ acts on $G$ via  $e^{2\pi i / {{\bf r}_1}} \mapsto h$.
            In particular $\bar{P}_2  (\CC _{\chi}) = \cO _{C_2}$.
            We define its associated object $(C_1, P_1, \kappa _1,  u_1 ) $ as:
             \begin{itemize}
             \item $C_1$ is the ${\bf r}_1$-th root stack $C_2 (D^{1/d_w})$ of $[\PP^1/\mu_{{\bf r}_1}]$
              associated to the Cartier divisor $D=[1,1] \in C_2$.
             Here we are using the convention that the three markings of 
              type $h, h^{-1}, J$ are respectively placed over $0=[0, 1]$, $\infty=[1,0]$, $[1,1]$ in $\PP^1$; 
              \item If  $\epsilon : C_1 \to C_2$ denotes the natural map,  $P_1$ is the \lq\lq reversed" Hecke modification of $\epsilon ^*(P_2:= \bar{P}_2 \ti _G \Gamma )$
              as in the proof of Lemma \ref{lem:tail}.
               \item $u_2 \in \Gamma (C_2, \cV_2:=P_2(V)) 
               \subset \Gamma (C_2, \epsilon _* (\cV_1:=P_1(V))$ which can be regarded as an element $u_1$ of $\Gamma ( C_1, \cV _1)$
               whose evaluation at $\sG_3$ lands on $\pi _* (P_1 \ti_{\Gamma} V^{\CR}|_{\sG_3})$ as seen in the proof of Lemma \ref{lemma:metric} (3); 
              \item  $\kappa _1$ denotes the composition  of 
              $ P_1 (\CC _{\chi})  = \epsilon^*\bar{P}_2 (\CC_{\chi}) ( [\underline{\sG}] ) = \cO_{C_1} ([\underline{\sG}] ) = \omega _{C_1}^{\log}$.
                 \end{itemize} 
                  The obvious family version of the above assignment gives rise to a morphism $\varphi$.
                  It is straightforward to check that $\varphi$ is an inverse of $ev_1$.
\end{proof}

\subsection{Grading and Homogeneity}\label{Homog}

In this section we consider an open and closed component of $\Ugrd$ where the type of the $r$ markings
are the fixed $r$ conjugacy classes $\ug = (h_1, ..., h_r) \in G^r$.
If  the cyclic group $\lan h_i \ran$ has order $l$, then
the age of $h_i$ is by definition the sum of the exponents $k/l$ of the eigenvalues $\exp (2\pi \sqrt{-1} k/l)$ of the $\lan h_i \ran$-representation space $V$,
where $k \in \{ 0, 1, ..., l-1 \}$.

The virtual dimension $\virdim$ of $\Ugrd(\underline{h})$ is by definition
$\dim \fB_{\Gamma} + \chi (\mathbb R\pi _* \cV)$, whose Riemann-Roch formula  (see  \cite[Lemma 6.1.7]{FJR:GLSM})      
 is given by
  \[
  \dim \fB_{\Gamma} + \chi (\mathbb R\pi _* \cV) = \int _d c_1(X) + (\hat{c} - 3 ) (1-g ) + r -   \sum_i  (\age (h_i)  - q)
   \]
   with the notation $\int _d c_1(X) := d (\mathrm{det } V)$,
$\hat{c} = \dim V - \dim G - 2 q$, and $q := \sum_j  c_j / \bfdeg $.
The number $\hat{c}$ is called the central charge of the GLSM.

We assign a $\ZZ$-grading on the state space $\sH$ as folllows.
For each $i$-th cohomology class $\gamma$ in $\sH$ supported in a component of the inertia stack $I\cX$,
we let $\age(\gamma)$ be the age of the component.
Now we put
\[ \deg \gamma :=  i + 2(\age (\gamma) - q) .\]

\begin{example} Continuing the running example \ref{ex: LG/CY 2}, the age grading splits
$\sH_-$  into one dimensional subspaces of degree $0, 2, 4$ and  $6$,
and a $204$ dimensional subspace of degree $3$.   On the other hand, the age grading on $\sH_+$ and homological grading agree since the corresponding quotient stack is a variety in this case.  However, there is a shift of 2 in comparing this grading with the  usual homological grading on the singular cohomology of $Z(f)$ (see Example~\ref{ex: Rchargechoice}).
\end{example}

 We also equip $H^*( \overline{M}_{g, r})$ with the usual $\ZZ$-grading. 
It is clear that the pairing from \S\ref{subsec: pairings} is supersymmetric
with respect to both the usual $\ZZ$-grading and the age-shifted grading.

Let  
\begin{align*} 
\underline{\LGvir} \in \HH ^{*} _{\Z} ( \Ugrd ^{an}, (\Omega ^{\bullet}_{\Ugrd ^{an}}, -dW )))  
\end{align*}
 be the $2\rank B$-degree part of the analytification of the class $\LGvir$.
We define a new collection of GLSM invariants        
\[
\underline{\Omega} _{g, r, d} := fgt_* \circ \underline{\LGvir}  \wedge \circ ev^*.
\]
 \begin{Thm}\label{thm: homog}
If $\ka$ has an algebraic model, then the GLSM invariants $\underline{\Omega}_{g, r, d}$ form a homogeneous cohomological field theory.  In particular, for every convex hybrid GLSM 
with conditions $\nu \gg 0$, $d_w \ge c_i \ge 0$, $\forall i$, and $(\dagger)$, 
the GLSM invariants $\underline{\Omega}_{g, r, d}$ form a cohomological field theory with a flat identity $\one$ which is homogeneous of degree $-2( \int _d c_1(X) + (1-g ) \hat{c})$. 
\end{Thm}
\begin{proof}
We only need to check the tree gluing axiom as the rest are automatic for any homogeneous component of $\Omega _{g, r, d}$.  Now, by Lemma \ref{lem minimal} $\underline{\LGvir}$ has degree at least $2\rank B$.  Hence, the tree gluing axiom for $\{ \underline{\Omega} _{g, r, d}\}_{2g-2+r > 0}$ follows from the tree gluing axiom for $\{\Omega _{g, r, d}\}_{2g-2+r > 0}$.  
The fact that a convex hybrid model has an algebraic model is the main construction of \cite{MF}.
\end{proof}

\subsection{Sum of singularities} \label{sec: sum of sing}
Let $(V_i, \Gamma _i, \chi _i, w_i)$ be affine LG models for $i=1,2$.  Their sum is the affine LG model $(V:=V_1\oplus V_2, \Gamma , \chi, w_1 + w_2)$ where
$\Gamma$ is the fiber product $\Gamma _1 \ti _{\CC^{\ti}} \Gamma _2$ by $\chi _1, \chi_2$
and $\chi$ is the induced character of $\Gamma$.
Furthermore assume that $\nu=\nu _1 |_{\Gamma} = \nu _2 |_{\Gamma}$.

\begin{Rmk}
Starting from any two GLSMs, we can form their sum as above.  However, requiring that $\nu=\nu _1 |_{\Gamma} = \nu _2 |_{\Gamma}$ forces them to be affine LG models since $\nu|_{G = G_1 \times G_2} =1$ and hence $G_1, G_2$ must be finite.
\end{Rmk}

In this case, $\fB _{\Gamma}^{g, r}$,  $\fB_{\Gamma _i}^{g, r}$ are the moduli stacks of $\Gamma$-spin curves, $\Gamma_i$-spin curves, respectively.
Furthermore, since $V^{ss}=V$, $V_1^{ss}=V_1$, $V_2^{ss}=V_2$, their LG moduli spaces and their virtual factorizations 
are independent of the variation of $\ke$.

By the K\"unneth isomorphism \eqref{kunneth} we have the relation
$\sH \cong \sH _1 \ot \sH_2$ for the state spaces $\sH$, $\sH_1$, $\sH_2$ of these three GLSMs.
The virtual factorizations are also related using the isomorphism $\delta: \fB _{\Gamma}^{g, r} \to \fB_{\Gamma _1}^{g, r}\ti_{\overline{M}_{g, r}} \fB_{\Gamma _2}^{g, r}$.
Namely, consider the natural fiber product diagram
\[ \xymatrix{      \delta ^* (\U _1 \ti \U _2 ) \ar[r]^{\cong}_{\tilde{\delta}}  \ar[d] & \U_1 \ti_{\overline{M}_{g, r}} \U_2 \ar[d] \ar[r] & \U_1 \ti \U_2  \ar[d] \\
 \fB _{\Gamma} \ar[r]^{\cong \ \ \ \ \ }_{\delta \ \ \ \ \ } &  \fB_{\Gamma _1} \ti_{\overline{M}_{g, r}}  \fB_{\Gamma _2} \ar[r]  \ar[d]
                                                       &  \fB_{\Gamma _1} \ti  \fB_{\Gamma _2}     \ar[d]    \\
                                                       & \overline{M}_{g, r} \ar[r] & \overline{M}_{g, r} \ti \overline{M}_{g, r}  }      \]

In analogy to  \cite[Theorem 5.8.1]{PV:MF} (see also  \cite[Theorem 4.1.8. (8)]{FJR}
and Theorem 4.2.2) we have the following.
\begin{Thm}\label{thm: sum of sing} 
\begin{enumerate}
\item The virtual factorizations $\mathbb K ( \delta^*  \mathfrak a _1, \tilde{\delta}^*  \kb _1)\boxtimes \mathbb K (\delta^*  \mathfrak a _2, \tilde{\delta}^* \kb _2)$
 and $\mathbb K (\mathfrak a, \kb )$ are transitively related with respect to $((I\cX)^r \ti \overline{M}_{g, r}, \boxplus _{i=1}^r w\boxplus 0)$. 
\item Under the K\"unneth isomorphism $\sH ^{\ot r} = \sH _1^{\ot r} \ot \sH_2 ^{\ot r}$, 
 \[ \Omega _{g, r, \Gamma}  = \Omega _{g, r, \Gamma _1} \wedge \Omega _{g, r,\Gamma _2}, \]
 where  $\Omega _{g, r, \Gamma}$, $\Omega _{g, r, \Gamma _i}$ denote the cohomological invariants for the corresponding GLSMs.
\end{enumerate}
\end{Thm}

\begin{proof}
(1) Consider three universal curves $\fC\xrightarrow{\pi} \fB _{\Gamma} $, $\fC_i \xrightarrow{\pi_i}  \fB _{\Gamma_i} $;
and the universal bundle $\cP$, $\cP_i$ on $\fC$, $\fC _i$, respectively. 
We have natural maps $\fC \to \fC_i |_{ \fB_{\Gamma _1} \ti_{\overline{M}_{g, r}}  \fB_{\Gamma _2} }$.
First note that $$ \cP (V) = \delta ^* ( \cP _1 (V_1) \boxplus  \cP _2 (V_2)).$$
Hence starting from $[\cA_i \to \cB_i]$ with restriction maps $rest_i$ satisfying Proposition \ref{prop:U} for $\cP _i (V_i)$ on $\fC_i$, their sum 
$$[\cA \to \cB ]:=[\cA_1\to \cB_1]|_{\fC} \oplus [\cA_2 \to \cB_2]|_{\fC}$$ with 
$rest _1|_{\fC} \oplus rest _2|_{\fC}$,
satisfies Proposition \ref{prop:U} for $\cP  (V)$ on $\fC$. Hence we have $\mathfrak a = \delta ^* (\mathfrak a_1\boxplus \mathfrak a_2)$ and can take
$A = (A_1 \boxplus A_2) |_{\fB_{\Gamma}}$, 
$B = (B_1 \boxplus B_2) |_{\fB_{\Gamma}}$
and $\beta = \delta ^* (\beta _1 \boxplus \beta _2)$.
The result follows from \eqref{eqn: outer tensor of TK} and Theorem \ref{thm:ind}.

(2) This is immediate by (1) and the fact that the integration map along the fibers respects arbitrary base change (see Lemma \ref{lemma: integration along the fiber} (3)).
\end{proof}

\appendix
\section{Non-Hausdorff trace maps} \label{sec: trace}

Let $\tX$ be a DM stack of finite type over $\Spec\/\bfk$ which is not necessarily separated. Suppose that every component of $\tX$ is of real dimension $m$.
For a fixed closed substack $\tZ$, we may consider the local cohomology 
\[ H^m_{\tZ} (\tX) := H^m_{\tZ ^{an}} (\tX ^{an}, \ubfk) 
 , \]
where $\tX^{an}, \tZ^{an}$ are  analytic DM stacks associated to $\tX$, $\tZ$, respectively and
$\ubfk$ is the constant sheaf on $\tX^{an}$ associated to $\bfk$.
Suppose that $\tZ$ is proper over $\bfk$ so that $\tZ^{an}$ is compact and Hausdorff. 
In this appendix, we define a trace map for such a pair $(\tX, \tZ)$                                                                                                                                                                                                                                                                                                                                                                                                                                                                                                                                                                                                                                                                                                                                                                                                                                                                                                                                                                                                                                                                                                                                                                                                                                                                                                                                                                                                                                                                                                                                                                                                                                                                                                                                                                                                                                                                                                                                                                                                                                                                                                                                                                                                                                                                                                                                                                                                                                                                                                                                                                                                                                                                                                                                                                                                                                                                                                                                                                                                                                                                                                                                                                                                                                                                                                                                                                                                                                                                                                                                                                                                                                                                                                                                                                                                                                                                                                                                                                                                                                                                                                                                                                                                                                                                                                                                                                                                                                                                                                                                                                                                                                                                                                                                                                                                                                 
\begin{eqnarray}\label{eqn:def int general}  \int _{\tX} : H^m_{\tZ} (\tX) \to \bfk . \end{eqnarray}

When $\tX^{an}$ is a Hausdorff complex analytic space, we have the cap product operation, see for example \cite[\S III.1 \& IX.3]{Iversen},
as well as the fundamental class $[\tX^{an}] \in H^{BM}_{2m}(\tX^{an}) $, see for example \cite[X.3]{Iversen}.
Thus, in this ``Hausdorff analytic space case'', we can define the trace map by the cap product $[\tX^{an}]\cap$ followed by the augmentation map $\deg$
\begin{eqnarray}\label{eqn:def int Haus}  H^m_{\tZ} (\tX) \xrightarrow{[\tX^{an}]\cap } H^{BM}_0(\tZ ^{an})  \xrightarrow{\mathrm{deg}} \bfk \end{eqnarray}
where $H^{BM}_0(\tZ ^{an} )$ is the zero-th Borel-Moore homology with coefficients in $\bfk$.
Here note that the augmentation map is well-defined since $H^{BM}_0(\tZ ^{an}  )$ is canonically isomorphic to the singular homology 
$H_0(\tZ ^{an}  )$ due to the compactness of $\tZ^{an}$.
For simplicity we write $[\tX]\cap$ instead of $[\tX^{an}]\cap$.

In the following sections, we will define \eqref{eqn:def int general} for a non-Hausdorff space starting from \eqref{eqn:def int Haus}, using the Mayer-Vietoris sequence, and by considering a generically finite, proper surjective map $X \to \tX$ from a scheme $X$.

\subsection{A topological definition}\label{sub: top def of trace}  We first handle the following analytic case. Let
$X$ be a complex analytic space, possibly non-Hausdorff, of real dimension $m$. 
Let $Z$ be an analytic closed subset of $X$ which is compact. 
We construct a trace map 
\[  \int _{X} : H^m_{Z} (X) \to \bfk . \]

Since $Z$ is compact, there is a finite collection  $\{ X_i \}_{i\in I}$ of open subsets of $X$ such that each $X_i$ is 
Hausdorff  and the union $\bigcup _{i\in I} X_i $ contains $Z$. We may further assume that the closure  of $X_i$ in $X$ 
is a bounded closed semi-algebraic subset of $\RR ^{m'}$ for some positive integer $m'$.
Since $Z$ is a closed analytic subset of $X$, for each point $p$ of $Z$, we may choose an open neighborhood $V_p$ 
of $Z$ such that the closure $\overline{V}_{p}$ in $Z$ is a compact semi-algebraic set and 
$\overline{V}_p \subset  Z \cap X_i$ for some $i$.
Since $Z$ is compact, there is a finite set $A$ such that  $\bigcup_{a \in A} V_{p_a} = Z$.
Let $Z_i = \bigcup_{\overline{V}_{p_a}  \subset X_i}  \overline{V}_{p_a}$. This yields a collection of compact semi-algebraic subsets $Z_i$ of $X$ such 
that $Z_i \subset X_i$ for all $i$ and $\bigcup Z_i = Z$.

The following vanishing property 
 should be well-known. However as we do not know an adequate reference, we provide a proof.
 \begin{Lemma}
 Let $N$ be a subset of $\RR ^{m'}$ whose closure is a bounded closed semi-algebraic subset with real dimension $m$
 and let $V$ be a compact subset of $N$. Then we have
  \begin{eqnarray*}\label{eq:van} H^l_{V} (N)=0 \text{ for } l > m . \end{eqnarray*}
 In particular we have
 \begin{eqnarray}\label{eq:van} H^l_{Z_{i_0}\cap ... \cap Z_{i_q}} (X_{i_0}\cap ... \cap X_{i_q} )=0 \text{ for } l > m . \end{eqnarray}
 \end{Lemma}

\begin{proof}
By excision (see \cite[III, Exercise 2.3]{Hart}), $H^l _{V} (N) = H^l _V (\overline{N})$ and thus
we may assume that $N = \overline{N}$.
Let $M = N-V$, then  by the relative cohomology sequence \cite[III, Exercise 2.3]{Hart}
it is enough to prove that $H^l (N) = H^l (M) = 0$ for $l > m$ and the surjectivity of the restriction map
$H^m (N) \to H^m (M)$. For this we will use triangulations of compact semi-algebraic sets and we will view $H^m (N)$, $H^m (M)$
as singular cohomology.
Choose a sequence of compact semi-algebraic subsets $M_r$ of $M$, for $r=1, 2, ... $, such that
$M_r \subset M_{r+1}$ for every $r$ and $\cup_{r=1}^{\infty} M_r = M$.
By the triangulation theorem \cite[Theorem 1]{Loj}, the semi-algebraic sets $M_r$, $N$ are geometric realizations of 
finite simplicial complexes. 
Hence the singular cohomology with coefficients in $\bfk$ vanishes:
$H^{l}(N  )=  H^l(M_r)= 0 $ for $l > m$ and for every $r$.
It remains to show the surjectivity, which amounts to the injectivity 
of $H_m (M)  \to H_m(N)$ by the universal coefficient theorem for cohomology.
Notice that for all $r$, $H_m(M_r)  \to H_m(N)$ is injective 
since the triangulation theorem allows us to arrange $M_r$ to be a subsimplicial complex of $N$ for some triangulation of $N$.
Furthermore, by \cite[Proposition 3.33]{Hatcher}, 
\[
H_m (M) = \lim_{\longrightarrow} H_m(M_r).
\]
  Hence since filtered colimits are exact, the natural map between singular homology with coefficients in $\bfk$
 $$H_m (M) = \lim _{\longrightarrow} H_m(M_r)  \to H_m(N) $$ is injective.
\end{proof}

Fix a well-ordering of the index set $I$. 
By the Mayer-Vietoris sequence
and the vanishing \eqref{eq:van}, we have the following diagram where the top line is  exact
\[ \xymatrix{ \oplus _{i > j } H^m_{Z_{i}\cap Z_{j}}(X_{i}\cap X_{j}) 
\ar[r]_{\ \ \ \ \ \  \delta} \ar@/_1pc/[rd]_{ 0 }  & \oplus _i  H^m_{Z_i} (X_i )    \ar[r]_{\ \iota _*} \ar[d]^{\sum_{i\in I} \int _{X_i} } 
& H^m_{Z} (X ) \ar[r]  \ar@/^1pc/@{-->}[ld]^{\int _{X}} &  0  . \\
                       &  \bfk  &  &  } \]
 Here again we use the canonical excision isomorphism $H_{{\tt A}} ({\tt C }) \cong H_{{\tt A}} ({\tt B})$ 
 for  any closed subset ${\tt A}$ in an open subset $\tt B$ of a topological space $\tt C$.
The vertical arrow in the above diagram is defined as the sum of the trace maps $\int _{X_i}$, i.e., the composition of maps
$$ H^m_{Z_i} (X _i ) \xrightarrow{[X_i]\cap }  H^{BM}_0 (Z_i)  \xrightarrow{\mathrm{deg}} \bfk ,$$ using the fact that $X_i$ 
is a complex analytic space 
with fundamental class $[X_i] \in H^{BM}_{m}(X _i)$ given by a locally finite orientable triangulation of $X_i$. 
The projection formula shows that $(\sum_{i\in I} \int _{X_i}) \circ \delta$ is zero.
The dashed arrow which we call the \emph{trace map} (following the terminology in \cite{Iversen}) 
and denote by $\int_{X}$ can be defined as 
\begin{equation}\label{eq: def top trace map} \int _{X} \gamma := \sum _{i\in I} \int _{X_i} \gamma _i \end{equation}
 by choosing any lift
 $\iota _* \sum _{i\in I} \gamma _i  = \gamma $. Right exactness of the top line of the diagram shows that this is independent of the choice of lift.
 For a general $\gamma \in H^*_{Z} (X)$, we define $\int_{X} \gamma$ as the trace of the degree $m$-part of $\gamma$.
                       
 \begin{Lemma}
 The trace map $\int _{X}$ is well-defined, i.e., it is independent of choice of cover.
 \end{Lemma}
  \begin{proof}                     
  Let $\{(X _i ,Z_i)\}_{i\in I}$ and $\{ ( X'_j, Z'_j )\}_{j\in J} $ be two such covers.  
Taking their intersection we get a common refinement.  Hence, it is enough to verify independence on a refinement.
That is, we may assume  there is a partition $\{ J_i \} _{i\in I}$ of $J$ such that for every $i\in I$, $j\in J_i$, $X ' _j \subset X  _i$,  $Z ' _j \subset Z _i$, $\bigcup _{j\in J_i} Z ' _j = Z_i$.

  For each $j\in J$, choose $\kg _j ' \in H^m_{Z_j ' } (X _j ')$ with $\sum_{j\in J} \kg _j' = \kg$ in $H^m_{Z} (X)$.
Let $\delta _i := (\sum_{j\in J_i} \kg '_j) - \kg _i \in H^m_{Z_i}(X_i)$ and note that $\sum_{i\in I} \delta _i = 0$ in $H^m_{Z}(X)$.
Hence we have,
\begin{align*}
\sum_{j\in J}  \int _{X'_j} \kg _j '  & = \sum _{i \in I} \sum_{j \in J_i} \int _{X'_j}  \kg _j ' \\
& =  \sum _{i \in I} \int _{X_i} (\kg _i + \delta _i) \\
& =  \sum _{i \in I} \int _{X_i} \kg _i + \sum _{i \in I} \int_{X} \delta _i \\
& = \sum _{i \in I}\int _{X_i} \kg _i.
\end{align*}
\end{proof}

We now define the trace map for the pair $(\tX, \tZ)$,                                                                                                                                                                                                                                                                                                                                                                                                                                                                                                                                                                                                                                                                                                                                                                                                                                                                                                                                                                                                                                                                                                                                                                                                                                                                                                                                                                                                                                                                                                                                                                                                                                                                                                                                                                                                                                                                                                                                                                                                                                                                                                                                                                                                                                                                                                                                                                                                                                                                                                                                                                                                                                                                                                                                                                                                                                                                                                                                                                                                                                                                                                                                                                                                                                                                                                                                                                                                                                                                                                                                                                                                                                                                                                                                                                                                                                                                                                                                                                                                                                                                                                                                                                                                                                                                                                                                                                                                                                                                                                                                                                                                                                                                                                                                                                                                                                   where $\tX$ is a DM stack finite type over $\bfk$ which may not be separated 
and the closed substack $\tZ$ of $\tX$ is a proper DM stack. 
Choose a (possibly non-separated) scheme $\tX$ of finite type over $\bfk$ together with a proper, generically finite, surjective map $f: X \to \tX$;
see \cite[Theorem 2.7]{EHKV}. 
If $\tX$ is irreducible, define 
 \begin{eqnarray}\label{eq:def trace1} \int _{\tX} \gamma := \frac{1}{\deg f} \int _{X} f^*\gamma.  \end{eqnarray}
If $\tX$ is a finite union of irreducible components $\tX^{(l)}$, define  
 \begin{eqnarray}\label{eq:def trace} \int _{\tX} \gamma := \sum _{l,  2\dim \tX_l = \deg \gamma } \int _{\tX^{(l)}} \gamma |_{\tX^{(l)}} , \end{eqnarray}
which holds true for complex analytic Hausdorff  (and hence possibly non-Hausdorff) spaces.
 
 \begin{Cor}\label{cor: TraceU}
   The quantity $\int _{X} \gamma $ is independent of the choice of proper, generically finite, surjective covering $f: X \to \tX$.
 \end{Cor}                    
\begin{proof} It is enough to show it for an irreducible stack $\tX$. 
Let $f: X \to \tX, f' : X' \to \tX$ be proper, generically finite, surjective maps where $X, X'$ are schemes of finite type over $\bfk$.   
Let $Z = f^{-1} (\tZ)$ and $Z ' = (f')^{-1} (\tZ)$.
We may compare $X, X'$ to their fiber product $X \times_{\tX} X'$.  Hence, without loss of generality we may assume that
there exists a proper, generically finite, surjective map $g: X' \to X$ such that $f' = f \circ g$  and that $g^{-1}(Z) = Z'$. 

We take $X_i' := g^{-1}(X_i)$ and $Z_i' = g^{-1}(Z_i)$. Let $g_i: X_i' \to X_i$ be the induced proper map. 
Recall we have $\gamma _i \in H^m_{Z_i} (X_i)$ such that $\iota _* \sum_i \gamma _i  = \gamma$.
By the projection formula \cite[IX.3]{Iversen},
\[
\int _{X_i'} g_i^* \gamma_i =  \deg g \int _{X_i} \gamma_i
\]
 for each $i$.
 And by the commutativity $g^* \iota _* = \iota '_* \sum g_i^*$,  we have
$\sum_i g_i^*\gamma _i = (f')^*\gamma$ in $H^m_{Z'}(X ')$. 
Hence 
\begin{align*}
\frac{1}{\deg f'} \int_{X'} (f')^*\gamma & = \frac{1}{\deg f'} \sum _i  \int _{X'_i} g_i^*   \gamma_i    \\
& = \frac{1}{\deg f} \sum_i \int_{X_i} \gamma_i\\
& = \frac{1}{\deg f} \int _{X}f^*\gamma.
\end{align*}
\end{proof}

 \begin{Cor}\label{cor: ProjFor}
    Let $f: (\tX ', \tZ') \to (\tX , \tZ)$ be a proper map of degree $\deg f \in \QQ$. Then
    \begin{eqnarray}\label{eqn:ProjFor}   \int _{\tX'}  f^* \gamma = \deg f  \int _{\tX} \gamma . \end{eqnarray} 
 \end{Cor}       
       
\begin{proof}
As there is a commuting diagram of proper maps 
\[ \xymatrix{ \tX' \ar[r]_{f} & \tX \\
                     X' \ar[u] \ar[r]_{h} & X \ar[u]_{g} } \]
such that all maps are surjective and of finite degree,
the result follows by definition.
\end{proof}

\subsubsection{}       
Let $\tX \to \PP^1$ be an algebraic map from a DM stack $\tX$ of dimension $m/2$. Assume that
the map is {\em dominant} in each irreducible component of $\tX$.
Let $\tZ$ be  a proper closed substack of $X$ (here by proper substack we mean that $\tZ$ itself is proper over $\bfk$).

Given two closed points $p_0, p_1 \in \PP ^1$ we can consider the corresponding fibers $\tX^0$, $\tX^1$ of $\tX \to \PP ^1$. 
 Let $i_0 : \tX^0 \to \tX$, $i_1 : \tX^1 \to \tX$ be the inclusions. The fibers $\tX ^j$, $j=0, 1$, if nonempty, are DM stacks of real dimension $m -2$.
 Let $\tZ ^j = \tZ \times _{\tX} \tX ^j$, which are also proper closed substacks of $\tX^j$.  For $\kg \in H^{m-2}_{\tZ} (\tX)$, we have
  $i_0^* \gamma  \in H^{m-2}_{\tZ ^0} (\tX ^0 )$ and $i_1^* \gamma  \in H^{m-2}_{\tZ ^1} (\tX ^1 )$.

\begin{Lemma}\label{lemma:def inv} For $\kg \in H^{m-2}_{\tZ} (\tX)$,
$\int _{\tX ^0} i_0^* \gamma = \int _{\tX^1} i^*_1 \gamma $.
\end{Lemma}

\begin{proof}
By taking a generically finite covering $X \to \tX$, we may assume that $\tX$, $\tZ$ are algebraic schemes $X$,  $Z := \tZ \ti _{\tX} X$. 
Choose any finite open covering $\{(X _i ,Z_i)\}_{i\in I}$ associated to a collection of separated algebraic schemes $X_i$,
and define $X^j_i = X^j \cap X_i$, $j=0, 1$. 
For each $i\in I$, we have fundamental classes $[X ^j_i] \in H^{BM}_{m-2} (X ^j_i)$, $j=0, 1$ and they 
yield  the same class $(i_0)_*[X_i^0] = (i_1)_*[X_i^1] \in H^{BM}_{m-2} (X _i)$ since
$X^0_i$, $X ^1_i$ are rationally equivalent as algebraic cycles in $X_i$.
Then we have
\begin{align*}
\int _{X ^0} i_0^* \gamma & = \sum _{i\in I} \int _{X ^0_i} i_0^* \gamma_i & \\
& =  \sum _{i\in I} \deg ( (i_0)_*[X ^0_i]  \cap \gamma_i )  & \text{ by the projection formula \cite[IX.3]{Iversen}} \\
& =  \sum _{i\in I} \deg ( (i_1)_*[X ^1_i]  \cap \gamma_i )  & \text{ by } (i_0)_*[X_i^0] = (i_1)_*[X_i^1]   \\
& = \sum _{i\in I} \int _{X ^1_i} i_1^* \gamma_i  & \\
& = \int _{X ^1} i_1^* \gamma, &
\end{align*}
where $\deg$ are  the degree maps $H_0^{BM}(Z_i ) \to \bfk$ for $i\in I$.
\end{proof}

\subsubsection{}

Consider a map $\tX \to \tM$ and let $\tZ$ be a proper closed substack of a DM stack $\tX$ of finite type.
Here by a proper substack, we mean that $\tZ$ itself is proper over $\bfk$.
We say a Thom class   exists if there is a class $\tau \in H^d_\tX(\tM)$ uniquely characterized by the property $\tau |_{\tM _i} \in H^d_{\tX_i} (\tM_i)$ is the usual Thom class 
for every separated open $\tM_i$ and $\tX_i = \tX \ti _{\tM} \tM_i$; see \cite{Iversen}.

We are now able to define Gysin maps in the following two cases.

Case 1. If $\tM$ is a proper smooth DM stack, 
we define 
\begin{equation} \label{eqn: def of trace push}
f_*:  H^{*}_{\tZ}(\tX) \to H^* (\tM)
 \end{equation}
by the requirement: 
\begin{equation*}   \text{for } a\in H^{*}_{\tZ}(\tX),   \int _{\tM} f_* a \cup c = \int _{\tX} a \cup f^* c,  \text{ for all } c \in H^*(\tM) . \end{equation*}

Case 2. If $f$ is a closed regular immersion of real codimension $d$ and a Thom class  $\tau \in H^d_\tX(\tM)$ exists, then 
we define the Gysin map as
\begin{equation} \label{eqn: def of Gysin}
 f_* : H^{*}(\tX) \to H^{*+d}_{\tX}  (\tM) ; a \mapsto \tau \cup a.
\end{equation}

\begin{Lemma} \label{lem: Gysin props}
The Gysin map for closed regular immersions enjoys the following properties:
\begin{itemize}
\item The Gysin map is an isomorphism.
\item The Gysin map  satisfies flat base change.
\item Let $\tV$ be a proper closed substack of $\tM$. Then for $a \in H^{*}(\tX)$,  $b \in H^*_{\tV}(\tM)$ 
$$ \int _{\tM} r (f_* a) \cup b = \int _{\tX} a \cup f^* b ,$$
where $r$ is the restriction map $H^*_{\tX} (\tM) \to H^*(\tM)$.
\end{itemize}
\end{Lemma}
\begin{proof}
In the separated case, the first two properties are well-known.  These properties in the general case follow from the Mayer-Vietoris sequences for separated open coverings.
The third property in the separated case follows the Mayer-Vietoris sequence for closed separated semi-algebraic coverings of $\tV$ and from the fact that $[\tM] \cap \tau = [\tX]$ . We verify this last equality.
 
 Let $\tM^{ns}$, $\tX^{ns}$ be the nonsingular loci of $\tM$, $\tX$, respectively. 
Then $[\tM^{ns}] \cap \tau = [\tX^{ns}]$ and hence  $[X] \cap \tau = [\tM]$ because $H^{BM}_{2\dim \tX} (\tX) \cong H^{BM}_{2\dim \tX} (\tX^{ns})$ under the restriction map.
  \end{proof}

Consider a fiber product 
 \[ \xymatrix{  \tY \ar[r]^{j} \ar[d]_{g}   &  \tX \ar[d]^{f}  \\
          \tN  \ar[r]_{i} &  \tM }. \] 
         of DM stacks.
Suppose that: 
$\tM$ is a proper smooth DM stack; 
$i$ is a closed regular immersion; 
and $f$ is flat (and hence $j$ is also a regular immersion whose normal cone is isomorphic to the pullback of the normal cone of $i$).

\begin{Lemma}\label{lemma: base change of Gysin maps} 
         For any $a \in H^{*}_{\tZ}(\tX)$, $b\in H^* (\tN)$ with $|a|+|b| = 2\dim \tY$,  the following equality holds:
         \begin{eqnarray*}\label{eq:base change of integration} \int _{\tN}     b \cup i^*  (f_* a)  = \int _{\tY}       g^* b \cup j^ * a  .  \end{eqnarray*}
         Here $|a|, |b|$ mean the degree of $a$, $b$, respectively and $f_*$ is defined by \eqref{eqn: def of trace push}.
       \end{Lemma} 
\begin{proof} 
Let $r_1: H^* _{\tN} (\tM) \to H^* (\tM)$, $r_2: H^* _{\tY} (\tX) \to H^* (\tX)$ be the restriction maps
and let $d$ be the real codimension of $\tN$ in $\tM$. Using the Thom class $\tau$ of
the regular immersion $i$, which exists since $\tM$ is separated,
and the Thom class $f^*\tau$ of the regular immersion $j$, we have
Gysin maps  $i_* : H^*  (\tN) \to H^{*+d}_{\tN}(\tM) $ and $j_*: H^* (\tY) \to H^{*+d} _{\tY} (\tX)$.
Now we have a sequence of equalities by Lemma~\ref{lem: Gysin props}:   
\begin{eqnarray*} & & \int _{\tN}   b \cup i^*  (f_* a)   =   \int _{\tM}     r_1(i_* b) \cup f_* a 
   \underset{ \eqref{eqn: def of trace push} }{=}   \int _{\tX}       f^* r_1(i_* b)   \cup  a  \\
                              &   \underset{ \because f^*r_1= r_2f^*}{=}  & \int _{\tX}       r_2 (f^* i_* b)  \cup a    
                                                              =  \int _{\tX}       r_2 (j_* g^* b) \cup a 
                                  \underset{ \eqref{eqn: def of trace push} }{=}   \int _{\tY}       g^*b  \cup j^* a   . \end{eqnarray*} 
                                When we use the equality $f^*r_1= r_2f^*$, the right $f^*$ is the relative pullback while the left $f^*$ is the \lq\lq absolute" pullback.
\end{proof}

\begin{Rmk}
For non-separated \'etale groupoids which admit an open covering whose elements are paracompact, 
Hausdorff, locally compact, and of uniformly bounded cohomological dimension, 
Crainic and Moerdijk \cite{CM2} constructed  compactly supported cohomology, 
adjoint operations $f_!$, $f^!$, a cap product, Verdier duality, etc.
\end{Rmk}

\subsection{The smooth case} \label{sec: smooth case} 
In this section, we assume that $\tX$ is a smooth {\em complex analytic} DM stack, which is
possibly not separated over $\bfk$.  We set $m$ to be the real dimension of $\tX$.

In this setting, we may express the trace map as a sum of integrals of compactly supported differential forms as follows.
Fix a closed substack $\tZ$ of $\tX$. Let $\Gamma _{\tZ}$ (resp. $\Gamma_{c}$) be global 
section functor with support on $\tZ$ (resp. with compact support). 
Let $\{ \tX_i \}_i$ be a finite open covering of $\tX$ where each $\tX_i$ is  separated.  
Consider the Mayer-Vietoris functor for $\cO_{\tX}$-modules
\[ \label{eq: MV}
\mathcal{MV}(\cF) = \cdots \to \oplus _{i <  j} \iota_{(\tX_i \times_{\tX} \tX_j)!} \cF \to \oplus _i  \iota_{\tX_i!} \cF.
\]
Let $\dR$ be the sheafify $C^{\infty}$ de Rham complex of $\tX$.  Applying $\mathcal{MV}$ yields a double complex.  We take the compactly supported cohomology of the total complex to obtain
\begin{equation} \label{eqn: MV seq} \cdots \to \oplus _{i <  j} \Gamma _{c} (\tX_i \times_{\tX} \tX_j,  \dR ) \to \oplus _i  \Gamma _{c} (\tX_i ,\dR ) \to 0 ,\end{equation}
which we denote by $MV_c (\{ \tX_i\}_i , \dR )$.

The trace map fits into the following commutative diagram of complexes in the derived category $D^+(\bfk )$ of lower bounded complexes of $\bfk$-modules
{\small  \[ \xymatrix{ \RR \Gamma _{\tZ} \bfk_{\tX} \ar[r] \ar[d]  &  H^m_{\tZ} (\tX, \bfk_{\tX}) [-m]  \ar[d] \ar@/^10pc/[dd]^{\int _{\tX}} \\
              MV_c (\{ \tX_i\}_i , \dR ) \ar[r] &  
               \oplus _i  H^m_{c} (\tX_i , \dR ) / \oplus _{i>  j} H^m_{c} (\tX_i \times_{\tX} \tX_j,  \dR )) [-m]  \ar[d]^{\sum \int_{\tX_i}}     \\
               & \bfk [-m] .} \] }
   
   We {\em define} $H^*_c (\tX)$ as the cohomology of   the total complex of the Mayer-Vietoris sequence above following \cite{CM2},
    which shows in particular that the definition does not depending on the choices of separated open coverings and c-soft resolutions.
         Hence we have
         \[ H^m_{\tZ } (\tX ) \to H^m_c (\tX) \xrightarrow{\int _{\tX}}  \bfk .\]

   In fact we can do more, i.e., in the (possibly) non-separated case we can define an integration along  the fiber as follows.        
   Let $f : \tX \to \tM$ be a {\em smooth} map between smooth complex analytic DM stacks.  
Let $\tW_{\tM}$ be a regular function on $\tM$. Let $\{ \tM _j \}_{j\in J}$ be a collection of separated open substacks covering $\tM$
such that $\tX _i$ maps to $\tM _{j_i}$ for some assignment $j_{-} : I \to J$. We may assume that $j_{-}$ is an injective assignment,
allowing repeated elements $\tM_j$ but with different index $j$.  Fix well-orderings on $I$ and $J$ such that the map $j_{-}$ preserves
the well-ordering. 

Let $j_{\vec{i}} := (j_{i_1}, ..., j_{i_k})$, $\tX_{\vec{i}} :=  \tX_{i_1} \ti _{\tX} ... \ti _{\tX} \tX_{i_k}$, and $\tM_{j_{\vec{i}}} :=  \tM_{j_{i_1}} \ti _{\tM} ... \ti _{\tM} \tM_{j_{i_k}}$.
Then the usual pushforward maps (given by integration along the fiber of the smooth submersion maps $f |_{\tX_{\vec{i}} }$ in the $C^{\infty}$ sense),
\[
( f |_{\tX_{\vec{i}} })_{ *} : \Gamma_c (\tX_{\vec{i}} , \Dolbb) \to \Gamma_c (\tM_{j_{\vec{i}}}, \Dolbb[\dim f])
\]
 for all strictly-increasing-ordered multi-indices $\vec{i}=(i_1, ..., i_k)$,
together form a map
\begin{align}  f_* : MV_c (\{ \tX_i \}_i , (\Dolbb ,   \bar{\partial}  + f^*d\tW_{\tM} ))    \to    MV_c (\{ \tM _j \}_j ,  (\Dolbb ,   \bar{\partial}  + d\tW_{\tM} )) .\end{align}

This is in fact a cochain map since it commutes with all 3 differentials, i.e.,
\[
\sum _{\vec{i},  |\vec{i}|=l} (f|_{\tX_{\vec{i}}})_* \circ \delta_{MV} =  \delta_{MV}  \circ \sum _{{\vec{i}}',  |{\vec{i}}'|=l+1} (f|_{\tX_{{\vec{i}}'}})_*.
\]
Indeed, the standard argument that the $C^\infty$ de Rham differential commutes with $ (f |_{\tX_{\vec{i}} })_{ *}$ can be adapted to our situation (of a holomorphic fibration) to show $f_*$ commutes with 
$\bar{\partial}$.  Similarly,  $f_* ( \bullet \wedge f^*d\tW_{\tM}) = d\tW_{\tM} \wedge f_* \bullet$ (using \cite[Lemma 7.2.6]{RS}).  

In conclusion, we obtain an induced map on hypercohomology
 \begin{align}  f_* :  \HH ^* _{c}( \tX, (\Omega _{\tX }^{\bullet} , f^* d\tW_{\tM})) \to \HH^{*-2\dim f}_{c} (\tM, (\Omega _{\tM }^{\bullet}, d\tW_{\tM} )) ,
 \label{eq: pushforward map}
 \end{align}
 where $\HH ^* _{c}$ is defined to be the cohomology of the Mayer-Vietoris sequence. 
 Note that $f_*$ is independent of the choices of such coverings $\{ \tX_i\}_i $, $\{ \tM_j\} _j $. We call  $f_*$ {\em integration along the fiber}.
When $\tM = \Spec \bfk$, we denote  $f_*$ also by $\int _{\tX}$.

For a smooth complex analytic DM stack $\tX$ of dimension $n$, 
we have the composition of maps \begin{equation} \label{eqn: com of integration maps}
\xymatrix{        H^n _{\tZ} (\tX, \Omega ^n_{\tX}) \ar[r]
&  H^{2n}_{\tZ} (\tX) \ar[r] \ar@{..>}[d]_{\text{trace map}} & H^{2n}_c (\tX) ,  \ar@/^.5pc/[ld]^{\text{integration}}  \\
                                            & \bfk  . &
}
\end{equation}
Here the first map is given by
\[ \Omega ^n_{\tX} [-n] \to (\Omega ^{\bullet}_{\tX}, d) \sim_{qis} \underline{\bf k} \] 
and the vertical dotted map is defined in \eqref{eq:def trace1}  whenever $\tX$ is the analytification of an DM stack, in which case
the diagram obviously commutes. Hence by abusing notation we will also denote the composited map by 
\begin{eqnarray} \label{eqn: def of integration map} \int _{\tX} :  H^n _{\tZ} (\tX, \Omega ^n_{\tX}) \to \bfk . \end{eqnarray}

             \begin{Lemma}\label{lemma: integration along the fiber}   Consider the abstract wedge product described in  \S \ref{sec:  wedge product}.  The following holds for integration along fibers.
             \begin{enumerate}

                 \item (functoriality) Consider a smooth map $\tM\to \tL$ between smooth complex analytic DM stacks with $\tW_{\tM} = g^* \tW _{\tL}$ for some $\tW _{\tL}$. Then 
              \[          f_* \circ g_* = (f\circ g )_* :          \HH _c (\tX, (\Omega _{\tX }^{\bullet} , (g\circ f)^* d\tW_{\tL})) \to  \HH _c (\tL, (\Omega _{\tL }^{\bullet} ,  d\tW_{\tL})) . \]

             \item  (projection formula) 
             Let $\tW_1$, $\tW_2$ be regular functions on $\tM$. Then
             for $a \in \HH ^*_c( \tX , (\Omega _{\tX }^{\bullet} , f^* d\tW_1))$, $b \in \HH^* (\tM, (\Omega _{\tM }^{\bullet} ,  d\tW_2))$, we have
             \[ f_* (a \wedge f^* b) = f_* a \wedge b \] in  $\HH^*_c (\tM, (\Omega _{\tM }^{\bullet} ,  d\tW_1+ d\tW_2))$.

             \item  (base change)
             Consider a fiber product of (possibly non-separated) pure-dimensional smooth complex analytic DM stacks 
\[ \xymatrix{ \tY \ar[r]^{j} \ar[d]_{g} & \tX \ar[d]^{f} \\
         \tN \ar[r]_{i} & \tM }\] 
         with $f$ smooth.  
         Then for $ a \in \HH^*_{c} ( \tX,  (\Omega ^{\bullet}_{\tX}, f^*d\tW_{\tM})   )$  we have \[  i^* f_* a = g_*j^* a \]
         in $\HH^*(\tN,  (\Omega ^{\bullet}_{\tN}, i^*d\tW_N))$. Furthermore if $i$ is proper, the equality holds
         in $\HH^*_c(\tN,  (\Omega ^{\bullet}_{\tN}, i^*d\tW_N))$.
             \end{enumerate}
             \end{Lemma}
      \begin{proof} (1) is clear.
     For (2),  let $\eta: \mathcal MV \to Id$ be the natural isomorphism of functors induced by the summing map.
     Then we have the following realization of the abstract wedge product in $D^+( \bfk )$.
\tiny
           \[
       \begin{tikzcd}\label{diag: MV C wedge} 
        \RR \Gamma _{c} (\Omega _{\tX}^\bullet ,  d\tW _1 ) \ot_{\bfk} \RR \Gamma  (\Omega _{\tX}^\bullet , d\tW _2 ) \ar[r] \ar[d, "{\mathbb L \Delta^*}"] &   MV_c(\{ \tX_i\}_i, (\Dolbb, \bar{\partial} + d\tW_1 )) \ot_{\bfk}   
        \Gamma(\mathcal{MV} (\{ \tX_i\}_i, (\Dolbb, \bar{\partial} + d\tW_2 )) \ar[d, "{\Delta^* = \mathbb L \Delta^* \text{ (by Malgrange's Theorem) }}"]  \\
                 \RR \Gamma _{c} (\Omega _{\tX}^\bullet \ot_{\cO_{\tX}} \Omega _{\tX}^\bullet,  d\tW _1 \boxplus d\tW_2)  \ar[d, "="]  \ar[r] & \Gamma_{c} (\mathcal{MV} (\{ \tX_i\}_i, (\Dolbb, \bar{\partial} + d\tW_1 )) \ot_{\cO_{\tX}}   \mathcal{MV}(\{ \tX_i\}_i, (\Dolbb, \bar{\partial} + d\tW_2 )))\ar[d, "{\Gamma_c(Id \otimes \eta)}"]  \\
         \RR \Gamma _{c} (\Omega _{\tX}^\bullet \ot_{\cO_{\tX}} \Omega _{\tX}^\bullet,  d\tW _1 \boxplus d\tW_2)  \ar[d, "{\wedge}"]  \ar[r] & MV_{c} (\{ \tX_i\}_i, (\Dolbb \otimes_{\cO_{\tX}} \Dolbb, (\bar{\partial} + d\tW_1) \ot 1 + 1\ot (\bar{\partial}  + d\tW_2)) \ar[d, "{\wedge}"] \\
          \RR \Gamma _{c} (\Omega_{\tX}^{\bullet} , d\tW_1 + d\tW_2 )    \ar[r]             
                & MV_{c} (\{ \tX_i\}_i, (\Dolbb , \bar{\partial} + d\tW_1 + d\tW_2))   \\
             \end{tikzcd}
             \]
             \normalsize     
where  $b|_{\tX_{\vec{i}}} := \eta(b|_{\tX_i})$ denotes the restriction of the class $b$ to the separated DM stack $\tX_i$ which lies in $\Gamma_c(\tX_i, (\Dolbb, \partial + dW_2))$.
  The right hand side of the diagram identifies the abstract wedge product as
  \begin{equation} \label{eq: wedge MV}
  (a_{\vec{i}} )_{\vec{i}} \wedge b =  (a_{\vec{i}} \wedge b|_{\tX_{\vec{i}}})_{\vec{i}}.
  \end{equation}    
       Hence
             \begin{align*}
             f_* (a \wedge f^*b) & = f_* (a_{\vec{i}}\wedge  f^*b|_{\tX _{\vec{i}}}) & \text{by \eqref{eq: wedge MV}}\\
             & = f_* (a_{\vec{i}}) \wedge  b|_{\tM _{j_{\vec{i}}}} & \text{by the usual projection formula} \\
             & = f_*a \wedge b & \text{again by \eqref{eq: wedge MV}.}
             \end{align*}

       (3) follows from the base change property for $(f|_{\tX_{\vec{i}}})_{ *}$ (this holds for any $C^\infty$ map).  
        \end{proof}

\begin{Cor}\label{cor: base change Poincare}

Consider a fiber product of pure-dimensional smooth complex analytic DM stacks 
\[ \xymatrix{ \tY \ar[r]^{j} \ar[d]_{g} & \tX \ar[d]^{f} \\
         \tN \ar[r]_{i} & \tM }\] 
         with $f$ smooth and $i$ proper. 
         Let $\tW$ be a regular function on $\tM$ with $\tW |_{\tN} = 0$ and let $\tZ$ be a proper closed substack of
         $\tX$.    
         Fix $\eta_{{\tN}} \in \HH^{2(\dim \tM - \dim \tN)}  (\tM, (\Omega ^{\bullet}_{\tM}, - d\tW)) $ and assume that $\int _{\tM} b \wedge \eta _{\tN} = \int_{\tN} i^*b $
          for every $b\in \HH^*_{c} (\tM, (\Omega ^{\bullet}_{\tM}, d\tW ))$.
         Then for any $a \in \HH^{*}_{\tZ}(\tX, (\Omega ^{\bullet}_{\tX}, f^*d\tW))$, we have
         \begin{eqnarray*}\label{eq:base change Poincare} \int _{\tX} a \wedge f^* \eta _{\tN} = \int _{\tY} j^* a  .             \end{eqnarray*}
       
\end{Cor}

\begin{proof}
We assert the following sequence of equalities 
\begin{align*}    \int _{\tX} a \wedge f^* \eta_{\tN}  
 = \int _{\tM} f_{ *} {a}  \wedge   \eta_{\tN}  =  \int _{\tN}      i^* f_* {a} 
 =  \int _{\tN}     g_*  j^* {a} 
 =    \int _{\tY}    j^* a .
\end{align*} 
The first equality is  by the projection formula and the functoriality in Lemma \ref{lemma: integration along the fiber}. 
The second equality uses the assumption explicitly stated in the lemma. 
The last two equalities follow from the base change and the functoriality in Lemma \ref{lemma: integration along the fiber} respectively.
\end{proof}

   \begin{Lemma} \label{lemma: projection for Q}
    Let $\pi: Q \to \tX$ be a vector bundle on a smooth complex analytic DM stack $\tX$ and let $\tZ$ be a proper closed substack of $\tX$. Then
 for $a\in H^*_{\tZ} (\tX, (\Omega ^{\bullet}_{\tX} , d\tW)) $ and 
   $b \in  H^*_{\tX} (Q, (\Omega ^{\bullet}_{Q}, - \pi^* d\tW))$ we have
   \[ \int _{Q} \pi^* a \wedge b = \int _{\tX} a \wedge \pi _* b , \]
   where $\pi^* a \in H^*_{\pi^{-1}\tZ }(Q, (\Omega ^{\bullet}_{Q} , d\tW))$, $\pi_* b \in H^* (\tX, (\Omega ^{\bullet} _{\tX}, -d\tW ))$.
 \end{Lemma}            
 
  \begin{proof}
 The proof is the same as that of Lemma~\ref{lemma: integration along the fiber} (2).  The only modification is by replacing the support conditions. 
 \end{proof}

\subsection{The separated case}

Let us now drop the assumption that $f: \tX \to \tM$ is smooth.  Instead, suppose that
there is an open immersion $\tX \to \tY$, where $\tX, \tY$ are smooth DM stacks of finite type, and a {\em proper} map $g: \tY \to \tM$ extending $f$.
Then applying Nironi's Grothendieck-Serre duality \cite{Nironi}  for DM stacks to the complex $(\wedge ^{\bullet} T_{\tY} , \iota _{f^*d\tW_{\tM}})$, we obtain 
an algebraic version of the above map 
$$  f_* :  \HH ^* _{\tZ}( \tX, (\Omega _{\tX }^{\bullet} , f^* d\tW_{\tM})) \to \HH^{*-2\dim f} (\tM, (\Omega _{\tM }^{\bullet}, d\tW_{\tM} )) , $$
where $\tZ$ is a closed substack of $\tX$ which is proper over $\bfk$.

\begin{Rmk}
The existence of an open immersion $\tX \to \tY$ as above implies that $f$ is separated since $\tY$ is separated over $\tM$.  
On the other hand, at least for algebraic spaces, Nagata's Theorem implies the converse, i.e., that such an open immersion exists whenever $f$ is separated \cite{CLO}.
\end{Rmk}

\section{Factorizations and twisted Hodge cohomology}\label{sec: B}

\subsection{Wedge product for the twisted Hodge complex} \label{sec: wedge product}

Let $\tX$ be a smooth DM stack with two global functions $\tW_1, \tW_2$ and two closed substacks $\tZ_1, \tZ_2$. There is a wedge product operation defined by the composition of the following maps
\begin{equation} \label{kunneth} 
\begin{array}{ccll}
& & \HH^*_{\tZ_1}(\tX, (\Omega ^{\bullet}_{\tX}, d\tW_1 )) \ot _{\bf k}\HH^*_{\tZ_2}(\tX, (\Omega ^{\bullet}_{\tX}, d\tW_2 ))  & \\
&  \xrightarrow{\sim} & \HH^*_{\tZ_1 \times \tZ_2} (\tX \times \tX, ( \Omega ^{\bullet}_{\tX}, d\tW_1) \boxtimes _{\cO_{\tX}}  ( \Omega ^{\bullet}_{\tX}, d\tW_2) ) 
& \text{ by K\"unneth isomorphism}       \\
& \to  &  \HH^*_{\tZ_1 \cap \tZ_2}(\tX, ( \Omega ^{\bullet}_{\tX}, d\tW_1) \ot _{\cO_{\tX}}  ( \Omega ^{\bullet}_{\tX}, d\tW_2) )  & \text{ by diagonal pullback}\\
& \xrightarrow{\wedge} & \HH^*_{\tZ_1 \cap \tZ_2}(\tX,  \Omega ^{\bullet}_{\tX}, d(\tW_1 + \tW_2))  &  \text{ by wedge product. }  \end{array} \end{equation}

\subsection{Tensor product for factorizations} \label{sec:tensor}
Let $ D_{fl}(\tX, \tW)$ denote the full subcategory of the co-derived category of $\cO_{\tX}$-module factorizations consisting of objects which are co-quasi isomorphic to objects with flat components.   There is a tensor product functor (see \cite[{\S}B.2.2]{EP}) 
\begin{align*}
\bullet \ot_{\cO_\tX}^{\mathbb L} \bullet : D_{fl}(\tX, \tW_1) \times D(\tX, \tW_2) & \to D(\tX, \tW_1 + \tW_2) \\
(\mathcal F, \mathcal G) & \mapsto \cF' \ot _{\cO_\tX} \cG
\end{align*}
where $\cF'$ is a flat replacement of $\cF$.
From this we deduce
\begin{Lemma}
Let $\tX$ be  a smooth DM stack and $\tW$ be a global function. Let $\cR$ be a resolution of $\cO_{\tX}$ and $\cF \in D_{fl}(\tX, \tW)$ be a factorization.  Then
\begin{equation}
\cF \cong \cF \ot^{\mathbb L}_{\cO_{\tX}} \cR.
\label{eq: fact res}
\end{equation}
\end{Lemma}
\begin{proof}
As the functor is well-defined, the lemma is the special case where  $\tW_1 = \tW, \tW_2 = 0$ and  $\cG = \cR$.
\end{proof}

\subsection{Super trace for factorizations}\label{sec:super trace}
Consider a DM  stack $\tX$. 
Let $\tW$ be a regular function on $\tX$ and $\cH$ and $\cG$ be $\cO_{\tX}$-factorizations of $(\tX,\tW)$.
\begin{Def} \label{def: perfect}
We say that $\cH$ is {\em perfect} if there exists a surjective, \'etale map from a  variety $T$ to $\tX$ 
such that $\cH$ is isomorphic in the derived category $D(T, \tW |_T)$ of $\cO_{\tX}$-factorizations for $\tW |_T$ 
to a factorization with locally-free components of finite rank.
\end{Def}

\begin{example}
Any virtual factorization $\mathbb K( \alpha, \beta)$ is perfect (see Proposition~\ref{prop: perfect})
\end{example}

Let $\cF$ be a $\cO_{\tX}$-factorization for the $0$ function. 
We have $\mathrm{Hom}$-$\otimes$ adjunction:
\begin{equation} \label{homtensor}
\text{Hom}_{D(\tX,\tW)}(\cF \otimes^{\mathbb L }_{\cO_{\tX}}  \cH, \cG) \xrightarrow{\phi} \text{Hom}_{D(\tX, 0)}(\cF, \mathbb R \cH om(\cH,\cG)).
\end{equation}
\
\begin{Rmk}
In our case, $\cH = \cG = \mathbb K( \alpha, \beta)$  is a virtual factorization  hence $\cH,\cG$ are both perfect and $\cO_{\tX}$-flat.
In particular,  $\cF \otimes^{\mathbb L }_{\cO_{\tX}}  \cH = \cF \otimes_{\cO_{\tX}}  \cH$.  On the other hand, $\mathbb R \cH om(\cH,\cG)$ is less straight-forward.  
It can be defined, for example, by taking an injective replacement of $\cG$.
\end{Rmk}
Setting $\cF = \mathbb R \cH om(\cH,\cG)$ we get an evaluation map
\[
ev : \mathbb R \cH om(\cH,\cG) \otimes^{\mathbb L }_{\cO_{\tX}} \cH \to \cG
\] 
defined as $ev= \phi^{-1}(Id_{ \mathbb R \cH om(\cH,\cG)})$.

Setting $\cF = \cH^\vee \otimes^{\mathbb L }_{\cO_{\tX}} \cH$ and $\cG = \cH$ in \eqref{homtensor} yields a map
\[
\psi: \cH^\vee \otimes^{\mathbb L }_{\cO_{\tX}} \cH \to \RR \cH om(\cH,\cH).
\]
defined as $\psi= \phi(ev \otimes Id_\cH)$.

\begin{Lemma}
If $\cH$ is perfect then $\psi$ is an isomorphism in $D(\tX,\tW)$.
\end{Lemma}
\begin{proof}
It is enough to show that $\psi$ is an isomorphism locally (see, e.g., \cite[Corollary 2.2.7]{MF}).  
Since $\cH$ is perfect, this means we may assume that $\cH$ is a complex of locally free coherent sheaves. 
In this case $\psi$ is component wise, the usual isomorphism from $ \cH^\vee \otimes_{\cO_{\tX}} \cH$ to $\cH om(\cH,\cH)$.
\end{proof}

Hence, when $\cH$ is perfect we may define a supertrace map
$$\mathrm{str}: \mathbb R \cH om(\cH,\cH) \to \cO_{\tX}$$
    by $\mathrm{str} = ev \circ \psi^{-1} $.

Let $\cS$ be a sheaf of abelian groups on $\tX$.  For any \'etale map $U \to \tX$ where $U$ is a scheme, we define the \emph{subsheaf of }$\cS$\emph{ with supports in }$\tZ$ by
\[
\gamma_\tZ (\cS)(U) := \Ker (\cS (U) \to \cS (U \times_{\tX} (\tX-\tZ)).
\]
This is a left exact functor and we define the \emph{local cohomology with support along }$\tZ$ as the hypercohomology of  corresponding right derived functor  
\[
H^*_{\tZ}(\tX, \cS) := \HH^*(\tX, \RR\gamma_\tZ(\cS)).
\]
\label{sheafylocalcohom}
In general for a complex $\cS$ of sheaves of abelian groups on $\tX$ or a factroization on $\tX$ for the $0$ function, 
we defined $H^*_{\tZ}(\tX, \cS) = H^*\Gamma _{\tZ} ( \God \cS)$ where 
$\God \cS$ is the Godement resolution of $\cS$; see \S \ref{sec: Godement} and Lemma \ref{lem: God acyclic}

Let $\cF$ be a factorization of the $0$ function and let $\cH$ be a perfect object of $D(\tX,\tW)$.
Consider $\mathrm{str}\ot \mathrm{id}_\cF: \RR\cE nd (\cH) \ot^{\mathbb L }_{\cO_{\tX}}  \cF \to \cO_{\tX} \ot^{\mathbb L }_{\cO_{\tX}} \cF  = \cF $. 
This induces a map \[ \RR\gamma_{\tZ}(\mathrm{str}\ot \mathrm{id}_\cF): \RR\gamma _{\tZ} (\RR\cE nd (\cH) \ot^{\mathbb L }_{\cO_{\tX}}  \cF)  \to \RR\gamma _{\tZ} (\cF). \]

\subsection{Atiyah class of a factorization} \label{subsec: Atiyah class}
We now define the Atiyah class of a factorization following \cite{KP}. In this subsection $\tX$ can be any smooth DM stack
and $\tW$ can be any regular function on $\tX$. 

\subsubsection{}
Let $\Omega ^{d\tW}_\tX$ denote the factorization
\begin{eqnarray}\label{Omega dW}  \Omega _\tX \xrightarrow{0} \cO _{\tX} \xrightarrow{d\tW} \Omega _\tX  \end{eqnarray}
for $(\tX, 0 )$. 
For a  factorization $\cH$ on $(\tX, \tW)$, we consider the tensor product of factorizations
$ \cH \ot_{\cO_{\tX}} \Omega _\tX^{d\tW}$. 

Define the 1-st jet factorization of $\cH$ as follows. 
As a $\ZZ/2\ZZ$-graded $\CC$-vector space, it is a sum $ \Omega ^{d\tW}_\tX \ot_{\cO_{\tX}} \cH [-1] \oplus \cH $.
Its $\cO _\tX$-module structure is, by definition,
\[   f \cdot (\sigma, s ) = (f \sigma + df  {\wedge} s , f s ) \text{ for } f \in \cO _\tX , \  (\sigma, s ) \in \Omega ^{d\tW}_\tX \ot \cH [-1] \oplus \cH . \]
The factorization structure for $\tW$ is defined by
\[ (\sigma , s) \mapsto (\delta _{ \Omega ^{d\tW}_\tX  \ot \cH } (\sigma) + 1\ot s, \delta _\cH(s) ) , \]  
where $\delta _{\mathcal E}$ means the boundary map of the factorization $\mathcal E$.

This yields a short exact sequence of $\cO_\tX$-module factorizations of $\tW$
\begin{eqnarray}\label{jet seq} 0 \to  \Omega ^{d\tW}_\tX \ot \cH [-1] \to   \Omega ^{d\tW}_\tX \ot \cH [-1] \oplus \cH \to \cH \to 0  \end{eqnarray}
and hence a map 
\[
\widehat{at}(\cH) : \cH \to  \Omega ^{d\tW}_\tX \ot_{\cO_{\tX}} \cH
\]
in the derived category of factorizations for $(\tX, \tW)$. 
We call $\widehat{at}(\cH)$ the {\em Atiyah class} of $\cH$ which can also be viewed as an element
\[ \widehat{at}(\cH) \in H^0(\RR\Hom (  \cH ,      \Omega ^{d\tW}_\tX \ot_{\cO_{\tX}} \cH   )) . \]

\begin{Rmk} \label{rmk: functor at}
Consider a map $f : \tY \to \tX$ between smooth DM stacks. 
It is clear that $$f^*\widehat{at}(\cH) = \widehat{at}(\LL f^*\cH) : \mathbb L f^*\cH \to  \Omega ^{df^*\tW}_\tY \ot_{\cO_{\tY}} \mathbb L f^*\cH . $$
\end{Rmk}

\begin{Rmk} \label{rmk: Cech functor}
\begin{enumerate} 
\item Let $E$ be a vector bundle on $\tX$. Then there is an \'etale atlas $Y\to \tX$ where $E|_Y$ has a connection.
Denote by $\vC$ the \v Cech functor giving the \v Cech resolution $\vC E$ of $E$. Then we have a component-wise connection $\nabla$
on $\vC E$, i.e., $\nabla$ is a collection $\{ \nabla ^{(i)}\}_{i\in \ZZ}$ where $\nabla ^{(i)}$ is a usual connection 
$\vC ^i E \to \vC ^i E \ot \Omega _{\tX}^\bullet$. 
The failure whether it is a cochain map is measured by a cochain map $[\nabla, \check{\delta}] : \vC E \to \vC E \ot \Omega _{\tX}^\bullet [1]$ 
which represents the usual Aityah class of $E$. The usual \v Cech representative in literature is  the composition $E \to \vC E \to \vC E \ot \Omega _{\tX}^\bullet [1]$.

\item Suppose that there is a component-wise connection $\nabla :  \cH \to \Omega _{\tX}^1 \ot \cH$.
Using the connection we take a component-wise right splitting $\cO_{\tX}$-module morphism $\cH \to   \Omega ^{d\tW}_\tX \ot \cH [-1] \oplus \cH $ of the form $(\nabla, \mathrm{id}_{\cH})$.
This yields an expression for $\ath \cH$ given by the cochain map 
\begin{equation}\label{eqn: at nabla}  (\mathrm{id}_{\cH} , [\nabla, \delta _{\cH}]) : \cH \to  \Omega ^{d\tW}_{\tX} \ot \cH  = \cH \oplus \Omega _{\tX}^1 [1] \ot \cH  \end{equation}
between factorizations whenever $\cH$ has a connection. 

\item Whenever $\cH$ is a matrix factorization, we can obtain a \v Cech representative of $\ath ({\cH})$ using local connections.

\end{enumerate}
\end{Rmk}

\begin{Rmk}\label{rmk: cech vir fact is flat} 
Starting from the de Rham differenital $d_{dR} : \cO _{\tX} \to \Omega _{\tX}^1$, by applying the Godement resolution followed by the Thom-Sullivan functor 
(see Appendix \ref{sec: TS}) we first obtain an induced connection 
$d_{GdR} : \God \cO _{\tX} \to \God\cO_{\tX}\ot \Omega _{\tX}$ compatible with the cosimplicial maps and hence
an induced connection $d_{ThGdR} : \mathrm{Th}\God \cO _{\tX} \to \mathrm{Th}\God\cO_{\tX}\ot \Omega _{\tX}$ compatible with the
differential of the complexes.
Let $\Kos$ be a Koszul complex. Since it is a complex of vector bundles, a \v Cech resolution $\vC \Kos$ has
a component-wise connection $\nabla _{\vC \Kos}$. We then obtain a component-wise connection $\nabla _{\vC \Kos} \ot \mathrm{id} + \mathrm{id} \ot d_{ThGdR}$ on the complex
$\vC \Th = \vC \Kos\ot \mathrm{Th}\God \cO _{\tX}$. 
Hence, in particular for a virtual factorization $\cH$, we have a \v Cech resolution $\vC (\cO_{\tX}) \ot \cH$ 
(which is flat by flatness of the \v Cech resolution and Remark~\ref{rmk: flat}) of $\cH$
with a connection $\nabla _{\vC\cH}$.
If $\cH_i$, $i=1, 2$ are virtual factorizations,  then $(\cH _1 \ot \vC (\cO _{\tX})) \ot (\cH_2 \ot \vC (\cO _{\tX}))$, which is co-quasi-isomorphic to $\cH_1 \ot \cH_2$,
has a connection $\nabla _{\vC\cH_1} \ot \mathrm{id} + \mathrm{id} \ot \nabla _{\vC\cH_2}$. 
\end{Rmk}

\begin{Rmk}\label{rmk: analytic At} Let $\tX$ be a smooth complex analytic DM stack.   An analytic Atiyah class can be defined for any $\cO_{\tX}$-factorization the same way (by working over the sheaf of analytic functions $\cO_{\tX}$).  We abuse notation and terminology by still simply calling this the Atiyah class and writing $\hat{at} (\cH)$. 
Afterall, these definitions are compatible.  Namely, if $\tX = \tY^{an}$ and $\cH$ is a $\cO_{\tY}$-factorization then 
\[
\widehat{at}(\cH^{an}) = (\widehat{at}(\cH))^{an}.
\] 
This analytic Atiyah class also has the following nice description.  Assume that $\tX$ is separated and that $\cH$ is a locally-free $\cO_{\tX}$-factorization.  Then $\widehat{at}(\cH)$ has a Dolbeault representative (see also \cite[Proposition 1.5]{KP}).
In this case, we may always choose a $C^{\infty}_{\tX}$-connection of $\cH\ot_{\cO_{\tX}} C^{\infty}_{\tX}$,
\[ \nabla: \cH\ot_{\cO_{\tX}} C^{\infty}_{\tX} \ra \cH \ot _{\cO _{\tX}}  (\mathscr{A}^{1, 0}_{\tX} \oplus \mathscr{A}^{0, 1}_{\tX} ) . \]
Note that the $(1,0)$-part $\nabla ^{1, 0}$ of $\nabla$ provides a component-wise splitting of (the analytic analog of) the exact sequence \eqref{jet seq} as $\cO_{\tX}$-modules
after the replacement of  $ \Omega _{\tX}^{d\tW} \ot _{\cO_{\tX}} \cH [-1]  $ by  $ \Omega _{\tX}^{d\tW} \ot _{\cO_{\tX}} \cH  \ot_{\cO_{\tX}} \sA ^{0, \bullet}_{\bar{\partial}} [-1]$.
Therefore the Atiyah class is represented by the Dolbeault cocycle
\begin{eqnarray}\label{At Dol} & & (\mathrm{id} _\cH , [\nabla ^{1, 0} , \delta _\cH  ] - \bar{\partial}(\nabla ^{1, 0}) )  \\ \nonumber & \in & 
\Gamma (\tX, (\mathscr{A}^{\le 1, \bullet }_{\tX} \ot \cE nd \cH , [\delta _\cH, \ ] + d\tW +   \bar{\partial} ) ) . \end{eqnarray}
\end{Rmk}

\begin{Rmk} \label{rmk: top c1 and cat c1}
Let $\tX$ be a smooth complex analytic DM stack.
When $\tW=0$, $\hat{\mathrm{at}}(\cH) = \mathrm{id}_\cH + \mathrm{at} (\cH)$, where  $\mathrm{at} (\cH)$
is the usual Atiyah class of the 2-periodic complex $\cH$ of $\cO_{\tX}$-modules.
If $\cH$ is a factorization whose zero-th component is the line bundle $L$ and other component 0, then the topological first Chern class $c^{top}_1 (L)$ 
is $\frac{i}{2\pi} \mathrm{at} (L)$ (see \cite{Huyb}).
\end{Rmk}

\subsubsection{}
We denote  by $\Upsilon(\Omega ^{\bullet} _{\tX}, d\tW) $ the factorization of $0$ defined as the folding of 
the complex  $(\Omega ^{\bullet} _{\tX}, d\tW)$.
For positive integers consider the $n^{th}$ iteration
\[ \widehat{at}^n(\cH)  : \cH \xrightarrow{\widehat{at}(\cH)} \Omega ^{d\tW}_\tX \ot \cH 
\xrightarrow{ \mathrm{id}_{\Omega ^{d\tW}_\tX} \ot \widehat{at}( \cH )}(\Omega ^{d\tW}_\tX)^{\ot 2} \ot \cH \to ... \to (\Omega ^{d\tW}_\tX)^{\ot n} \ot \cH   \]

Notice that   $(\Omega ^{d\tW}_\tX)^{\ot n}$ is the folding of the complex
\[
(\Omega ^{d\tW}_\tX)^{\ot n} =  \cO \to ... \to ((\Omega_\tX)^{\otimes k})^{\oplus {n \choose k}} \to ... \to (\Omega_\tX)^{\otimes n}
\]
Consider the following map $\phi_n$ of complexes:
\[ \xymatrix{   
\cO  \ar[r]^{ d\tW} \ar[d]_{=}  & (\Omega _{\tX})^{\oplus n} \ar[rr]^{d\tW} \ar[d]_{\sum} & &  (\Omega ^{\otimes 2}_\tX)^{\oplus {n \choose 2}} \ar[r] \ar[d]_{\sum} 
&  ... \ar[r]^{d\tW} & (\Omega ^{\otimes n-1}_\tX)^{\oplus n} \ar[r]^{d\tW} \ar[d]_{\sum} & \Omega ^{\otimes n}_\tX \ar[d]_{\sum} \\
\cO  \ar[r]^{n d\tW} \ar[d]_{=}  & \Omega _{\tX} \ar[rr]^{(n-1)d\tW} \ar[d]_{\frac{1}{n}} & &  \Omega ^2_\tX \ar[r] \ar[d]_{\frac{1}{n (n -1)}} 
&  ... \ar[r]^{2d\tW} & \Omega ^{n-1}_\tX \ar[r]^{d\tW} \ar[d]_{\frac{1}{n(n-1) \cdots 2}} & \Omega ^n_\tX \ar[d]_{\frac{1}{n!}} \\
 \cO  \ar[r]_{d\tW}  &  \Omega _{\tX} \ar[rr]_{d\tW} & &   \Omega ^2_\tX \ar[r]_{d\tW} &  ... \ar[r]_{d\tW} & \Omega ^{n-1}_\tX \ar[r]_{d\tW} & \Omega ^n_\tX  ,} \]
 where $\sum$ means the sum followed by the wedge product, i.e., the natural map $[\cO \to \Omega _{\tX}]^{\ot n} \to \Sym ^n[\cO \to \Omega _{\tX}]$.
Then, abusing notation and denoting by $\phi_n$ the induced map on the foldings tensor with $\cH$, we define, for large enough $n$
\[
e^{\widehat{at}(\cH)} := \phi _n (\widehat{at}^n(\cH)): \cH \to \cH \ot \Upsilon(\Omega ^{\bullet} _{\tX}, d\tW),
\] 
while the left side is stabilized after $n \ge \dim \tX$.
Notice that this agrees with the usual Chern character (i.e., the exponential of the usual Atiyah class) when $w=0$.
 
From now on assume that $\cH$ is perfect. 
View $e^{\widehat{at}(\cH)}$ as a map
\[
e^{\widehat{at}(\cH)}: \cO_{\tX} \to \RR\cE nd \cH \ot \Upsilon(\Omega ^{\bullet} _{\tX}, d\tW) .
\]
If $(\cH, \delta _\cH)$ is supported on a closed substack $\tZ$ (including the case when $\tZ=\tX$),  we have an isomorphism
\[
 \HH^*( \tX, \RR\cE nd \cH \ot \Upsilon(\Omega ^{\bullet} _{\tX}, d\tW))
\xrightarrow{\cong} H^*_{\tZ}( \tX, \RR\cE nd \cH \ot \Upsilon(\Omega ^{\bullet} _{\tX}, d\tW))
\] so that we may regard $e^{\widehat{at}(\cH)}$ as an element of the latter hypercohomology at $* = 0$.
Utilizing the  supertrace map described in \S\ref{sec:super trace} 
we get an induced map  
\[
\HH^0\RR\gamma_\tZ( \mathrm{str} \ot {\mathrm{id}_{\Upsilon(\Omega ^{\bullet} _{\tX}, d\tW)}}) : 
 H^0_Z( \tX, \RR\cE nd \cH \ot \Upsilon(\Omega ^{\bullet} _{\tX}, d\tW)) \to   H^0_\tZ (\tX, \Upsilon(\Omega ^{\bullet} _{\tX}, d\tW) ).
\]
We define the {\em localized Chern character} of $\cH$ as the local cohomology element
\begin{eqnarray} \label{eq: local Chern}
 \ch ^{\tX}_\tZ (\cH)   & := & \HH^0\RR\gamma_\tZ( \mathrm{str} \ot {\mathrm{id}_{\Upsilon(\Omega ^{\bullet} _{\tX}, d\tW)}} ) (e^{\widehat{at}(\cH)} ) \\
& \in &   H^{0} _{\tZ} (\tX, \Upsilon(\Omega ^{\bullet} _{\tX}, d\tW) ) 
 =   H^{even} _{\tZ} (\tX, (\Omega ^{\bullet} _{\tX}, d\tW) ) , \notag \end{eqnarray}
 where $\mathrm{str} \ $ is the super trace operation.

\begin{Rmk} 
We can also define the localized analytic Chern character of a perfect factorization $\cH$ on a smooth complex analytic DM stack $\tX$, using the analytic Atiyah class (see Remark \ref{rmk: analytic At}) of
a perfect $\cO_{\tX}$-factorization $\cH$. Similarly, the definition is compatible with the algebraic one. Thus we again abuse terminology and notation by calling  it simply the localized Chern character and writing it  $\ch ^{\tX}_{\tZ} (\cH)$ if $\cH$ is supported on a complex analytic substack $\tZ$ of $\tX$.
\end{Rmk}

\subsubsection{}
Let $(\cH_1, \delta _{\cH_1})$ be a perfect factorization supported on a closed substack $\tZ_1$.
Let $\cH_2$ be a bounded perfect complex on $\tX$ supported on a closed substack $\tZ_2$.
Consider  $\cH = \cH_1\ot ^{ \LL}_{\cO_{\tX}} \cH_2$
as a perfect factorization for $\tW$---here we take the folding of $\cH_2$.

\begin{Lemma} \label{lemma: at mult} 
Suppose that $\cH_i$ admits a connection and every component of $\cH_i$ is $\cO_{\tX}$-flat for $i=1,2$. Then the following multiplicative properties hold.
\begin{eqnarray}  \label{at mult}      \ath \cH & = &  \ath \cH_1 \ot^{} \mathrm{id}_{\cH_2}  + \mathrm{id} _{\cH_1} \ot^{} \mathrm{at} \cH_2 ;  \end{eqnarray}
\begin{eqnarray}   \label{ch mult}        \ch^{\tX}_{ \tZ_1 \cap \tZ_2} \cH & = & \ch^{\tX}_{\tZ_1} \cH_1 \wedge \ch^{\tX}_{\tZ_2} \cH_2 , \end{eqnarray}
where $\mathrm{at} \cH_2$ is the $\ZZ_2$-folding of the usual Atiyah class for $\cH_2$ considered as a bounded perfect complex 
and $\ch^{\tX}_{\tZ _2} \cH_2$ is the localized Chern character of $\cH_2$ considered as a factorization.
\end{Lemma}

\begin{proof} 
By assumption we have connections $\nabla _{i}$ for each $\cH_i$ and thus a connection $\nabla := \nabla _1\ot \mathrm{id} + \mathrm{id} \ot \nabla _2$ for $\cH := \cH_1 \ot_{\cO_{\tX}} \cH_2$.
Hence we have the expression \eqref{eqn: at nabla} for the Atiyah class of $\cH$. From this 
the first equation \eqref{at mult} is immediate. 
The second equation \eqref{ch mult} follows  from  \eqref{at mult} and the equality
\[ \phi _{2n} ( \ath \cH_1 \ot^{} \mathrm{id}_{\cH_2}  + \mathrm{id} _{\cH_1} \ot^{}  \mathrm{at} \cH_2)^{2n}  = 
            \sum _{k=0}^n   \frac{1 }{k!}   \phi _{n+k} (\ath \cH _1)^{n+k} \wedge (\mathrm{at} \cH_2) ^k   ,\]
            for any $n$ with $n\ge \dim \tX$ (where the identity above uses the Koszul sign rule).
\end{proof}

Consider a map $f : \tY \to \tX$ between smooth DM stacks and let $\tZ ' := f^{-1} \tZ$.

\begin{Lemma} \label{lemma: functor tdch} In $H^{even} _{\tZ '} (\tY, (\Omega ^{\bullet} _{\tY}, f^* d\tW) )$ we have
 $f^* \ch^{\tX}_{\tZ} \cH = \ch^{\tY}_{\tZ'} \LL f^* \cH $.
\end{Lemma}

\begin{proof} 
This is immediate since all the maps $\RR\gamma_\tZ(\mathrm{str} \ot {\mathrm{id}_{\Upsilon(\Omega ^{\bullet} _{\tX}, d\tW)})}$, $\phi$,  $\widehat{at}(\cH)$
in derived categories  commute with $\LL f^*$.
\end{proof}

 We have the following multiplicative property for the tensor product of factorizations. 
 When $\tX$ is affine and the components of factorizations are vector bundles, it is Theorem 5.17 of Yu \cite{Yu}. 

\begin{Thm} \label{thm: mult chern}
Let $\cH_1$, $\cH_2$ be two perfect factorization on $\tX$ for $\tW_1, \tW_2$, respectively and let $n= \dim \tX$.
Assume that each of the components $\cH_i$, $i=1, 2$, are $\cO_{\tX}$-flat and admit connections.
Then as morphisms in derived category we have the following equality
\begin{equation} \label{eqn: Yu Mult} \phi_n (\ath \cH _1) \wedge \phi _n (\ath \cH_2) = 
 \phi _n (\ath (\cH _1 \ot \cH _2) )  \end{equation}
with the Koszul sign convention.
\end{Thm}

\begin{proof} 
Let $\nabla _i$ be connections of $\cH_i$.
Each sequence  \eqref{jet seq} for $\cH_i$ becomes component-wise split, i.e., there is a 
$\cO _{\tX}$-module homomorphism $(\nabla_i , \mathrm{id} ): \cH_i \to \Omega ^{d\tW} \ot \cH_i [-1] \oplus \cH_i $, which is not necessarily  a
chain map between factorizations.
The module homomorphism amounts to the connection property: $\nabla_i (fs) = \nabla_i (s) + df \wedge s $. 
We have  $\ath \cH_i = (\mathrm{id}, [\nabla_i , \delta _{\cH_i}] ) : \cH_i \to \Omega ^{d\tW} \ot  \cH_i $
and $\ath (\cH _1 \ot^{ } \cH _2) = (\mathrm{id}, [\nabla _1 \ot^{ } \mathrm{id}  + \mathrm{id} \ot^{ } \nabla _2 , \delta _{\cH _1 \ot  \cH _2} ] ) :
  \cH_1 \ot  \cH _2 \to \Omega ^{d\tW} \ot  (\cH_1 \ot  \cH _2) $. 
Now using these expressions of Atiyah classes, the proof follows by a formal algebraic computation as in 
the proof of \cite[Theorem 5.17]{Yu}. 
\end{proof}

Theorem above immediately implies the following.
\begin{Cor} \label{cor: multi of ch}
Let $\cH_1$, $\cH_2$ be two perfect factorization on $\tX$ for $\tW_1, \tW_2$, supported on $\tZ_1$, $\tZ_2$, respectively.
Assume that either (1) $\cH_i$, $i=1, 2$, are $\cO_{\tX}$-flat and allow connections; or (2) $\cH_i$ are virtual factorizations.
Then \begin{eqnarray}\label{eqn: multi of ch} \ch ^{\tX}_{ \tZ_1 \cap \tZ_2} ( \cH _1 \otimes_{\cO_{\tX}} \cH_2 )  = \ch ^{\tX}_{\tZ_1} \cH _1 \wedge  \ch ^{\tX}_{\tZ_2}  \cH_2 . \end{eqnarray}
\end{Cor}
\begin{proof} 
The first case is an obvious corollary of Theorem~\ref{thm: mult chern}. 
The second case follows from Remark \ref{rmk: cech vir fact is flat} and the first case. 
\end{proof}

\subsection{Atiyah classes and localized Chern characters in the $C^{\infty}$-case}
Let $\tX$ be a separated, smooth, finite type complex analytic DM stack and 
let $\cF$ be a $C^{\infty}_{\tX}$-module factorization for a function $\tW$ such that each component of $\cF$ as a $C^{\infty}_{\tX}$-module is
a finitely generated locally free module. We will simply say that $\cF$ is a $C^{\infty}_{\tX}$-module matrix factorization for $\tW$. 
For example, one can consider $\cH \ot_{\cO_{\tX}}C^{\infty}_{\tX}$ for any $\cO_{\tX}$-module matrix factorization $\cH$. 
Here again by a $\cO_{\tX}$-module matrix factorization we mean a factorization whose components are $\cO_{\tX}$ vector bundles.
We define the Atiyah class of $\cF$ as follows. First, we replace the $\cO_{\tX}$-factorization $\Omega _{\tX}^{dW}$ for the zero function by the $C^{\infty}_{\tX}$-factorization
\[ \Omega _{\tX, \bar{\partial}}^{dW} :=  \Omega _{\tX}^{dW} \ot _{\cO_{\tX}} \sA ^{0, \bullet}_{\bar{\partial}} \] for the zero function.
Here we omit the folding notation.
We take the $\ZZ/2\ZZ$-graded $\bf k$-vector space
\[ \Omega _{\tX, \bar{\partial}}^{dW}  \ot _{C^{\infty}_{\tX}} \cF [-1] \oplus \cF [-1] . \]
We define its $C^{\infty}_{\tX}$-module structure by
\[ f \cdot (\sigma, s ) := (f \sigma + (\partial f  + \bar{\partial} f )\wedge s , f s ) \text{ for } f \in C^{\infty}_{\tX} , \  (\sigma, s ) \in \Omega ^{d\tW}_{\tX, \bar{\partial}} \ot \cH [-1] \oplus \cH \]
and its $C^{\infty}_{\tX}$-factorization structure for $\tW$
by \[ (\sigma , s) \mapsto (\delta _{ \Omega ^{d\tW}_{\tX, \bar{\partial}}  \ot \cF } (\sigma) + 1\ot s, \delta _\cF (s) ).  \] 
By construction, we see that there is a short exact sequence of factorizations
\begin{eqnarray}\label{Cinfty jet seq} 0 \to  \Omega ^{d\tW}_{\tX, \bar{\partial}}  \ot_{C^{\infty}}  \cF [-1] \to   \Omega ^{d\tW}_{\tX, \bar{\partial}}  \ot_{C^{\infty}}  \cF [-1] \oplus \cF \to \cF \to 0.  \end{eqnarray}
Hence we obtain a class
\[ \hat{at} (\cF) \in  H^0 \RR\Hom (\cF,  \Omega _{\tX, \bar{\partial}}^{dW} \ot _{C^{\infty}_{\tX}} \cF) = \Gamma (\tX, \cE nd (\cF ) \ot  \Omega _{\tX, \bar{\partial}}^{dW} ) \]
corresponding to \eqref{jet seq} which we call the Atiyah class of $\cF$. 
For an  $\cO_{\tX}$-module matrix factorization $\cH$, we can tensor all elements in the construction by the identity to obtain the formula
\begin{equation} \label{eq: smooth atiyah class}
\hat{at} (\cH \ot_{\cO _{\tX}} C^{\infty}_{\tX}) = \hat{at} (\cH )  \ot 1.
\end{equation}

\begin{Rmk} As before 
we may always choose a $C^{\infty}_{\tX}$-connection of $\cF$,
\[ \nabla: \cF \ra \cF \ot _{\cO _{\tX}}  (\mathscr{A}^{1, 0}_{\tX} \oplus \mathscr{A}^{0, 1}_{\tX} ) . \]
Note that the $(1,0)$-part $\nabla ^{1, 0}$ of $\nabla$ provides a component-wise splitting of the exact sequence \eqref{Cinfty jet seq}.
Therefore the Atiyah class is represented by the Dolbeault cocycle
\begin{eqnarray}\label{Cinfty At Dol} & & (\mathrm{id} _\cF , [\nabla ^{1, 0} , \delta _\cF  ] - \bar{\partial}(\nabla ^{1, 0}) )  \\ \nonumber & \in & 
\Gamma (\tX, (\mathscr{A}^{\le 1, \bullet }_{\tX} \ot \cE nd \cF , [\delta _\cF, \ ] + d\tW +   \bar{\partial} ) ) . \end{eqnarray}
\end{Rmk}

We also define $e^{\hat{at} (\cF ) } : = \psi (\hat{at}^n (\cF))$
where $\hat{at}^n (\cF)$ is the composition $\cF \to \Omega _{\tX, \bar{\partial}}^{d\tW}  \ot_{C^{\infty} } \cF     \to ... \to   (\Omega _{\tX, \bar{\partial}}^{d\tW})^{\ot n}  \ot_{C^{\infty}} \cF$.
We need to explain $\psi$.
First consider the natural quotient map from $(\Omega _{\tX, \bar{\partial}}^{d\tW})^{\ot n}$ to $\Sym ^n \Omega _{\tX, \bar{\partial}}^{d\tW} $ which is the top line complex of 
\[ \xymatrix{   
\sA ^{0, \bullet}_{\bar{\partial}}   \ar[r]^{n d\tW} \ar[d]_{=}  & \sA ^{1, \bullet}_{\bar{\partial}} \ar[rr]^{(n-1)d\tW} \ar[d]_{\frac{1}{n}} & & \sA ^{2, \bullet}_{\bar{\partial}}     \ar[r] \ar[d]_{\frac{1}{n (n -1)}} 
&  ... \ar[r]^{2d\tW} &  \sA ^{n-1, \bullet} _{\bar{\partial}} \ar[r]^{d\tW} \ar[d]_{\frac{1}{n(n-1) \cdots 2}} & \sA ^{n, \bullet} _{\bar{\partial}}  \ar[d]_{\frac{1}{n!}} \\
\sA ^{0, \bullet}_{\bar{\partial}}  \ar[r]_{d\tW}  &  \sA ^{1, \bullet}_{\bar{\partial}}  \ar[rr]_{d\tW} 
& &   \sA ^{2, \bullet}_{\bar{\partial}} \ar[r]_{d\tW} &  ... \ar[r]_{d\tW} & \sA ^{n-1, \bullet}_{\bar{\partial}} \ar[r]_{d\tW} & \sA ^{n, \bullet}_{\bar{\partial}} } \]  
The map $\psi$ is defined to be the composition of the cochain maps above, after the folding, tensored with the identity of $\cF$.

If $\cF$ is supported on a closed sublocus $\tZ$  of $\tX$,
 we define the localized Chern character of $\cF$ as
\begin{eqnarray*} \ch ^{\tX}_{\tZ} (\cF)  & = &      \HH^0\RR\gamma_\tZ( \mathrm{str} \ot {\mathrm{id}_{\Upsilon(\sA ^{\bullet, \bullet} _{\tX}, d\tW + \bar{\partial} )}} ) (\psi ( e^{\widehat{at}(\cH)}) ) \\
& \in &   H^{0} _{\tZ} (\tX, \Upsilon(\sA ^{\bullet, \bullet} _{\tX}, d\tW + \bar{\partial} ) ).
                  \end{eqnarray*}
For an $\cO_{\tX}$-module factorization $\cH$ the identification
\begin{equation} \label{eq: smooth chern class}
\ch ^{\tX}_{\tZ} (\cH \ot_{\cO_{\tX} } C^{\infty}_{\tX}) =  \ch ^{\tX}_{\tZ} (\cH )
\end{equation}
follows from \eqref{eq: smooth atiyah class}.

We also have the multiplicative property (corresponding to  Corollary \ref{cor: multi of ch}), whose proof we omit since it is parallel to the given proof.
\begin{Cor} \label{cor: Cinfty multi of ch}
Let $\cF_1$, $\cF_2$ be two $C^{\infty}_{\tX}$ matrix factorization on $\tX$ for $\tW_1, \tW_2$, supported on $\tZ_1$, $\tZ_2$, respectively.
Then $$ \ch ^{\tX}_{\tZ_1 \cap  \tZ_2} ( \cF _1 \otimes_{C^{\infty}_{\tX}} \cF_2 )  = \ch ^{\tX}_{\tZ_1} \cF_1 \wedge  \ch ^{\tX}_{\tZ_2}  \cF_2 . $$
\end{Cor}

\subsection{Splitting Principle} \label{subsec: splitting prin}

Let $\tX$ be a separated, smooth, finite type complex analytic DM stack and let $E$ be a holomorphic vector bundle on $\tX$.
We consider the complete flag bundle $\tY$ of $E$. Let $\pi : \tY \to \tX$ be the projection and let
$Fl$ be the complete flag variety of $\CC ^n$ where $n$ is the rank of $E$.

\begin{Lemma} There is a quasi-isomorphism 
\begin{align*} \phi: (\Omega ^{\bullet}_{\tX}, d\tW )\ot _{\CC} (\bigoplus _{p, q} H^p(Fl, \Omega ^q _{Fl} ) [p+q]) \to \RR\pi_* (\Omega ^{\bullet}_{\tY}, \pi^*d\tW )) \end{align*}
in the bounded derived category of coherent sheaves on $\tX$, which extends 
$ \RR\pi_* \xi^*: (\Omega ^{\bullet}_{\tX}, d\tW ) \to \RR\pi_* (\Omega ^{\bullet}_{\tY}, \pi^*d\tW ))$, where  $\xi: \pi^*(\Omega ^{\bullet}_{\tX}, \pi^*d\tW ) \to (\Omega ^{\bullet}_{\tY}, \pi^*d\tW ))$ is the natural map.
\end{Lemma}

\begin{proof}
Consider $\bar{\partial}$-Harmonic $(1,1)$-forms 
representing the Chern classes of tautological quotient line bundles  on the flag bundle $\tY$. 
We take a suitable collection $\{ \eta _j \}_j$ of the wedge products of $\omega _i$ which,
under the restrction, represents a basis $\{ [\eta _j|_{Fiber}]\}_j$ of the Dolbeault cohomology of a (and hence every) fiber $Fiber$ of $\pi$.
For each integer $k\ge 0$ and $a \in \Omega ^k_{\tX}$, we now define $\phi (a \ot [\eta _j|_{Fiber}] ) = \pi^* a \wedge \eta _j$. This yields a cochain map $\phi$. 
Note that $\phi$ is a quasi-isomorphism since it is so locally where
 $\pi$ is a trivial fibration.
\end{proof}

\begin{Cor}\label{cor: splitting prin} The pullback $\pi^* : \HH^* (\tX, (\Omega ^{\bullet}_{\tX}, d\tW )) \to \HH^* (\tY, (\Omega ^{\bullet}_{\tY}, \pi^*d\tW ))$ 
is injective.
\end{Cor}

The vector bundle $\pi^*E$ has a filtration by vector bundles whose quotients are holomorphic line bundles $L_i$. Since $\Gamma (\tY, L_i \ot_{\cO_{\tY}} C^{\infty}_{\tY})$ is a 
finitely generated projective $C^{\infty} (\tY)$-module by Swan's theorem (\cite{Nes}),
$\pi^*E \ot_{\cO _{\tY}} C^{\infty}_{\tY}$ is isomorphic to $\oplus _i L_i \ot_{\cO_{\tY}} C^{\infty}_{\tY}$
 in the category of $C^{\infty}_{\tY}$ complex vector bundles on $\tY$. 
 Let $\sigma$ be a holomorphic section of $E$ and let $\tau$ be a section of $E^{\vee}$ such that $\lan \tau, \sigma \ran = \tW$. 
 Then $\pi^*\sigma$, $\pi^*\tau$ induce $C^{\infty}$-sections $\sigma _i$ of $L_i\ot C^{\infty}_{\tY}$, $C^{\infty}$-sections $\tau _i$ of $L_i^{\vee}\ot C^{\infty}_{\tY}$.
Hence if $m$ is the rank of $E$, 
the matrix factorization $\pi^* \{ \tau, \sigma \} \ot _{\cO _{\tY}} C^{\infty}_{\tY}  $ is quasi-isomorphic to $\{ \tau_1, \sigma_1\} \ot _{C^{\infty}_{\tY} } ... \ot _{C^{\infty}_{\tY} } \{ \tau_m, \sigma_m\} $
as  $C^{\infty}_{\tY}$-factorizations for the function $\pi^*\tW$.

\begin{Lemma}\label{lem minimal}  
Let $\tZ$ be the zero locus of the section $\tau \oplus \sigma$ of $E^{\vee}\oplus E$ and  let $\tZ_i$ be the zero locus of $\tau _i \oplus \sigma _i$ of $(L_i^{\vee} \oplus L_i)\ot C^{\infty}_{\tY}$. 
We have \begin{align} \label{eqn: splitting prin}
\ch^{\tY}_{\pi ^{-1}\tZ } (\pi^*\{ \tau ,\sigma \} )
=  \wedge _i \ch^{\tY}_{\tZ _i} (\{ \tau _i ,\sigma _i \}  ) \end{align}
in $\HH ^{*}_{\pi^{-1}\tZ} (\tY,  (\Omega ^{\bullet}_{\tY}, \pi^* d\tW ))$.
In particular, every homogeneous component of $\ch ^{\tX}_{\tZ} \{ \tau , \sigma \}$ has degree at least $2\rank E$.
\end{Lemma}

\begin{proof}
Equation~\eqref{eqn: splitting prin} follows from \eqref{eq: smooth chern class}, Corollary~\ref{cor: Cinfty multi of ch} and the discussion above.
It remains to prove the assertion on the degree. 
By  \eqref{eqn: splitting prin}, Corollary \ref{cor: splitting prin} it is enough to prove that 
every homogeneous component of $\ch^{\tY}_{\tZ _i} (\{ \tau _i ,\sigma _i \} )$ 
has degree at least $2$. For this,
let $\nabla _i$ be a connection on the complex vector bundle $C^{\infty}_{\tY} \oplus L_i^{\vee}$
and let $\delta _i$ be the odd degree differential of $\{ \tau _i, \sigma _i\}$ so that  $\{ \tau _i, \sigma _i\} = (C^{\infty}_{\tY} \oplus L_i^{\vee}, \delta _i)$.
We then consider the expression \eqref{Cinfty At Dol} of the Atiyah class of $\{ \tau _i, \sigma _i\} $.
The virtual rank of $\{ \tau _i, \sigma _i\}$ is zero.  Note also that the super trace of $[ \nabla ^{1,0}_i , \delta _i]$
vanishes since $[ \nabla ^{1,0}_i , \delta _i]$ is of  odd degree. Hence each component of the super trace of $e^{(\mathrm{id} , [\nabla _i^{1, 0} , \delta _i  ] - \bar{\partial}(\nabla _i^{1, 0}) )}$ 
has bidegree at least $(1, 1)$ which is (more than) enough to 
complete  the proof.  
\end{proof}

\section{Thom-Sullivan and Godement Resolutions} \label{sec: TS}
\subsection{The Thom-Sullivan Construction}
We very briefly treat the Thom-Sullivan construction here and prove only the properties we require.  
Please see, for example \cite{HS}, for much more detail.

 Let $\Omega_\bullet$ \label{derham} be the simplicial commutative differential graded (cdg) $\bf{k}$-algebra 
 which sends the $n$-simplex to the algebraic de Rham complex of ${\bf k}[t_0, ..., t_n] / (\sum t_i - 1)$, i.e.,  
 \[
 \Omega[n] := {\bf k}[t_0, ..., t_n, dt_0, ..., dt_n] / (\sum t_i - 1, \sum dt_i ).
 \]
 with the deRham differential $d_{dR}$.
 
  Let $R$ be a ${\bf k}$-algebra and $A$ be a cosimplicial $R$-module.  Let $f$ be a morphism in the simplex category.  The Thom-Sullivan construction associates to $A$ the subcomplex
\begin{align*}
\Thb(A) := \{ (c_i)  \ | \  (A(f) \otimes 1 )(c_n) & = (1 \otimes \Omega(f))(c_m)  \ \forall f \in \rm{Hom}_{\Delta} ([n], [m]) \} \\
&\subseteq  \prod_{n \in \mathbb N} A[n] \otimes_{\bf{k}} \Omega[n].
\end{align*}
As $ \Omega[n]$ is a complex, this amounts to a functor from cosimplicial $R$-modules to complexes of $R$-modules
\begin{align*}
\Thb: \Delta R & \to \mathbf{dg}_{R}.
\end{align*}

Now, let $\cX$ be a Noetherian site over $\Spec \bf k$ and $\cF \in  \Delta_{\cX}$ be a cosimplicial sheaf.  Then we get a presheaf of complexes
\begin{align*}
\Thb: \Delta_{\cX}   & \to \mathbf{dg}_{\cX} \\
U & \mapsto \Thb(\cF(U)).
\end{align*}

\begin{Prop} \label{prop:Thsheaf}
If $\cF$ is a cosimplicial sheaf, then the presheaf $\Thb(\cF)$ is a sheaf.
\end{Prop}
\begin{proof}
Since $\cX$ is Noetherian, by \cite[III. Proposition 3.5]{Milne}, it is enough to check the sheaf axioms on all opens $U$ and all finite covers  $U = \bigcup U_i$.  Then, the sheaf axioms for $\cF$ are equivalent to requiring that
\[
\begin{tikzcd}
\cF(U) \ar[r] & \prod_{i} \cF(U_i)  \ar[r, shift right] \ar[r, shift left] & \prod_{i,j} \cF(U_i \cap U_j)
\end{tikzcd}
\]
is an equalizer diagram for all finite covers.  Since $\Thb$ is exact \cite[Corollary 6.12]{HS}, it preserves finite limits.  Hence
\[
\begin{tikzcd}
\Thb(\cF(U)) \ar[r] & \prod_{i} \Thb(\cF(U_i))  \ar[r, shift right] \ar[r, shift left] & \prod_{i,j} \Thb(\cF(U_i \cap U_j))
\end{tikzcd}
\]
is an equalizer diagram for all finite covers i.e\ $\Thb(\cF)$ is a sheaf.
\end{proof}

Finally, given a cosimplicial complex of sheaves of $\cO_\cX$-modules, we may view it as a complex of cosimplicial sheaves of 
$\cO_\cX$-modules.  As $\Thb$ is a functor, we get a complex of complexes of sheaves of $\cO_\cX$-modules, i.e., 
a double complex of   sheaves of $\cO_\cX$-modules.  Taking the total complex, we obtain a functor
\[
\Thb: \Delta  \mathbf{dg}_{\cO_\cX} \to \mathbf{dg}_{\cO_X}.
\label{Thb}
\]
which we call the {\em Thom-Sullivan functor}.

The following proposition is very well known.  We record its simple proof since it is one of the most crucial aspects of the functor we require.
\begin{Prop} \label{prop: TS cdga}
The Thom-Sullivan functor takes cosimplicial cdg $\cO_X$-algebras to cdg $\cO_X$-algebras, i.e., it restricts to a functor
\[
\Thb: \Delta  \mathbf{cdga}_{\cO_\cX} \to \mathbf{cdga}_{\cO_X}.
\]
\end{Prop}
\begin{proof}
Suppose $\cA$ is a cosimplicial cdg $\cO_X$-algebra.  Since $\Omega_\bullet$ is a simplicial cdg $\bf{k}$-algebra we see that
\[
\prod_{n \in \mathbb N} \cA[n] \otimes_{\bf{k}} \Omega[n]
\]
is a cdg $\cO_X$-algebra.  The algebra structure on this product restricts to the  subcomplex $\Thb(A)$ since
\begin{align*}
& \ \ \ \   (\cA(f) \otimes 1)((a \otimes b) \cdot (c \otimes d))& \\
& = \pm (\cA(f) \otimes 1)(ac \otimes bd ) &  \\ 
  & = \pm \cA(f)(a)\cA(f)(c) \otimes bd  & \text{since $\cA$ is a cosimplicial cdga} \\ 
    & = \pm (\cA(f)(a) \otimes b) \cdot (\cA(f)(c) \otimes d ) &  \\ 
        & = \pm (a \otimes \Omega(f)(b) ) \cdot (c \otimes \Omega(f)(d) )  & \text{for } a \otimes b, c\otimes d \in \Thb(\cA) \\ 
                & = (1 \otimes \Omega(f))((a \otimes b) \cdot (c \otimes d))   & \text{since $\Omega_\bullet$ is a simplicial cdga.}
\end{align*}
\end{proof}
The following proposition is known as the ``de-Rham'' theorem.
\begin{Prop} \label{prop: derham}
Integration of algebraic forms induces a quasi-isomorphism \label{Thint}
\[
\int_{\cA} : \Thb \cA : \to   N(\cA).
\]
where $N$ is the functor which associates to a cosimplicial object its normalized cochain complex. \label{normalized}
\end{Prop}
\begin{proof}
This is very well known.  See for example \cite[Lemma 5.2.8]{HS}.
\end{proof}

The Thom-Sullivan construction behaves well with respect to acyclicity properties.  In this regard, we record the following two propositions.
\begin{Prop} \label{prop: TS flasque}
For a topological space $X$, the Thom-Sullivan functor takes cosimplicial  flasque sheaves of ${\bf k}$-vector spaces to  complexes of flasque sheaves of ${\bf k}$-vector spaces.
\end{Prop}
\begin{proof}
 By \cite[Corollary 6.12]{HS}, $\Thb$ is exact.  In particular, it preserves surjections.
\end{proof}

\begin{Prop}  \label{prop: TS acyclic}
Let $\emph{\'Et}(\tX)$ be the small \'etale site of a (not necessarily separated) Deligne-Mumford stack $\tX$.  Let $\tZ$ be a closed substack of $\tX$.
\label{Et}
The Thom-Sullivan functor takes cosimplicial $\GammaZ$-acyclic sheaves on $\emph{\'Et}(\tX)$ to  complexes of $\GammaZ$-acyclic sheaves  on $\emph{\'Et}(\tX)$.
\end{Prop}
\begin{proof}
Let $\cF$ be a cosimplicial $\GammaZ$-acyclic sheaf of ${\bf k}$-vector spaces, let $X \to \tX$ be an atlas, and let 
$Z := \tZ \ti _{\tX} X$. 
We have 
\begin{footnotesize}
\begin{align*}
\Thb \GammaZ(\tX, \cF) & = \Thb  \text{eq}( \Gamma_Z(X, \cF) \rightrightarrows \Gamma_{Z\ti _{\tX} Z} (X \times_{\tX} X, \cF))&  \text{ since }\cF\text{ is a sheaf on \'Et}(\tX)\\
& =  \text{eq}( \Thb \Gamma_Z(X, \cF) \rightrightarrows  \Thb \Gamma_{ Z\ti _{\tX} Z} (X \times_{\tX} X, \cF))&\text{ by \cite[Corollary 6.12]{HS} } \\
& =  \GammaZ( \tX, \Thb  \cF) &   \text{ by Proposition~\ref{prop:Thsheaf}.}&
\end{align*}
\end{footnotesize}
Hence 
\small
\begin{align*} 
\RR ( \GammaZ(\tX , -) \circ \Thb )\cF  & = \RR (\Thb \circ \GammaZ(\tX, -) )\cF  \\
& =  \Thb \circ \RR \GammaZ \cF&  \text{ by \cite[Corollary 6.12]{HS} } \\
                      & =  \Thb \circ \GammaZ (\tX, \cF) &  \text{ as }\cF\text{ is $\GammaZ$-acyclic} \\
                       & =  \GammaZ( \tX, \Thb  \cF). &
\end{align*}
\normalsize
\end{proof}

\subsection{Godement meets Thom-Sullivan}\label{sec: Godement}

Let $\tX$ be a DM stack and $\tZ$ be a closed substack of $\tX$.
Following \cite[Construction 1.31]{Thomason}, we have a canonical cosimplicial Godement resolution $\God\cF$ of any sheaf $\cF$ on  $\text{\'Et}(\tX)$.  \label{God}

\begin{Lemma} \label{lem: God acyclic}
Let $\cF$ be a sheaf of abelian groups on $\tX$.  The cosimplicial sheaf of abelian groups $\God\cF$ is $\Gamma_{\tZ}$-acyclic.
\end{Lemma}
\begin{proof}
Following Thomason \cite{Thomason}, 
we let $P$ be the set of all points of the topos $\tX_{\text{\'et}}$ associated to $\tX$ in some fixed universe.  Denote by $\tX_{dis}$ the discrete topos indexed by $P$.  There is a natural map of topoi $p: \tX_{dis} \to \tX_{\text{\'et}}$ and $\God[0] = p_*p^*$ by definition.  Hence
\begin{align*}
\RR\Gamma_{\tZ}(\God[0] \cF) & = \RR\Gamma_{\tZ}(p_*p^* \cF) & \\
 & = \RR\Gamma_{\tZ_{dis}}(p^* \cF) & \text{ since }p_*\text{ is exact} \\
  & = \Gamma_{\tZ_{dis}}(p^* \cF) & \text{ since }\Gamma_{\tZ_{dis}}\text{ is exact on a discrete topos.} \\
\end{align*}
Hence $\God[0] \cF$ is acyclic.  As $\God[n] := \overbrace{\God[0] \circ \cdots \circ \God[0]}^{n-\text{times}}$, the result follows by induction on $n$.
\end{proof}

Let $\cF$ be a sheaf of ${\bf k}$-vector spaces and consider it simultaneously as a constant cosimplicial sheaf.  The inclusion of cosimplicial $\cO_\cX$-modules
\[
\iota : \cF \to \God \cF
\]
 induces a map
\begin{equation} \label{eq: Thi}
\Thb(\iota): \cF = \Thb \cF  \to \Thb \God \cF.
\end{equation}
\begin{Lemma}  \label{lem: Th is qi}
The chain map $\Thb(\iota)$ is a quasi-isomorphism.
\end{Lemma}
\begin{proof}
There is a commutative diagram
\[
\begin{tikzcd}
\cF \ar[rr, "\Thb(\iota)"] \ar[dr, swap, "N(\iota)"] & & \ThbGod\cF \ar[dl, "{\int_{\God\cF}}"]  \\
& N(\God\cF) & 
\end{tikzcd}
\]
Since $N(\iota)$ and $\int_{\God\cF}$ are quasi-isomorphisms (see \ref{prop: derham}), it follows that $\Thb(\iota)$ is a quasi-isomorphism.
\end{proof}

Let us now gather much of what we have done into a single useful statement.
\begin{Prop} \label{prop:TSG acyclic}
Let $\tX$ be a (not necessarily separated) DM stack over ${\bf k}$ and $\cF$ be a quasi-coherent sheaf of cdg $\cO_{\tX}$-algebras. 
Let $\tZ$ be a closed substack of $\tX$.
The complex of sheaves $\ThbGod \cF$ is a $\GammaZ$-acyclic sheaf of  cdg $\cO_{\tX}$-algebras which is quasi-isomorphic to $\cF$.
\end{Prop}
\begin{proof}
This follows immediately from Propositions~\ref{prop: TS cdga} and~\ref{prop: TS acyclic} and Lemmas~\ref{lem: God acyclic} and~\ref{lem: Th is qi}.
\end{proof}

Finally, we record the following.
\begin{Lemma} \label{lem: projGod} 
Let $\cF$ be a bounded complex of locally-free $\cO_\cX$-modules of finite rank.   The natural map
\[
(\Thb \God\cO_{\cX} )\otimes_{\cO_{\cX}} \cF \to \Thb \God \cF.
\]
is an isomorphism.
\end{Lemma}
\begin{proof}
This can be checked locally by Proposition~\ref{prop:Thsheaf}.  
Hence, we may assume $\cF = \cO_{\cX}^{\oplus n}$.  Then, it follows from the fact that both $ \God$ and $\Thb$ are obtained as certain a combination of limits and colimits and hence commute with finite sums (which are both finite limits and finite colimits).
\end{proof}

\section{Notational glossary} \label{sec:glossary}
\begin{center}
\begin{longtable}{ l l }
${\bf k}$ & the field of complex numbers, \ \ \pageref{k} \\   
$(V, \kG, \chi, w, \nu)$  & the input GLSM data, \ \  \pageref{sec:input} \\  
$\CR$ & an auxiliary multiplicative group acting on $V$ \ \ \pageref{rcharge} \\
$d_w$ & the pairing of $\CR$ and $\chi$ \ \ \pageref{dw} \\
$P(Y)$ & the vector bundle associated to a $\Gamma$-representation $Y$ \ \ \pageref{pg: P(Y)} \\
 $\cX$ & the target of our theory, \ \  \pageref{cX}  \\ 
$\widehat{G}$ & the group of characters of $G$, \ \ \pageref{characters} \\
$\underline{\cC}$  & the coarse moduli space of a DM stack $\cC$, \ \ \pageref{coarse} \\ 
$\fB_\Gamma$ & a moduli stack of curves with additional bundle data, \ \ \pageref{fB} \\
  $LG(\cX)$ & a moduli space of LG maps to $\cX$, \ \ \pageref{LG}\\  
 $\U$ & a smooth DM stack containing $LG(\cX)$, \ \ \pageref{sec:const U}\\  
 $\pi :\fC \to \fB_{\Gamma}$ & the universal curve \ \ \pageref{pi} \\
$\cV$ & the universal vector bundle on the universal curve, \ \ \pageref{cV} \\
 $[\cA \to \cB]$ & a $\pi_*$-acyclic resolution of $\cV$, \ \ \pageref{cAB} \\
$[A \to B]$ & the pushforward of $[\cA \to \cB]$, \ \ \pageref{AB} \\
$\kb$ & a geometrically realized obstruction theory differential, \ \ \pageref{beta}\\
 $ev_i$ & the evaluation map, \ \ \pageref{sec: ev} \\
 $I{\cX}$  & the inertia stack of $\cX$, \ \ \pageref{inertia} \\    
 $\mathfrak a_\U$  & a class associated to the GLSM data, \ \ \pageref{au} \\      
$\{\tau, \sigma \} $ & the Koszul matrix factorization, \ \ \pageref{KosFact}  \\  
$\Delta\cC$ & cosimplicial objects in a category $\cC$, \ \ \pageref{cosimplicial objects} \\ 
$\mathbf{dg}_{\cO_{\cX}}$ & the category of complexes of $\cO_{\cX}$-modules, \ \ \pageref{sec:TS} \\ 
$\mathbf{cdga}_{\cO_{\cX}}$ & category of cdg ${\cO_{\cX}}$-algebras, \ \ \pageref{sec:TS} \\   
$\Kos$  & the Koszul complex, \ \ \pageref{Kos}\\
$\ka$ & a component of $\mathfrak a_\U$, \ \ \pageref{eqn:an alpha} \\
$\tX$  & a DM stack which is not necessarily separated, \ \ \pageref{tX}\\    
$\mathbb K( \tau, \sigma)$  & the TK factorization associated to $\tau, \sigma$, \ \ \pageref{TK} \\ 
$D(\tX, \tW)$ & the coderived category of $\cO_{\tX}$-module factorizations, \ \ \pageref{derived category} \\
 $\mathbb K(\ka, \kb)$ & the virtual factorization, \ \ \pageref{eq:fund factor}   \\   
 $\Mgr$ & the moduli stack of $n$-pointed genus $g$ stable  curves \ \ \pageref{Mgr} \\
$(\Omega_{I\cX}, dw)$ & the dual of the Koszul complex of $dw$ on $I\cX$, \ \ \pageref{cotangent} \\
$f^*\tau$ &  the pullback of $\tau$           \ \      \pageref{notation: pullback tau} \\
 $\Z$ &  the moduli space of LG maps to the critical locus, \ \ \pageref{Z}\\   
$\inv$  & an automorphism of $I\cX$, \ \ \pageref{inv} \\   
$\tdch$ & the Todd-Chern class of a TK factorization \ \ \pageref{def:tdch}\\    
$[\U]_W^{\vir}$ & the virtual fundamental class \ \ \pageref{def:virclass}\\    
$fgt$ & the forgetful map from $\U$ to $\Mgr$, \ \ \pageref{fgt} \\
$c_i$ & the weights of the $\CC^\times_R$-action, \ \ \pageref{pg: ci} \\
$\int_{\tX}$ & the trace map, \ \ \pageref{eq:def trace} \\
$\mathcal{MV}$ & the Mayer-Vietoris functor, \ \ \pageref{eq: MV} \\
$MV_c$ & compactly supported Mayer-Vietoris complex, \ \ \pageref{eq: MV}  \\
$\gamma_\tZ$ & the subsheaf of sections with supports in $\tZ$, \ \ \pageref{sheafylocalcohom} \\
$H^*_\tZ(\tX, \cF^\bullet)$ & hypercohomology of the complex $\RR\gamma_\tZ \cF^\bullet$ \ \ \pageref{sheafylocalcohom} \\
$\vC$ & the \v Cech functor for an \'etale cover, \ \ \pageref{rmk: Cech functor}\\
$(\Omega_\bullet, d_{dR})$ & the simplicial algebraic deRham complex, \ \ \pageref{derham} \\
$\Thb$  & the Thom-Sullivan functor, \ \ \pageref{Thb}  \\   
$\int_{\cA}$ & integration for a cosimplicial object $\cA$, \ \ \pageref{Thint} \\
$N$ & the normalized cochain complex, \ \ \pageref{normalized} \\ 
$\text{\'Et}(\tX)$ & the \'etale site associated to $\tX$, \ \ \pageref{Et} \\      
$\God$ & the Godement functor, \ \ \pageref{God}   
\end{longtable}
\end{center}

\end{document}